\documentclass[leqno,a4paper]{amsart}
\usepackage{amssymb,amsfonts,amsthm,amsmath,latexsym,amscd} 
\usepackage{mathrsfs,euscript,anyfontsize}
\usepackage[applemac]{inputenc}   
\usepackage{datetime}
\usepackage{xcolor}

\makeatletter
\renewcommand{\pod}[1]{\allowbreak\mathchoice
  {\if@display \mkern 18mu\else \mkern 8mu\fi (#1)}
  {\if@display \mkern 18mu\else \mkern 8mu\fi (#1)}
  {\mkern4mu(#1)}
  {\mkern4mu(#1)}
}
\usepackage{enumerate}[1.)]
\usepackage{pb-diagram}  
\renewcommand{\eqref}[1]{(\ref{#1})}   
\theoremstyle{plain}
\newtheorem{theorem}{Theorem}[section]
\newtheorem{theorema}{Theorem}

\newtheorem{corollary}[theorem]{Corollary}

\newtheorem{lemma}[theorem]{Lemma}

\newtheorem{proposition}[theorem]{Proposition}
\newtheorem{algorithm}{Algorithm}

\newtheorem{definition}[theorem]{Definition}

\theoremstyle{remark}
\newtheorem{remark}[theorem]{Remark}

\newcommand{\CC}{{\mathbb C}}
\newcommand {\R}{{\mathbb{R}}}
\newcommand {\Z}{{\mathbb{Z}}}
\newcommand {\N}{{\mathbb{N}}}

\newcommand {\PP}{\mathbb{P}}
\newcommand {\A}{\mathcal{A}}

\newcommand{\LL}{{\mathscr{L}}}
\newcommand{\C}{{\EuScript C}}
\newcommand{\scI}{{\EuScript I}}
\newcommand{\EE}{{\EuScript E}}
\newcommand{\FF}{{\EuScript F}}
\newcommand{\HH}{{\EuScript H}}
\newcommand{\MM}{{\EuScript M}}
\newcommand{\TT}{{\EuScript T}}
\newcommand{\VV}{{\EuScript V}}
\newcommand{\WW}{{\EuScript W}}
\renewcommand{\P}{{\EuScript P}}

\newcommand {\idb}{{\mathfrak{b}}}
\newcommand {\idc}{{\mathfrak{c}}}

\newcommand{\idm}{{\mathfrak m}}

\newcommand{\idr}{{\mathfrak r}}
\newcommand{\idv}{{\mathfrak v}}
\newcommand{\gL}{{\mathfrak L}}

\newcommand{\gE}{{\mathfrak E}}
\newcommand{\II}{{\mathfrak I}}
\newcommand{\gR}{{\mathfrak R}}
\newcommand{\gS}{{\mathfrak S}}

\newcommand{\gU}{{\mathfrak U}}
\newcommand{\gV}{{\mathfrak V}}
\newcommand{\gW}{{\mathfrak W}}

\newcommand{\bz}{{\boldsymbol z}}
\numberwithin{equation}{section}
\renewcommand{\leq}{\leqslant}
\renewcommand{\geq}{\geqslant}
\newcommand{\mmod}[1]{\,({\rm mod\,}#1)}
\def\md#1#2{\equiv#1\,({\rm mod\,}#2)}
\def\SW{{\rm SW}}
\def\N{{\mathbb N}}
\def\NN{{\EuScript N}}
\def\ES{{\EuScript S}}
\def\1{{\mathbf 1}}
\def\pnu{p^\nu}
\def\e{{\rm e}}
\def\d{\,{\rm d}}
\def\dd{{\rm d}}
\def\bal{{\boldsymbol\alpha}}
\def\bbe{{\boldsymbol\beta}}
\def\bsk{{\boldsymbol k}}
\def\bsl{{\boldsymbol \ell}}
\def\bst{{\boldsymbol t}}
\def\bsz{{\boldsymbol z}}
\def\bsxi{{\boldsymbol\xi}}
\newcommand\no[1]{\left\|#1\right\|}
\newcommand\fl[1]{\left\lfloor#1\right\rfloor}
 \def\sumast{\mathop{{\sum}^*}}
 \def\sumdag{\mathop{{\sum}^\dagger}}
 \def\eg{e.g.}
\newcommand\vbs[1]{\left|#1\right|}
\newcommand{\ov}{\overline}
\begin{document}

\title[Multiplicative functions in large arithmetic progressions]{Multiplicative functions in large arithmetic progressions and applications}
\date{\today, \currenttime}
 
\author{\'Etienne Fouvry}
\address{Universit\' e Paris--Saclay, CNRS, Laboratoire  de  Math\' ematiques d'Orsay,    \goodbreak    91405   Orsay,   France}
\email{Etienne.Fouvry@universite-paris-saclay.fr}
\author{G\'erald Tenenbaum}
\address{Institut \' Elie Cartan, 
Universit\' e de Lorraine,
B.P. 70239, \goodbreak
F--54506 Vand\oe{}uvre-l\`es-Nancy Cedex,  
France}
\email{Gerald.Tenenbaum@univ-lorraine.fr}



\begin{abstract}  
We establish new Bombieri-Vinogradov type estimates for a wide class of multiplicative arithmetic functions and derive several applications, including: a new proof of a recent estimate by Drappeau and Topacogullari for arithmetical correlations; a theorem of Erd{\H o}s-Wintner type with support equal to the level set of an additive function at shifted argument; and a law of iterated logarithm for the distribution of prime factors of integers weighted by $\tau(n-1)$ where $\tau$ denotes the divisor function. 
  \end{abstract}
  
\maketitle
\vskip-9.5mm
{\vskip0mm
\fontsize{9}{9.4}\selectfont{\tableofcontents}\goodbreak}
\newpage

\renewcommand{\theenumi}{\roman{enumi}}

 \section{Introduction} The idea of this paper came up to our minds on studying the work of Drappeau and Topacogullari \cite{DrTo}, in which 
 the authors  investigate sums of the form
 $$
\mathfrak T (x;f):= \sum_{1\leqslant n\leqslant x} f(n) \tau (n-1)\qquad (x\to\infty),
 $$
where  $f$ is a multiplicative function, periodic over the primes (see Definition \ref{defM} and Remark \ref{fperp}~below) and $\tau$ is the standard divisor function. In this work, an asymptotic formula for the sum $\mathfrak T (x;f)$ is derived with error term  $\ll x/(\log x)^{N}$, for arbitrary $N\geqslant 1$. Our approach consists in shifting this question to the problem of the  level of distribution of such multiplicative functions $f$. \par 
As a consequence,  we obtain an alternative proof of the results of \cite{DrTo} briefly described in Section \ref{sectapps} and, in a more innovative  way, we obtain  new  information on the joint distribution of $(f(n), g(n-1))$, for certain additive functions $f$, $g$. 
  \par 
  \smallskip
  We first describe the general framework for various  types of levels of distribution.
 \par 
  To study the  statistical behaviour of the arithmetical function $f$  over the arithmetic progression $a\,(\bmod \,q)$, with $(a,q)=1$, it is natural to introduce the error term
 \begin{equation}\label{defDeltaf}
 \Delta_f (x;q,a):= \sum_{\substack{n\leqslant x\\ n\md aq}} f(n)-\frac{1}{\varphi (q)} \sum_{\substack{n\leqslant x \\ (n,q)=1}} f(n)\qquad (x\geqslant 1).
\end{equation}
Whenever  $f$ is suspected to be well distributed among arithmetic progressions, the challenge is to prove that, for any $A >0$, there exists a constant $c(A)$ such that, for any $x\geqslant 1$, all $q$ in some specific range (depending on $x$ and  as large as possible), and any integer $a$ coprime to $q$, 
we have 
 \begin{equation}\label{SW0}
\bigl\vert \,  \Delta_f (x;q,a)\, \bigr\vert  \leqslant    \frac{c(A)\sqrt{x}}{\varphi (q)\LL^A} \bigg( \sum_{n\leqslant x} \vert f(n)\vert^2\bigg)^{1/2},
 \end{equation}
where  $\LL
:=\log 3x.$
 A  general presentation of these topics is displayed in \cite[p.205--210]{BFI1}. 
 \par 
 
A more tractable form of the question is obtained    by   studying    the {\it average distribution of $f$}.  For instance 
\par 
(a) {\it Find    (large) values of $Q=Q(x, A)$,   such that the bound
\begin{equation}\label{BV0}
\sum_{q\leqslant  Q} \ \max_{y\leqslant x} \max_{  (a,q)=1}
\vert   \Delta_f (y;q,a)\vert \ll_A   \frac{\sqrt{x}}{\LL^A} \bigg( \sum_{n\leqslant x} \vert f(n)\vert^2\bigg)^{1/2}\qquad (x\geqslant 1),
\end{equation}
holds for any    $A>0$. \par 
\rm (b) \it Given a fixed integer $a\neq0$, determine a large range of validity for the weaker requirement \begin{equation}\label{BV}
 \sum_{\substack {q\leqslant  Q \\ (q,a)=1}} 
\vert   \Delta_f (x;q,a)\vert \ll_A  \frac{\sqrt{x}}{\LL^A} \bigg( \sum_{n\leqslant x} \vert f(n)\vert^2\bigg)^{1/2}\qquad (x\geqslant 1).
\end{equation} }
 \par 
 Some classical results assert that, for many arithmetical functions such as the indicator $\1_\PP$ of the set of primes,  standard   multiplicative functions and others, the bound \eqref{BV0} does hold with  $$Q=\sqrt{x}/ \LL^{B(A)},$$
 the key point being that $f$ should present some valuable combinatorial structure  (in order to apply the large sieve inequality) and 
 should already satisfy \eqref{SW0} when $q$ does not exceed
 some power of~$\LL$---a Siegel--Walfisz type property. \par 
 The threshold $Q= \sqrt{x}$ has been overpassed only in very few examples for $f$, even in the case of the simpler inequality
 \eqref{BV}. These difficult results require sophisticated tools.  
  
An apparently much easier problem is:
\par 
(c) {\it Find (large) values of $Q= Q(x,A)$ such that the bound
\begin{equation}\label{but}
 \sum_{\substack{q\leqslant  Q \\ (q,a)=1}} 
   \Delta_f (x;q,a)  \ll_{a,A}\frac{\sqrt{x}}{\LL^A} \Bigl( \sum_{n\leqslant x} \vert f(n)\vert^2\Bigr)^{1/2}\qquad (x\geqslant 1)
\end{equation} 
holds for any $A>0$ and any integer $a\ne 0$.}\par 
Compared to \eqref{BV}, this question seems  simpler because  the sign oscillations of 
$\Delta_f (x;q,a)$ as $q$ varies may be exploited. Furthermore, Dirichlet's hyperbola technique opens the way to reach much larger value of $Q$. Indeed, writing the congruence condition \mbox{$n\md aq$} appearing in  \eqref{defDeltaf} as
 $$
 n=a+qr,\quad n\leqslant x,
 $$ 
we may replace the smooth summation over $q$ in \eqref{but} by a smooth  summation over~$r$. This method is highly efficient when $q$ runs over all integer values from an interval included in $[\sqrt{x}, x]$. Indeed,  the variable $r$ is then also smooth and, furthermore,  bounded above by
$\sqrt{x}$. As a consequence, for adequate functions $f$,   the proof of \eqref{but}  for some  $Q(x,A) >\sqrt{x}$ may be reduced to the case $Q\leqslant \sqrt{x}$. This will be illustrated in our approach---see \S \ref{Dirichlet}.
  Note that  \eqref{BV} is not yet known to hold for  $Q(x,A)= \sqrt{x}$   when  $f=\1_\PP$. However, we have the following theorem.
\begin{theorema}  \label{pointdedepart} 
{\rm (\cite[th. 9]{BFI1}, \cite{BFI4},  \cite[cor. 1]{FouCrelle})}. Let $f:=\1_\PP$.
For every $A$ and suitable $B=B(A)$, $ C=C (A)$,  the inequality
$$
\Bigg\vert
\sum_{\substack{q\leqslant Q\\ (q,a)=1}} \Delta_f(x;q,a)
\Bigg\vert \leqslant \frac{C x}{\LL^{A}}\qquad (x\geqslant 1)
$$
holds for all $Q \leqslant x/\LL^{B}$ and  any integer $a$ such that $1 \leqslant \vert a \vert \leqslant \LL^A$. 
\end{theorema}
This theorem is sufficiently strong to enable further progress in the well-known Titchmarsh divisor problem: if $\Lambda$ denotes the von Mangoldt function, an asymptotic expansion for the sum
$$
\mathfrak T (x;\Lambda) = \sum_{2\leqslant n\leqslant x  } \Lambda (n) \tau (n-1),
$$
is now available with  error term $\ll x/\LL^{A}$ for arbitrary, fixed $A$: see \cite[cor. 1]{BFI1} and \cite[cor. 2]{FouCrelle}.

Actually  Bombieri, Friedlander and Iwaniec proved a stronger form of Theorem~\ref{pointdedepart} which may be interpreted as an answer to a compromise  between  questions (b) and~(c).
\begin{theorema}\label{pointdedepart1}
\rm (\cite[th. 9]{BFI1}). \it  Let $f:=\1_\PP$.
Then, for all $A>0$,  $\varepsilon >0$, and suitable $B=B(\varepsilon, A)$, \mbox{$ C=C (\varepsilon, A)$}, the inequality
$$
\sum_{\substack{r\leqslant R \\ (r,a)=1}}\Bigg\vert
\sum_{\substack{q\leqslant Q\\ (q,a)=1}} \Delta_f(x;qr,a)
\Bigg\vert \leqslant \frac{Cx}{\LL^A}\qquad (x\geqslant 1)
$$
holds for every $Q$ and $R$ satisfying $1 \leqslant R \leqslant x^{1/10 -\varepsilon}$,  $QR \leqslant x/\LL^{B}$, and every integer $a$ such that $1 \leqslant \vert a \vert \leqslant \LL^A$. 
\end{theorema}
The first aim of this paper is to establish an analogue of this theorem in the context of multiplicative functions $f$ that are essentially periodic on the set of  primes. As mentioned above, we subsequently apply 
this result to various problems, related to joint distribution of pairs of additive functions,  one of them being sampled at a shifted argument.
 \par \smallskip
It is now time to state our central result, providing sufficient conditions to ensure that the statement of Theorem \ref{pointdedepart1} remains  true for multiplicative functions $f$ of the above mentioned type. Since we aim at  a large uniformity over $f$, 
our hypotheses  require specific notations used both in the proofs and in the applications.  
 \subsection{Conventions and notations}\label{conventions} The following notation will be used throughout this paper. 
\par 
$\bullet$ $\gamma$ is Euler's constant.
\par 
 $\bullet $ $\mathbb N$ is the set of non-negative integers, $\mathbb N^*:=\mathbb N\smallsetminus\{0\}$.
 \par 
 $\bullet$ For arbitrary sets $X$, $Y$,  the set of mappings $X\rightarrow Y$ is denoted as $ Y^X$.  
 \par 
  $\bullet$ The lower case letter $x$   denotes a real number $\geqslant 1$ and  $\LL$   is implicitly defined by  $\LL := \log 3x$. In some instances it will be implicitly assumed that $x$ is sufficiently large.
 \par 
 $\bullet$ Throughout this work, we let $\log_k$ denote the $k$-th iterated logarithm.\par 
  $\bullet$  The letter $p$ is reserved to denote a prime number. 
  \par 
  $\bullet$ $P^+(n)$ (resp. $P^-(n)$) denotes the largest (resp. the smallest) prime factor of a positive integer $n$, with the convention that $P^+(1)=1$, $P^-(1)=\infty$.
  \par 
  $\bullet$ The notation (especially in a subscript) $n\sim N$ means that the integer variable $n$ satisfies the inequality $N<n\leqslant 2N$, while we use  $n \simeq N$ to indicate that $n$ belongs to some (usually unspecified) interval included 
 in $]N, 2N]$. 
 
  $\bullet$ The letter $\idc$ denotes a constant depending on various parameters such as $\varepsilon>0$, $K>0$,  etc., and whose value may change at each occurence.

  $\bullet$ $C_0$ denotes an absolute constant, whose effectively computable value may change at each occurrence. It will mainly appear in upper bounds containing the factor $D^{C_0}$.

  $\bullet$ Given a complex sequence $\bal= (\alpha_m)_{m\geqslant 1} $ and a real number $M>1$, we define the $\ell^2$--norm of $(\alpha_m)_{m\sim M}$ by
 $$
 \Vert \bal\Vert_M ^2:= \sum_{m\sim M} \vert \alpha_m\vert^2.
 $$
\par 
  $\bullet$  For integers $\nu\geqslant 0$, $n\geqslant 1$, and primes $p$, we write $p^\nu \Vert n$ to mean that $p^\nu \mid n$ but 
 $p^{\nu +1} \nmid n$. The notation $d\mid n^\infty$ means that $p\mid n$ whenever $p\mid d$.  \par 
 $\bullet$ A {\it strongly multiplicative} (resp. {\it strongly additive}) arithmetical function is a multiplicative (resp.  an additive)  function such that $f(\pnu)=f(p)$ for all $\nu\geqslant 1$.
 \par 
  $\bullet$ $\omega (n)$ is the number of distinct prime factors of the integer $n\geqslant 1$. More generally, given integers $D\geqslant 1$ and $t\in\Z$, we put 
 \begin{equation}
 \label{omDtn}
\omega_{D,t} (n) :  =\sum_{\substack{p\mid n\\ p\md  tD}} 1.
\end{equation}
We also define
\begin{equation}
\label{omnt}
\omega(n,t):=\sum_{\substack{p\,\mid\,n\\ p\leqslant t}}1\qquad (n\geqslant 1, \,t\geqslant 3).
\end{equation} 
$\bullet $ Given a subset $\A$ of $\N^*$, we let $\1_\A:\N^*\to\{0,1\}$ designate the indicator function of $\A$. We simply write $\1$ for $\1_{\N^*}$. For $Y_0 \geqslant 2$, we put \begin{equation}
\label{gY0}
\mathfrak Y_0:=\1_{\{n\geqslant 1:P^-(n)\geqslant Y_0\}},
\end{equation}
 so that $\mathfrak Y_0= \mathbf 1$ whenever $Y_0 \leqslant 2$.
\par 
 $\bullet$ For $k \geqslant 1$, we write $w_k:=\1_{\{n\geqslant 1:\,\omega (n) =k\}}$. In particular, $w_1$
 is the characteristic function of the set of prime powers. The summatory function of $w_k$ is denoted by $\pi_k(x)$.
\par 
  $\bullet$ The M\"obius function is denoted by $\mu$ is, the von Mangoldt function by $\Lambda$, and, given $\alpha>0$, we put
\begin{equation}
\label{balpha}
b_\alpha (n) := \prod_{p \vert n} \Bigl( 1 + \frac{1}{p^\alpha}\Bigr)\qquad (n\geqslant 1).
\end{equation}
\par 
  $\bullet $ For $n\in\N^*$ and $z\in\CC$, we define the Piltz divisor function $n\mapsto\tau_z (n)$ by the   Dirichlet series expansion of $\zeta (s)^z = \sum_{n\geqslant 1} \tau_z (n)/n^{s}$, converging in the half-plane $\Re s >1$. We often simply note $\tau=\tau_2$. 
\par 
  $\bullet$ Given coprime integers $q\geqslant 1$ and $a$,    we define  $g_q (n;a)$ for $n\geqslant 1$ by the formula
 $$
 g_q(n;a):=
 \begin{cases}
 1-1/\varphi (q) &\text{ if } n\md aq,\\
-1/\varphi(q) &\text{ if } n\not\md aq \text{ and } (n,q)=1, \\
0 & \text{ if } (n,q) >1.
 \end{cases}
 $$
This notation will be used to shorten some formulae. For instance, we have 
 $$
 \Delta_f (x;q,a) = \sum_{n\leqslant x} f(n) g_q (n;a)\qquad (x\geqslant 1).
 $$
 
  $\bullet$ Given an an arithmetical function  $f$ and  two positive integers, $b$, $c$, we denote by $f_{b,c}$ the arithmetical  modification of the function $f$ defined by
 \begin{equation}\label{deffbb}
 f_{b,c} (n):=
 \begin{cases}
 f(bn) &\text{ if } (n,c)=1,
 \\
 0 & \text{ otherwise}.
 \end{cases}
 \end{equation}

\subsection{Definitions}
 Our central definition is the following (see \cite[\S1]{DrTo}).
 \begin{definition}
 \label{defM} {\rm [$\FF (D,K)$]} Let $D\in\N^*$ and $K>0$.  We denote by $\FF (D,K)$ the set of those multiplicative functions $f$ verifying the following properties:
 \begin{enumerate}  
\item  There exists a sequence of real numbers
$$
2=\Upsilon_1< \Upsilon_2 <\Upsilon_3 < \cdots
$$
such that
$$
\frac {\Upsilon_{n+1}}{\Upsilon_n} \geqslant 1 + \frac {1}{(\log 2\Upsilon_n)^K}\qquad (n \geqslant 1),
$$
\item 
For any pair of primes $p,\,p'$ with $\Upsilon_n <p, p' \leqslant \Upsilon_{n+1} $ and $p \md{p'}D$,  we have
\begin{equation}\label{important}
f(p) =f (p'),
\end{equation}
\item We have
\begin{equation}\label{f<tau}
\vert f (n) \vert \leqslant \tau_K (n)\qquad (n\geqslant 1).
\end{equation}
 \end{enumerate}
 \end{definition}
\par 
\begin{remark}
\label{fperp}
Some comments are in order regarding hypothesis (ii) above. A similar regularity condition actually appears  in many works dealing with Bombieri--Vinogradov type theorems for multiplicative functions, e.g. in
 Wolke's \cite[Satz~1, cond.~1.1.2]{Wolke73}, where the  $ f(p)$
are assumed to be close, on average,  to a fixed number~$\tau$, or  in the definition of the set  $\FF_D (A)$
given in \cite[p.2384]{DrTo}.
As mentioned earlier, an hypothesis of this type is needed to get  a  Siegel--Walfisz  property for the  restriction $f|_\PP$. 
Now this Siegel--Walfisz hypothesis
 alone is not sufficient to reach high levels of distribution for the function $n\mapsto f(n)$:  see   the instructive example provided in \cite[Prop 1.3]{GS19}.
For our present purposes, hypothesis (ii) in Definition~\ref{defM} turns out to be also crucial in  Subsections \ref{<2/7} and \ref{>2/7}.  Indeed, various combinatorial transformations
 and reductions lead to handle the exponent of distribution in arithmetic progressions  of the sequence $f(p_1) f(p_2) f(p_3)$
where the primes  $p_j$ satisfy   $p_1p_2p_3\leqslant  x$ and  $p_j >x^{2/7}$.  Since  $f|_\PP$ is assumed to be essentially   constant, this amounts to handle the level of distribution of the product  $\Lambda (n_1) \Lambda (n_2) \Lambda (n_3)$,
where the integers  $n_j$ satisfy  $n_1n_2n_3 \leq x$ and  $n_j > x^{2/7}$.
Further combinatorial transformations enable to reduce the problem to studying the exponent of distribution 
of the products $n_1n_2n_3\leq x$ with  $n_j >x^{2/7}$. At that point, we may conclude by appealing to the deep result concerning the distribution of the function  $\tau_3$  in arithmetic progressions---see Lemma \ref{tau3}.
\par 
We note incidentally that the insertion of the sequence $\{\Upsilon_n\}_{n=1}^{\infty}$ will be of crucial importance in some applications, for instance Theorem \ref{thLIL} below.
\end{remark}
\par \smallskip
For $f \in \FF (D,K)$  the  restriction $f|_\PP$ of $f$ to the set $\PP$ of all primes is uniformly bounded since 
 \begin{equation}\label{f(p)leq}
\vert \, f(p)\, \vert \leqslant  K\qquad (p\in\PP).
\end{equation}
\par 
Conditions (i)  and (ii) imply that the function $p\mapsto f(p)$  is equidistributed among the arithmetic progressions 
$\{p\in\PP:p\leqslant x,\,p\md aq\}$, when $(aD,q)=1$ and $q \leqslant \LL^A$ for any fixed $A$. This is a Siegel--Walfisz type assumption, as presented in
Definition  \ref{SW} below. The periodicity of $p\mapsto f(p)$  is crucial in \S \ref{>2/7} where we need results concerning the distribution levels   of  $\tau_2$ and $\tau_3$.
  
Typical examples of elements of   $\FF (1,K)$  for some $K$  are provided by the functions
$ n \mapsto z^{\omega (n)}$ and $\tau_z$, with  $\vert z \vert \leqslant K$. \par 
More elaborate is the case of  the characteristic function $\1_{\mathfrak Q}$ of the set of those integers representable as the
sum of two squares. This function  belongs to $\FF (4, 1)$ but not to $\FF (1,K)$ for any $K$.  This follows from Fermat's theorem for primes in $\mathfrak Q$. However  $\1_{\mathfrak Q}$ 
  is not of Siegel--Walfisz type $\SW(1,K)$ in the sense of Definition \ref{SW} below. Indeed, for instance, given any  integer $q$  divisible by $4$ but not by $3$, we have $\1_{\mathfrak Q} (n) =0$ whenever  $n\md3q$.  \par 
  To circumvent this difficulty, we introduce
a new definition extending that of $\Delta_f (x;q,a)$. Given an arithmetic function  $f$, a real number $x\geqslant 1$ and integers   $q\geqslant 1$ and $D\geqslant 1$, we put
\begin{equation}\label{factorizeq}
q_D:= (q,D^\infty)=\prod_{\substack{p^\nu\|q\\ p\,\mid\,D}}p^\nu, \qquad  q'_D:= \frac q{q_D}=\prod_{\substack{p^\nu\|q\\ p\,\nmid\,D}}p^\nu,
\end{equation}
and, for integer $a\not= 0$, 
\begin{equation}
\begin{aligned}
\Delta_f(x;q,D,a) &:= \sum_{\substack{n\leqslant x\\ n\md aq}} f(n) -\frac{1}{\varphi (q'_D)} \sum_{\substack{n\leqslant x\\ n\md a{
q_D}\\ (n,q'_D)=1} }f(n)\\
 &= \sum_{\substack{n\leqslant x\\ n\md a{q_D}\\ n\md a{q'_D}}} f(n) -\frac{1}{\varphi (q'_D)} \sum_{\substack{n\leqslant x\\ n\md a{q_D} \\ (n,q'_D)=1}}f(n),
\end{aligned}
\label{350}
\end{equation}
since $(q_D,q'_D)=1$. \goodbreak\par 
When $(q,D)=1$, we have 
\begin{equation}\label{Delta=Delta}
 \Delta_f(x;q,D,a) = \Delta_f (x;q,a).
 \end{equation}\par 
When $(D,a)=1$ formula \eqref{350} can be expressed in terms of Dirichlet characters~$\chi$ modulo $q_D$, viz. 
\begin{equation}\label{Dirichletcharacter}
 \Delta_f(x;q,D,a) = \frac1{\varphi (q_D)}\sum_{\chi \mmod{q_D}}\overline{\chi (a)}\, \Delta_{f\chi} (x; q'_D, a),
 \end{equation}
  since the function 
 $$n\mapsto  \frac1{\varphi (q_D)} \sum_{\chi \mmod {q_D}}\overline {\chi (a)}\chi (n)$$ is the characteristic function of the arithmetic progression $\{n:n\md a{q_D}\}.$
 
 \par \smallskip
 
 Another important definition is the following. We recall the definition \eqref{omDtn} for the function $\omega_{D,t}$.
\begin{definition}\label{defPhi} \rm{ [$\chi_{D, \TT , \Phi}$}] Let $D\in\N^*$, $\TT \subset \bigl( \mathbb Z/D\mathbb Z\bigr)^*$ and $\Phi\in\N^\TT$ be a function defined on $\TT $ with non-negative integral values. We denote by
$$\chi_{D, \TT , \Phi}$$
the characteristic function of the set of those integers $n\geqslant 1$, such that
$$
 (\forall t\in \TT)\qquad \omega_{D,t}(n) = \Phi (t).
$$
\end{definition}  
If $\TT =\varnothing$, then  $\chi_{D, \TT , \Phi}=\mathbf 1.$ If $D=1$, $\TT  =\{0\}$ and $\Phi (0)=k$
is a fixed non-negative integer, then  $\chi_{D, \TT , \Phi}$ is the characteristic function $w_k$ of those integers with $k$ distinct prime factors, as
introduced in \S \ref{conventions}. More generally,  $\chi_{D, \TT , \Phi}$
 detects those integers $n$ with prescribed number of 
distinct prime divisors in some fixed reduced classes modulo $D$. The function $\chi_{D, \TT , \Phi}$ is linked to elements of $\FF (D,K)$ via the identity given  in \eqref{integralequation} {\it infra}.
\par \smallskip
We next bring up  the {\it Siegel--Walfisz condition} to express equidistribution of sequences among reduced arithmetic progressions, with modulus coprime to $D$.
\begin{definition} \label{SW} {\rm [Siegel--Walfisz condition $\SW(D,K)$]} Let $D\in\N^*$ and $K>0$ be  given, and let $\bbe = ( \beta_n)\in\CC^{\N^*}$  be a complex sequence. We say that $\bbe$ satisfies 
   the Siegel--Walfisz condition $ \SW(D,K)$, and write $\bbe\in \SW(D,K)$,  if for all $A>0$,  we have
\begin{equation*}
\sum_{\substack{n\sim N, \, (n,d)=1\\ \ n\md\ell k}  } \beta_n -\frac{1}{\varphi (k)} \sum_{\substack{n\sim N\\ (n,dk) =1}} \beta_n\ll \frac{\no\bbe_N  \tau (d)^K \sqrt{N}}{ (\log 2N)^A}, \leqno{(\SW(D,K))}
\end{equation*}
uniformly for $N \geqslant 1$,  $d\geqslant 1$, $k\geqslant 1$,  $(\ell D, k) =1$.
\end{definition}

\subsection{The central results}  First of all, we establish an analogue of  Theorem \ref{pointdedepart1}  concerning multiplicative functions $f$ in the class $\FF (D,K)$.

\begin{theorem} 
\label{central}   
The following statement holds for suitable, absolute $C_0$. Let $A>0$, $\varepsilon>0$, $K>0$.
There exist $B= B( A,\varepsilon,K) $ and $C= C(A,\varepsilon, K)$  such that, uniformly for
\begin{align*}
&D\geqslant 1, \quad  f\in \FF (D,K), \quad x\geqslant 1,\quad R\leqslant x^{1/105-\varepsilon},\\ &QR\leqslant x/\LL^{B}, \quad (a,D)=1,\quad 1\leqslant \vert a \vert  \leqslant \LL^A, \\
\end{align*}
we have
\begin{equation}\label{244}
\sum_{\substack{r\leqslant R\\ (r,a )=1}}\Bigg\vert \sum_{\substack{q\leqslant Q \\ (q,a  )=1}} \Delta_f (x;qr,D,a ) \Bigg\vert \leqslant \frac{ C  D^{C_0}  x}{\LL^A}\cdot
\end{equation}
\end{theorem}
We  deduce the following two corollaries.
\begin{corollary} \label{easy1} The following statement holds for suitable, absolute $C_0$. Let $A>0$, $\varepsilon>0$, $K>0$.
There exist $B= B( A,\varepsilon, K) $ and $C= C(A,\varepsilon, K)$   such that, uniformly for
\begin{align*}
&D\geqslant 1,\,  f\in \FF (D,K),  \, x\geqslant 1,\\ 
 & R\leqslant x^{1/105-\varepsilon},\, QR\leqslant x/\LL^{B},\  \vert \xi_r\vert \leqslant \tau_K (r)\ (1\leqslant r\leqslant R), \\ &(a,D)=1,\, 1 \leqslant \vert a \vert  \leqslant \LL^A,
\end{align*}
we have\begin{equation}\label{244ter}
\Bigg\vert  \sum_{\substack{r\leqslant R\\ (r,a )=1}}\xi_r\sum_{\substack{q\leqslant Q \\ (q,a )=1}} \Delta_f (x;qr,D,a)  \, \Bigg\vert \leqslant  \frac{C\,D^{C_0}\,x}{\LL ^A}\cdot
\end{equation}
\end{corollary}
The second corollary deals with the function $\chi_{D, \TT , \Phi}$ from Definition \ref{defPhi}. 
\begin{corollary}\label{easy2}  The following statement holds for suitable, absolute $C_0$.  Let $A>0$, $\varepsilon>0$, $K>0$.
There exist $B= B( A,\varepsilon, K) $ and $C= C(A,\varepsilon, K)$   such that, uniformly for
\begin{equation*}
D\geqslant 1,\,  \TT  \subset (\mathbb Z/D \mathbb Z)^*, \, \Phi \in \N^\TT,  \,  (a,D)=1
 \end{equation*}
 and
 \begin{equation}
 \label{COND10}
\begin{aligned}
  &x\geqslant 1, \  R\leqslant x^{1/105-\varepsilon}, \ QR\leqslant x/\LL^{B},\\  &\vert \xi_r\vert \leqslant \tau_K (r)\ (1\leqslant r\leqslant R),\ 1 \leqslant \vert a \vert  \leqslant \LL^A,
\end{aligned}
\end{equation}
we have
\begin{equation} \label{244ter1}
\Bigg\vert  \sum_{\substack{r\leqslant R\\ (r,a)=1}}\xi_r\sum_{\substack{q\leqslant Q \\ (q,a )=1}} \Delta_{\chi_{D, \TT , \Phi}}(x;qr, D,a  ) \Bigg\vert \leqslant  \frac{CD^{C_0} x}{\LL ^A}\cdot
\end{equation}
In particular,   for all  $A>0,\,\varepsilon >0$ there exist $B=B(A,\varepsilon)$ and $C= C( A,\varepsilon)$, such that, under conditions \eqref{COND10} and uniformly for $k\geqslant 1$, we have
\begin{equation}
\label{weakerform} 
\Bigg\vert  \sum_{\substack{r\leqslant R\\ (r,a )=1}}\xi_r\sum_{\substack{q\leqslant Q \\ (q,a  )=1}}  \Delta_{w_k}(x;qr,a  )   \, \Bigg\vert \leqslant  \frac{Cx}{\LL ^A}\cdot
\end{equation}
\end{corollary} 
The structure of the upper bounds in \eqref{244}, \eqref{244ter} and \eqref{244ter1} allows selecting  $D=D(x)$ tending to infinity with $x$, but not faster than a bounded power of $\log x.$
For $k=1$, the upper bound in \eqref{weakerform} is a weaker form of Theorem \ref{pointdedepart1}.
Note that this inequality  is actually useless whenever $
k/\log_2 x\to\infty$, since, for such $k$, the  level set $\{ n\leqslant x: \omega (n)=k\} $ is so thin---see for instance \cite[pp. 311-312]{TenlivreUS}--- that the stated upper bound does not enable to recover those of \eqref{BV} or  \eqref{but}.

\subsection{Back to the original question}  
We now address the problem of the average distribution of a function $f$ in $\FF (D, K)$ as stated in \eqref{BV0}. In \S \ref{proofoftheoremabs}, taking advantage of the combinatorial preparation leading to Proposition \ref{firststep}, we  provide
a condensed proof of the following theorem, which may be seen as a  variant of a result of Wolke \cite[Satz 1]{Wolke73}. 
  \begin{theorem} 
 \label{centralabsolutevalue}   
The following statement holds for suitable, absolute $C_0$. Let $A>0$,  $K>0$.
There exist $B= B( A,K) $ and $C= C(A, K)$  such that, uniformly for
\begin{align*}
 D\geqslant 1, \quad  f\in \FF (D,K), \quad x\geqslant 1,\quad Q\leqslant \sqrt{x}/\LL^{B},   
\end{align*}
we have
\begin{equation*}
\sum_{  q\leqslant Q} \max_{(a, qD)=1} \ \bigl\vert \Delta_f (x;q  ,D,a ) \bigr\vert   \leqslant \frac{ C  D^{C_0}  x}{\LL^A}\cdot
\end{equation*}
 \end{theorem}
We note here that, at the cost of mild modifications in the arguments, this statement could be used in place of Corollary \ref{easy1} for the proofs of Theorems \ref{EW}, \ref{thEK} and \ref{thLIL} stated {\it infra}. Indeed, in all three instances, one may manage to use a level of distribution $<\tfrac12$ (actually any  strictly positive value suffices for Theorems \ref{EW} and \ref{thEK} while the proof of Theorem \ref{thLIL} requires a level arbitrary close to $\tfrac12$) and appeal to  Cauchy-Schwarz or Hölder's inequality in order to deal with weights bounded above by some function $\tau_K$. However, applying Corollary \ref{easy1} turns out to be simpler and  more straightforward.

\subsection{Comments}
 Fouvry and Radziwi\l\l  \, \cite[cor. 1.3]{FouRadzi1}\, proved that, for any multiplicative function  $f$, satisfying, instead of  condition \eqref{important},   the  more general
 assumption that $p \mapsto f(p)$ is of Siegel--Walfisz type (see Definition \ref{SW} above),
  we have the bound
\begin{equation}\label{weak}
\sum_{\substack{Q< q\leqslant 2Q\\ (q,a)=1}} \bigl\vert \Delta_f (x;q,a) \bigr\vert \ll_{a, \varepsilon} \frac x{\LL^{1-\varepsilon }},
\end{equation}
for any integer $a\not= 0$ and any $\varepsilon >0$, provided $Q \leqslant  x^{17/33-\varepsilon}$. At first sight this result may 
seem  deeper  than \eqref{244} since the error terms $\Delta_f$ are summed in modulus.  However the upper  bound  in \eqref{weak}  appears to be too weak
for the applications described in the next section.


\section{Applications}\label{sectapps}
We list here, among many possible ones, four applications of our main results. The first is a quick, natural proof of theorems 1.3, 1.4 and 1.6 of \cite{DrTo} and their corollaries.
\par
Let us start with sketching the proof of a weaker form of \cite[th. 1.3]{DrTo} in this framework, namely show that, for all $A>0$, $N\geqslant 1$,   and uniformly for $|z|\leqslant A$, $x\geqslant 2$, $1\leqslant |h|\leqslant \LL^A$, we have
\begin{equation}
\label{DT1.3}
\sum_{|h|<n\leqslant x}\tau_z(n)\tau(n+h)=x(\log x)^z\Bigg\{\sum_{0\leqslant j\leqslant N}\frac{\lambda_{h,j}(z)}{(\log x)^j}+O\bigg(\frac 1{(\log x)^{N+1}}\bigg)\Bigg\},
\end{equation}
where the $\lambda_{h,j}$ are entire functions. Note that the error term above is actually slightly more precise than stated in \cite{DrTo}.
\par
Let $S$ denote the left-hand side of \eqref{DT1.3}. Using the symmetry of the divisors of $n+h$ around $\sqrt{n+h}$, we may write, for any $c>\tfrac12$,
\begin{align*}
S&=2\sum_{|h|<n\leqslant x}\tau_z(n)\sum_{\substack{d\,\mid\, n+h\\d< \sqrt{n+h}}}1+O\big(x^c\big)\\
&=2\sum_{t\,\mid\,|h|}\sum_{\substack{m\leqslant \sqrt{x+h}/t\\(m,h/t)=1}}\sum_{\substack{\max\{|h|/t,m^2t-h/t\}<n\leqslant x/t\\ n\md{-h/t}m}}\tau_z(nt)  +O\big(x^c\big).
\end{align*}
From this point on, the strategy is clear and so we omit  the computational details: (i) show that if $m\leqslant M:=x^{1/5}$, say, then one can ignore the lower constraint in the inner $n$-sum; (ii) split the $m$-sum into intervals $]M\Delta^j,M\Delta^{j+1}]$ with $\Delta:=1+1/\LL^C$ and $C$ sufficiently large in terms of~$N$; (iii) show that, to within the required accuracy, one can replace, in the $n$-sum, the lower limit $m^2t$  by $M^2\Delta^{2j}t$; (iv) apply Theorem \ref{central} with $R=1$, $Q=M\Delta^j$ and $Q=M\Delta^{j+1}$; (v) apply the Selberg-Delange method as displayed in \cite[ch. II.5]{TenlivreUS} to sum $\tau_z(nt)$ over integers coprime to $m$ (see \cite[(5.36) p. 287]{TenlivreUS}); (vi) rearrange  the main terms by expanding the various powers of $\log (x/t)$ and $\log (M\Delta^j)$. 
\par
Observe that the assumptions of \cite[th. 1.3]{DrTo} are more flexible than those of the above statement, inasmuch they  assert that   \eqref{DT1.3} holds in the larger domain 
$1 \leq \vert h \vert \leq x^\delta$ for some small constant $\delta >0$. Such uniformity may also be derived from our approach. We now describe which modifications should be incorporated in order to reach this goal. The key-point concerns  Lemma 
\ref{lemma4} below. Following the original proof of \cite[Theorem 6]{BFI1} and tracking the dependency upon the congruence class $a$ (particularly in applying bounds for sums of Kloosterman sums), it can be seen that the estimate \eqref{423} remains true if the list of conditions \eqref{X0<N<X13} 
is replaced 
by
$$
D, M, N, Q, R \geq 1,\quad \vert a\vert ^\kappa R X^\varepsilon \leq N \leq \vert a\vert ^{-\kappa}X^{-\varepsilon} (X/R)^{1/3},\quad Q^2 R\leq X,
$$
where $\kappa$ is a suitable absolute constant. Under the hypothesis $1\leq \vert a \vert \leq X^\delta$, where $\delta$ is a small positive constant, the effect of the factors 
$\vert a \vert^{\pm\kappa}$ in the above conditions is absorbed by other terms provided some exponents in the sequel of the proof of Theorem \ref{central} are slighted modified. In conclusion, we claim that,  provided  condition $R \leq x^{1/105 -\varepsilon}$ is replaced by $R \leq x^{1/106 -\varepsilon}$, Theorem \ref{central} still holds true if hypothesis $1 \leqslant \vert a \vert \leqslant \LL^A$ is relaxed to $1\leqslant  \vert a \vert \leqslant x^\delta$. 
\par  \medskip
Our second application is  a theorem of Erd\H os--Wintner type conditional to the level set of an additive function at shifted argument. As a typical illustration, we prove the following statement, corresponding to the case when the additive function employed to define the support is the number of prime factors function, $\omega$. \par 
Let us recall that  the classical Erd\H os--Wintner theorem (see, \eg, \cite[th. III.4.1]{TenlivreUS}) states that the convergence of the following three series  is  necessary and sufficient for a real, additive function $f$ to possess a limiting distribution $F$:
\begin{equation}
\label{CNS-EW}
\sum_{|f(p)|>1}\frac1p,\qquad \sum_{|f(p)|\leqslant 1}\frac{f(p)^2}p,\qquad \sum_{|f(p)|\leqslant 1}\frac{f(p)}p\cdot
\end{equation}
When this is the case, the characteristic function of $F$ is given by the formula
\begin{equation*}
\varphi_F(\vartheta):=\int_\R\e^{i\vartheta t}\d F(t)=\prod_{p}\Big(1-\frac1p\Big)\sum_{\nu\geqslant 0}\frac{\e^{i\vartheta f(\pnu)}}{\pnu}\qquad (\vartheta\in\R),
\end{equation*}
where the convergence of the infinite product is a consequence of that of the three series \eqref{CNS-EW}.\par 
 We shall consider the family of distribution functions $F_r$ $(r>0)$ with characteristic functions
\begin{equation}
\label{CarFr}
\varphi_{F_r}(\vartheta):=\prod_{p}\Bigg(1+\Big(1-\frac1p\Big)\sum_{\nu\geqslant 1}\frac{e^{i\vartheta f(\pnu)}-1}{p^{\nu-1}(p-1+r)}\Bigg),
\end{equation}
so that $F_1=F$. From a classical theorem of Lévy (see, e.g. \cite[th. III.2.7]{TenlivreUS}, it follows that $F_r$ is continuous if, and only if 
\begin{equation}
\label{CNScont}
\sum_{f(p)\neq0}1/p=\infty.
\end{equation}
and that  the set of discontinuities is otherwise included in $f(\N^*)$. We henceforth define a common continuity set $\C(f):=\R$ if \eqref{CNScont} holds and $\C(f):=\R\smallsetminus f(\N^*)$ otherwise.
\begin{theorem}
\label{EW} Let $f$ be a real, additive function satisfying \eqref{CNS-EW}. Then, uniformly for $$0\leqslant r:=(k-1)/\log_2x\ll1, $$  we have
\begin{equation}
\label{fEW}
\frac1{\pi_k(x)}\sum_{\substack{1<n\leqslant x\\\omega(n-1)=k\\f(n)\leqslant t}}1=F_r(t)+o(1) \qquad \big(t\in\C(f),\,x\to\infty\big).
\end{equation}
\end{theorem}
The proof is given in Section \ref{pfTh2.1}. From general results on weak convergence of distribution functions, it follows that formula \eqref{fEW} is actually uniform with respect to $t$ on any compact subset of $\C(f)$ and valid uniformly for $t\in\R$ if \eqref{CNScont} holds.
\par 
For $k=1$, we recover, in a slightly more general setting, a result of K\'atai \cite{Ka68}.
\par 
It is of interest to observe that if $\omega(n-1)$ is replaced by $\omega(n)$ in \eqref{fEW} then, as shown in \cite{TV20},  a limiting distribution still occurs but has a different value for $r\neq1$: when, for instance, $f$ is strongly additive, the corresponding characteristic function turns out to be
 \begin{equation}
 \label{chfTV}
 \prod_{p}\bigg(1+\frac{r\big(\e^{i\vartheta f(p)}-1\big)}{p-1+r}\bigg)
 \end{equation}
 whereas  in this case
\begin{equation}
\label{chFr-fa}
\varphi_{F_r}(\vartheta)=\prod_{p}\bigg(1+\frac{\e^{i\vartheta f(p)}-1}{p-1+r}\bigg).
\end{equation}
Qualitatively, these results tell us that, as expected, the perturbation is more significant in the case when the same variable is used for the additive function and the definition of the level set, as is clear from comparing the coefficients of $\e^{i\vartheta f(p)}-1$ in \eqref{chfTV} and \eqref{chFr-fa}. In the case of a shifted argument, the distributions of $\omega(n-1)$ and $f(n)$ are ``almost''  independent.
\par 
Of course \eqref{fEW} opens the way to estimating the distribution function of an additive satisfying \eqref{CNS-EW} with respect to various probability measures related to the function $\omega(n-1)$. As an illustration, we state without proof a standard consequence, the proof of which simply involving a re-summation procedure and, say, a weak version of \cite[th. II.6.1]{TenlivreUS}.
\begin{corollary}
\label{EW-y-om}
Let $y>0$ and let $f$ be a real, additive function satisfying \eqref{CNS-EW}. Then, we have
$$\sum_{\substack{1<n\leqslant x\\ f(n)\leqslant t}}y^{\omega(n-1)}=\Big\{F_y(t)+o(1)\Big\}\sum_{n\leqslant x}y^{\omega(n)}\qquad \big(t\in\C(f),\,x\to\infty\big).$$
\end{corollary}
\medskip
A third application, which we shall not develop here in full generality, is the variant of the previous one consisting in establishing an Erd\H os-Kac theorem over the level set of an additive function  at shifted argument.  Letting $\Phi(t)$ denote the normalized Gaussian distribution function, a typical statement in this direction is as follows.
\begin{theorem}
\label{thEK}
Let $f$ be a real, strongly additive arithmetical function such that
\begin{align}
&B_x^2:=\sum_{p\leqslant x}\frac{f(p)^2}p\to\infty\quad (x\to\infty),&\label{Var}\\
&B_y\sim B_x \quad (y:=x^{1/\log_2x},\,x\to\infty),&\label{Var+}\\
&\sum_{y<p\leqslant x}\frac{f(p)}p=o\big(B_x\big),\quad(\forall\varepsilon>0)\ \sum_{\substack{p\leqslant x\\|f(p)|>\varepsilon B_x}}\frac{f(p)^2}p=o\big(B_x^2\big)\quad(x\to\infty).
&\label{Lind}\end{align}
Then, uniformly for $x\to\infty$, $1\leqslant k\ll \log_2x$ and $t\in\R$, we have 
\begin{equation*}
\frac1{\pi_k(x)}\sum_{\substack{1<n\leqslant x\\ \omega(n-1)=k\\ f(n)\leqslant A_x+tB_x}}1=\Phi(t)+o(1),
\end{equation*}
with $A_x:=\sum_{p\leqslant x}f(p)/p$.
\end{theorem}
This result generalizes to a natural framework a recent work of Goudout \cite{Go20} in which $f=\omega$. When $k=1$, it follows from a general theorem of Barban, Vinogradov and Levin \cite{BVL65}.
\par 
The hypotheses of the above theorem could be lightened further with some extra work. Note that \eqref{Var} is classical, that \eqref{Var+} is a slight strengthening of the requirement in Kubilius' class~$H$ (see \cite[ch. IV]{Kub64}), and that the second condition in \eqref{Lind} coincides with the usual Feller-Lindeberg condition---see, e.g. \cite[lemma 1.30]{Elliott79}. \par 
\par 
As in the case of Theorem \ref{EW}, we can derive, by re-summation over $k$, a number of estimates  from Theorem \ref{thEK}. The following statement, the proof of which we leave to the reader, is emblematic. 
\begin{corollary}
Let $f$ be a real, strongly additive arithmetical function satisfying \eqref{Var}, \eqref{Var+} and \eqref{Lind} and let $\wp_x$ denote a probability measure on $]1,x]$ ascribing to each integer $n$ a weight depending only on $\omega(n-1)$. Assume furthermore that $$\wp_x\big(\omega(n-1)>T\log_2x\big)=o(1)\qquad (T,x\to\infty).$$ Then, uniformly for $t\in\R$, we have
$$\wp_x\big(f(n)\leqslant A_x+tB_x\big)=\Phi(t)+o(1)\qquad (t\to\infty).$$ 
\end{corollary}  

\medskip
Our fourth application is a law of iterated logarithm for integers weighted with $\tau(n-1)$.  More precisely, given a function $\xi(x)\to\infty$, this deals with the behaviour of the quantities
\begin{equation}
\label{defsLIL}
\begin{aligned}
&\Lambda(n,t):=\frac{\omega(n,t)-\log_2t}{\sqrt{2\log_2t\log_4t}}\ \big(t>\xi(x)\big),\quad M(n,\xi):=\sup_{\xi(x)<t\leqslant x}|\Lambda(n,t)|,\\
& M^+(n,\xi):=\sup_{\xi(x)<t\leqslant x}\Lambda(n,t),\quad M^-(n,\xi):=\inf_{\xi(x)<t\leqslant x}\Lambda(n,t).
\end{aligned}
\end{equation}
It will be convenient to equip $\{1<n\leqslant x\}$ with the probability $P_x$ ascribing to each integer $n$ a weight proportional to \mbox{$\tau(n-1)$}. With this setting, we  prove the following result in Section \ref{pfLIL}. 
\begin{theorem}
\label{thLIL}
Let $\varepsilon>0$. If $\xi(x)$ tends to $\infty$ with $x$ sufficiently slowly, then we have
\begin{equation}
\label{LIL-maj}
P_x\Big(M(n,\xi)\leqslant 1+\varepsilon\Big)=1+o(1)\qquad (x\to\infty).
\end{equation}
Moreover, 
\begin{align}
&P_x \Big(M^+(n,\xi)\geqslant 1-\varepsilon\Big)=1+o(1)\qquad (x\to\infty),& \label{LIL+}\\
&P_x \Big(M^-(n,\xi)\leqslant -1+\varepsilon\Big)=1+o(1)\qquad (x\to\infty).& \label{LIL-}
\end{align}
\end{theorem}
It is well-known that the sum 
$$\sum_{n\leqslant x}\tau(n)=x\log x+O(x)$$
is dominated by integers with $\omega(n)\sim 2\log_2x$. Moreover, these prominent integers actually satisfy $\omega(n,t)\sim 2\log_2t$ uniformly for $\xi(x)<t\leqslant x$, and so a twisted version of the law of iterated logarithm could be proved for the probability on $[1,x]$ ascribing to each integer $n$ a weight proportional to $\tau(n)$. Thus, from a qualitative perspective, Theorem \ref{thLIL} tells us that, unlike the weight $\tau(n)$ which has the effect of recentering  averages on integers for which $\omega(n,t)$ is roughly twice its normal order, the weight $\tau(n-1)$ involves very little perturbation on the fine structure of the sequence of prime factors of $n$.
 
 \section{Proofs of Corollaries \ref{easy1} and \ref{easy2}}
\subsection{Proof of Corollary \ref{easy1}} We show here how Corollary \ref{easy1} may be deduced from Theorem \ref{central}.
\par 
 By the Cauchy--Schwarz inequality  and the size hypothesis for the $\xi_r$, we have
\begin{equation}
\label{283}
\Bigg\vert \sum_{\substack{r\leqslant R\\ (r,a)=1}}\xi_r\sum_{\substack{q\leqslant Q \\ (q,a)=1}}  \Delta_f (x;qr,D,a)\Bigg\vert^2 \leqslant \Sigma_1 \Sigma_2,
\end{equation}
with
\begin{equation*}
\begin{aligned}
&\Sigma_1:=  \sum_{\substack{r\leqslant R\\ (r,a)=1}}\tau_K^2 (r) \Bigg\vert \sum_{\substack{q\leqslant Q \\ (q,a)=1}}  \Delta_f (x;qr,D,a) \Bigg\vert,\qquad \Sigma_2:= \sum_{\substack{r\leqslant R\\ (r,a)=1}}\Bigg\vert \sum_{\substack{q\leqslant Q \\ (q,a)=1}} \, \Delta_f (x;qr,D,a) \Bigg\vert.
\end{aligned}
\end{equation*}\par 
From Theorem \ref{central}, we get
\begin{equation}
\label{Sigma2<<}
\Sigma_2 \ll D^{C_0}\,x/\LL^{A}.
\end{equation}
\par 
We bound $\Sigma_1$ trivially as in Lemma \ref{exceptional}, equation  \eqref{G<}, {\it infra}. Selecting $S=QR$, we get 
$$
\Sigma_1 \ll \tau_K (\vert a \vert) QR \LL^\idc +x \LL^\idc,
$$
where $\idc= \idc (\varepsilon,K)$. Inserting this last bound and  \eqref{Sigma2<<} into \eqref{283} completes the proof.

\subsection{Proof of Corollary \ref{easy2}}
For $t \in \TT $, choose $z_t\in \mathbb C$ with $\vert z_t\vert =1$. Define ${\bsz} :=(z_t)_{t \in \TT }$ and define the strongly multiplicative
function $f_{\bsz}$ by its values on primes as follows 
$$
f_{\bsz} (p)=
\begin{cases} z_t & \text{ if } p\md tD\ (t\in \TT ),\\
1 & \text{ if } (\forall t\in\TT)\ p\not\md tD.
\end{cases}
$$
The function $f_{\bsz}$ belongs to $\FF (D,1)$ and, as a consequence  of Cauchy's integral formula,  we have the equalities
\begin{equation}
\label{integralequation}
\begin{aligned}
\chi_{D, \TT , \Phi} (n)&= \prod_{t \in \TT } \Bigg( \frac{1}{2\pi i}   \oint_{|z_t|=1}\frac{z_t^{\omega_{D,t} (n) -\Phi (t)}}{z_t} \, {\rm d} z_t\Bigg)\\
& =   \Bigl( \frac{1}{2\pi i}\Bigr)^{\vert \TT \vert} \oint_{\substack{|z_t|=1\\ (t\in\TT)}}\frac{f_{\bsz} (n)}{\prod_{t\in \TT } z_t^{\Phi (t) +1}} {\rm d} {\bsz},
\end{aligned}
\end{equation}
Now, let us insert  this expression  into the left hand side of \eqref{244ter1}, interchange summation and  integration, apply  bound \eqref{244ter}, and finally   integrate trivially over the product of the units circles $|z_t|=1$. This completes the proof of \eqref{244ter1}.

\section{Lemmas}
\subsection{Classical lemmas} Our first lemma provides a bound for short sums of powers of the Piltz divisor function over arithmetic progressions---see \cite [Theorem 2]{Sh} for instance.
\begin{lemma} \label{lemma1}Let $K\in\N^*,\,\ell\in\N^*$,  and $\varepsilon >0$ be fixed.  Then, uniformly for $x^\varepsilon  \leqslant y < x$,  $1\leqslant q \leqslant y/x^ {\varepsilon}$, $(a,q)=1$,
we have
$$
\sum_{\substack{x-y <n\leqslant x\\ n\equiv a \bmod q}} \tau_K ( n )^\ell \ll   \frac{y}{ q} \LL^{K^\ell-1}.
$$
\end{lemma}
We next consider the general sum 
  $
\mathcal G = \mathcal G (\boldsymbol c, f, S, x) $ 
defined by
$$
\mathcal G: = \sum_{\substack{s\leqslant S\\ (s,a)=1}} \vert c_s \vert \Bigg\{\sum_{\substack{n\leqslant x \\ n\equiv a \bmod s}} \vert f(n)\vert + \frac{1}{\varphi (s)} \sum_{\substack{n\leqslant x \\ (n,s)=1}}\vert f(n)\vert \Bigg\}.
$$
 It will be useful to dispose of a bound for $\mathcal G$ when the summation variables 
 $n$ or $s$ are restricted to a sparse set. 
\begin{lemma}\label{exceptional}Let $K\geqslant 1$ be an integer and $\varepsilon >0$ be given. There exists $\idc = \idc (\varepsilon,K)$ such that,  uniformly  for 
$$x \geqslant 1,\quad S\in\N^*\cap\big[1, x^{1-\varepsilon}\big],\quad1 \leqslant \vert a \vert \leqslant 10\,x,\quad\boldsymbol c = (c_s)\in\CC^S,\quad f\in\CC^{\N^*},$$ and assuming that both $\boldsymbol c$ and $f$ satisfy $\eqref{f<tau}$, we have
 \begin{equation}\label{G<}
\mathcal G 
 \ll \vert \NN\vert ^{1/3} x^{2/3} \mathcal \LL ^{\idc} +  \tau_K (\vert a \vert)S\LL^\idc,
\end{equation}
where $\NN:=\{n\leqslant x:  f(n) \not=0\} $. Moreover, we  also have
\begin{equation}\label{G<<}
\mathcal G    
 \ll \Bigg(\ \sum_{\substack{s\leqslant S\\ c_s \not= 0}}\frac{1}{\varphi (s)} \Bigg)^{1/2} x\LL^\idc.
 \end{equation}

\end{lemma} 
\begin{proof} 
We have
\begin{align*}
\mathcal G   
& \leqslant \sum_{s\leqslant S} \tau_K (s)\sum_{\substack{n\in \NN,\, n\not= a \\ n\equiv a \bmod s}} \tau_K (n) +\vert f (\vert a \vert) \vert \,S\LL^\idc +
\sum_{  n\in \NN} \tau_K (n) \sum_{s\leqslant S} \frac{\tau_K (s)}{\varphi (s)}\\
& \ll \sum_{n \in \NN, \, n\not= a} \tau_K (n) \tau_{K+1} (\vert n-a\vert)+\tau_K (\vert a \vert ) S \LL^\idc + \LL^\idc  \Bigg(\vert \NN\vert \sum_{n\leqslant x} \tau_K^2 (n)\Bigg) ^{1/2}\\
& \ll  \Bigg( \vert \NN\vert\sum_{n\leqslant x} \tau_K( n)^3\Bigg)^{1/3} \Bigg( \sum_{\substack{n\leqslant x\\ n\not= a}} \tau_{K+1} (\vert n-a\vert )^3\Bigg)^{1/3}  +\tau_K (\vert a \vert ) S \LL^\idc +\sqrt{x\vert \NN\vert} \LL ^\idc.
\end{align*}
Appealing to Lemma \ref{lemma1}, this furnishes  \eqref{G<}.
\par 
To prove \eqref{G<<}, let us denote by $\ES$ the set of those integers $s \leqslant S$ such that $c_s \not= 0$. Applying Lemma \ref{lemma1}  and the Cauchy--Schwarz inequality, we obtain 
\begin{align*}
\mathcal G &\ll  \Bigg(\ \sum_{s \in \ES}\frac{\tau_K (s)}{\varphi (s)}\ \Bigg) x\LL^\idc \\
&\ll \Bigg(\ \sum_{s\leqslant S}\frac{\tau_K (s)^2}{\varphi (s)}\ \Bigg)^{1/2}  \Bigl(\ \sum_{s \in \ES}\frac{1}{\varphi (s)}\ \Bigr)^{1/2}  x\LL^\idc\ll  \Bigl(\ \sum_{s \in \ES}\frac{1}{\varphi (s)} \Bigr)^{1/2} x\LL^\idc,
\end{align*}
as required.
\end{proof}

The following lemma (see, e.g., \cite[ex. 293]{TenlivreUS}, \cite[ex. 293]{TW14}) provides a quantitative form of the assertion that the product of small prime factors of an integer is usually small. 
\begin{lemma}\label{gerald}Uniformly  for $2\leqslant y\leqslant z\leqslant x$, we have 
$$
\Bigl\vert \Bigl\{ n\leqslant x : \prod_{\substack{p^\nu \Vert n \\ p \leqslant  y}} p^\nu > z\Bigr\}
\Bigr\vert \ll x \, \exp \Bigl( - \frac{\log z}{2\log y} \Bigr).
$$
\end{lemma}

The next lemma is proved via a standard application of Rankin's method. We recall the notation \eqref{balpha} for the multiplicative function $b_\alpha$.

\begin{lemma}\label{rankin}
Uniformly for integers $m\geqslant 1$, $n\geqslant 1$, and real $y\geqslant 1$, we have
$$
\sum_{\substack{\delta \leqslant y\\  \delta \vert m^\infty,\,  (\delta, n)=1}}
 \frac{1}{\delta}=\prod_{\substack{p\vert m\\ p\,\nmid\, n}}
 \Bigl( 1-\frac 1p\Bigr)^{-1}+O\bigg( \frac{b_{3/4} (m)}{y^{1/4} }\bigg).
$$
\end{lemma}
\begin{proof} The main term of the stated formula coincides with the limit of the left-hand side as $y\to\infty$. The remainder is  
$$
\sum_{\substack{\delta > y\\  \delta \vert m^\infty,\,  (\delta, n)=1}}
\frac{1}{\delta}\ll
\sum_{\substack{\delta \mid m^\infty\\ (\delta , n) =1}}\frac {1}{\delta}\Bigl( \frac{\delta}{y}\Bigr)^{1/4}
.
$$
\end{proof}
Finally we apply  an identity due to Heath-Brown---see \cite[prop. 13.3]{IwKo}--- in order to split the von Mangoldt function into sums of bilinear forms.  
\begin{lemma}\label{H-Bid}  Let $y\geqslant 1$ and let   $n$, $J$,  be two integers such that  $J\geqslant 1$, $1\leqslant n\leqslant 2y$. 
We have 
$$
\Lambda (n) = \sum_{1\leqslant j\leqslant J}  {(-1)^{j-1}} \binom{J}{j}\underset{\substack{\prod_{h=1}^jm_hn_h =n\\ \max_{1\leqslant h\leqslant j}m_h\leqslant y^{1/J}}}{\sum}\prod_{1\leqslant h≤j}\mu (m_h) \log n_1.
$$
\end{lemma}
\subsection{Lemmas of Siegel--Walfisz type.} Recall that $w_1$ denotes the indicator function of the set of prime powers. The classical Siegel--Walfisz theorem asserts that the bound 
\begin{equation}\label{SWforpi1}
\Delta_{w_1} (x;q,a) \ll_A x/\LL^{A}
\end{equation}
holds for any fixed $A$ and all coprime integers $a$, $q$.\par 
A Siegel--Walfisz theorem for the M\"obius function is also known, viz.
$$
\Delta_{\mu} (x;q,a) \ll_A x/\LL^{A},
$$
see for instance \cite[cor. 5.29]{IwKo}. \par 
We shall need extensions of these two results. Recall the definition of $\mathfrak Y_0$ in~\eqref{gY0}.
\begin{lemma} 
For every $A>0$, there exists a constant $C(A)$ such that,  for  all $Y_0>1$, $x\geqslant 1$ and coprime integers $a$ and $q$ with $q\geqslant 1$, we have
\begin{align}
\label{SWforY0}
\bigl\vert\, \Delta_{\mathfrak Y_0}(x;q,a)\, \bigr\vert &\leqslant C(A)\, x/\LL^{A},\\
\label{SWformuY0}
\bigl\vert \Delta_{\mu \mathfrak Y_0}(x;q,a)\, \bigr\vert &\leqslant C(A)\, x/ \LL^{A}.
\end{align}
\end{lemma}
\begin{proof} To prove \eqref{SWforY0}, it is plainly sufficient to establish a bound of same type for
$$
 \Delta'_{\mathfrak Y_0} (x;q,a):= \Delta_{\mathfrak Y_0} (x;q,a)- \Delta_{\mathfrak Y_0} (x/2;q,a).
$$
Every integer $n\in]x/2,x]$ may be uniquely represented as 
$$
n=pm, \text { with }  P^+ (m)\leqslant p\leqslant x/m.$$
Let  $\EE$ be the set  of those integers  $n\in]x/2,x]$ such that $x/m < y:=\e^{\sqrt{\log x}},$  and hence
such that $p=P^+ (n) \leqslant y$.  Lemma \ref{gerald} then yields that $\EE$ is a thin set, indeed
\begin{equation} 
 \label{majEE}
\vert \EE\vert \ll x/\sqrt{y}.
\end{equation}
Now we have
$$
 \Delta'_{\mathfrak Y_0}(x;q,a)= \sum_{\substack{m\leqslant x\\ (m,q)=1}}
 \mathfrak Y_0 (m)  \sum_{\max(x/(2m), Y_0, P^+(m))\leqslant p \leqslant x/m}   g_q(p; a \overline m),
$$
where $\overline {m}$ is the multiplicative inverse of $m$ modulo $q$.  From \eqref{majEE} and \eqref{SWforpi1} we deduce that
$$\vert  \Delta'_{\mathfrak Y_0}(x;q,a)\vert 
\ll    \sum_{\substack{m\leqslant x/y\\ (m,q)=1}}   \frac{\mathfrak Y_0 (m)x}{m \bigl( \log (x/m)\bigr)^A} + \frac x{\sqrt{y}},
$$
where $A$ is arbitrary. Summing over $m$ furnishes the required estimate.
\par 
The proof of \eqref{SWformuY0} is similar.

\end{proof}

\subsection{Lemmas from the dispersion technique and  bounds on Kloosterman sums}  

Let  $\bal = (\alpha_m) $ and $\bbe = (\beta_n)$ be  complex sequences. For $M,\,N>1$ we consider 
\begin{equation}
\label{defDab}
\Delta_{\bal,\bbe} (M, N; q, a) :=  \underset{\substack{m\sim M, n\sim N\\ mn\equiv a \pmod q}}{\sum \sum} 
\alpha_m \beta_n -\frac{1}{\varphi (q)}
 \underset{\substack{m\sim M, n\sim N\\ (mn, q)=1}}{\sum \sum} 
\alpha_m \beta_n. 
\end{equation}
which is a variant of $\Delta_f$ (see \eqref{defDeltaf}) where  the sizes of the variables $m$ and $n$ in the convolution  $\bal * \bbe$ are restricted to dyadic intervals.  
We also introduce
$$
X:=MN,$$
and for $1\leqslant u < v$ the shortened  sum $ \Delta^{u,v}_{\bal,\bbe} (M, N; q, a)$
$$
\Delta_{\bal,\bbe}^{u,v}(M, N; q, a) :=  \underset{\substack{m\sim M, n\sim N\\ mn\equiv a \pmod q\\u<mn\leqslant v}}{\sum \sum} 
\alpha_m \beta_n -\frac{1}{\varphi (q)}
 \underset{\substack{m\sim M, n\sim N\\ (mn, q)=1\\ u<mn\leqslant v}}{\sum \sum} 
\alpha_m \beta_n, 
$$
where the variables $m$ and $n$ satisfy the extra multiplicative constraint $u<mn\leqslant v$. Note that $\Delta_{\bal,\bbe}^{u,v}$ vanishes when $u>4X$  or $v<X$.

We now list several lemmas providing 
instances in which the  $X^\frac 12$-barrier   for the level of distribution of the convolution $\bal * \bbe$ can be overpassed. The proofs of Lemmas \ref{lemma4} and \ref{lemma5} below are both based on Linnik's dispersion method and on bounds of various types
of Kloosterman sums. 
 
The first part of the following statement is a weak version  of \cite [Theorem 6, p.~242]{BFI1} which however will
be sufficient for our purpose.  Extending the validity to $\Delta_{\bal,\bbe}^{u,v}$ is now standard by using the Mellin transform of a smooth approximation to $\1_{]u,v]}$ in order  to separate the variables $m$ and $n$ in the  summation condition $u<mn\leqslant v$---see for instance  the proof of \cite[cor. 1.1]{FouRadzi1}  or \cite[p. 371--372]{BFI2}. Therefore, we omit the proof of this extension.\par  Similarly,
albeit \cite [Theorem 6]{BFI1} only deals with the  case $D=1$, we omit the straightforward proof 
of the extension to general $D$.
\goodbreak
\begin{lemma}\label{lemma4} Let  $A>0$,  $K>0$,  $\varepsilon >0$.
Uniformly for 
\begin{equation} 
\label{X0<N<X13}  
\begin{aligned}
&D,\, M,\,  N,\, Q, \,R \geqslant 1, \quad  1\leqslant \vert a \vert \leqslant (\log X)^A, \\
&RX^\varepsilon \leqslant  N \leqslant X^{ -\varepsilon} ( X/R)^{1/3} ,\quad Q^2R \leqslant X,
\end{aligned}
\end{equation} 
all $\bbe = ( \beta_n)\in\CC^{\N^*}\cap\SW(D,K)$ satisfying the sifting condition
\begin{equation}
\label{sifting}
 P^-(n)\leqslant \e^{(\log_2n)^2}\Rightarrow \beta_n =0,
\end{equation} 
and all $\bal =(\alpha_m) \in \CC^{\N^*}$, we  have
\begin{equation}
\label{423}
\sum_{\substack{r\leqslant R\\  (r,aD)=1}} \Bigl\vert \, \sum_{\substack{q \simeq Q \\  (q,aD)=1}} \Delta_{\bal,\bbe}(M, N; qr, a) \,\Bigr\vert \ll \frac{ \no\bal_M  \no\bbe_N \sqrt{X} }{ (\log X)^A}\cdot
\end{equation}   
\par 
Under the same assumptions, the bound \eqref{423}  persists, uniformly for $1\leqslant u<v$, on replacing $\Delta_{\bal,\bbe}$ by $\Delta_{\bal,\bbe}^{u,v}$.
\par 
In particular,  estimate \eqref{423} and its extension to $\Delta_{\bal,\bbe}^{u,v}$ hold uniformly for 
\begin{align*}
&D,\,  M,\,  N,\, Q, \,R \geqslant 1, \quad 1\leqslant \vert a \vert \leqslant (\log X)^A, \\
& X^{1/105} \leqslant N \leqslant X^{2/7}, \quad 1\leqslant R \leqslant X^{1/105-\varepsilon}, \quad  Q^2 R \leqslant X.\\
\end{align*}
\end{lemma}
   Fouvry  \cite [Proposition 1, p. 61]{FouCrelle} proved  a similar result in the special case
 $\bbe=\mu$ or $\1$. These special cases are  sufficient  
to derive Theorem \ref{pointdedepart}  above. 

When $N=X^{o(1)}$, Lemma \ref{lemma4} does not allow selecting $R= X^\delta$, for some fixed \mbox{$\delta >0$}. The following result, due to Fouvry and Radziwi\l\l \,\cite[cor. 1.1,(i) \& Proposition 8.1(i)]{FouRadzi1},  fills this gap when $D=1$. Here again the extension to  general $D\geqslant 1$ is straightforward. Accordingly, we state the following lemma.
\begin{lemma} \label{lemma5}
 Let $A>0$,   $K>0$, $\varepsilon >0$. Uniformly for $$D,\, M,\, N,\, Q\geqslant 1,\quad 1\leqslant \vert a\vert \leqslant \tfrac1{12} X,\quad
\e^{(\log X)^\varepsilon} \leqslant N \leqslant Q^{-11/12} X^{17/36 -\varepsilon}.
$$
and all complex sequences $\bal=( \alpha_m)\in\CC^{\N^*}$, $\bbe =(\beta_n)\in\CC^{\N^*}\hskip-1.5mm\cap\SW(D,K)$ such that
\begin{equation}\label{boundforalphaandbeta}
\vert \alpha_m \vert \leqslant \tau_K (m) \ (m\geqslant 1),\qquad  \vert \beta_n\vert \leqslant \tau_K (n)\ (n\geqslant 1),
\end{equation}
   we have 
\begin{equation}\label{FoRa}
\sum_{\substack{q\leqslant Q \\  (q,aD)=1}}\vert \, \Delta_{\bal,\bbe} (M, N; q, a)\, \vert  \ll \frac X{(\log X)^A}\cdot
\end{equation}
\par 
Under the same hypotheses, the bound \eqref{FoRa}  persists, uniformly for $1\leqslant u<v$, on replacing $\Delta_{\bal,\bbe}$ by $\Delta_{\bal,\bbe}^{u,v}$.

In particular, for all $\varepsilon >0$, $A>0$,  the estimate
\begin{equation}\label{544}
\sum_{\substack{r\leqslant R\\  (r,aD)=1} }\Bigl\vert \sum_{\substack{q\simeq Q \\  (q,aD)=1}}\, \Delta_{\bal,\bbe}(M, N; qr, a)
\Bigr\vert \ll \frac X{(\log X)^A},
\end{equation}
holds uniformly for 
\begin{align*}
&D,\, M,\, N,\,  Q,\, R\geqslant 1,\quad  1 \leqslant \vert a \vert \leqslant \tfrac1{12}X,\\  &\exp\big\{(\log X)^{1/4}\big\} \leqslant N \leqslant  X^{1/105 }, \quad
1 \leqslant R \leqslant X^{1/105-\varepsilon},\quad Q^2R < X.
\end{align*}
Under the same conditions,  the estimate \eqref{544} persists, uniformly for $1 \leqslant u <v$, on replacing $\Delta_{\bal,\bbe}$ by $\Delta_{\bal,\bbe}^{u,v}$ .
\end{lemma}
The derivation of \eqref{544} from \eqref{FoRa} is standard and follows lines similar to those described in the proof of Corollary \ref{easy1} above. \par 
The next statement is obtained by combining Lemmas \ref{lemma4} and \ref{lemma5}.
\begin{lemma}
\label{lemma6} 
Let $A>0$,   $K>0$, $\varepsilon >0$. Uniformly for   
 \begin{align*}
&D,\, M,\, N, \, Q,\, R\geqslant 1,\quad 1\leqslant \vert a \vert \leqslant (\log X)^A, \\ &\exp\big\{(\log X)^{1/4}\big\} \leqslant N < X^{2/7},\quad 
1 \leqslant R \leqslant X^{1/105-\varepsilon},\quad  Q^2R < X,  
\end{align*}
and all complex sequences  $\bal=( \alpha_m)\in\CC^{\N^*}$, $\bbe =(\beta_n)\in\CC^{\N^*}\cap\SW(D,K)$    satisfying   conditions   \eqref{sifting} and \eqref{boundforalphaandbeta}, we have
\begin{equation} 
\label{majDqr}
\sum_{\substack{r\leqslant R\\  (r,aD)=1}} \Bigl\vert \, \sum_{\substack{q \simeq Q \\  (q,aD)=1}} \Delta_{\bal,\bbe}(M, N; qr, a) \,\Bigr\vert \ll \frac X{(\log X)^A}\cdot
\end{equation}
 On replacing  $\Delta_{\bal,\bbe}$ by $\Delta_{\bal,\bbe}^{u,v}$, the estimate \eqref{majDqr} persists, uniformly for $1 \leqslant u<v$. 
\end{lemma}
As will be explained in \S \ref{Dirichlet}, it turns out  that condition $Q^2R< X $ in Lemmas \ref{lemma4}, \ref{lemma5} and \ref{lemma6}, may be relaxed to the weaker condition
$QR < X/(\log X)^B$ for some $B=B(A)$ by exploiting an idea due to Dirichlet.
\subsection{The convolution principle} We recall here a by now classical principle which is implicit in many works related to the Bombieri--Vinogradov theorem.  This principle asserts that the convolution of two well-behaved sequences has an exponent of distribution equal to $1/2$. The following statement is a straightforward  variation of \cite[Theorem 0(b)]{BFI1}.
\begin{lemma}
  Let $A>0$, $K>0$,  $\varepsilon >0$. There exists $B= B( A,K, \varepsilon)$ such that, uniformly for 
 $$D\geq 1, \  M,\, N \geqslant \e^{(\log X)^\varepsilon}, \   
 Q \leqslant X^{1/2} / (\log X)^A,
$$
 and all complex sequences $\bal=( \alpha_m)\in\CC^{\N^*}$, $\bbe =(\beta_n)\CC^{\N^*}\hskip-1.5mm\cap\SW(D,K)$   satisfying 
\eqref{boundforalphaandbeta}, 
we have
\begin{equation}\label{Largesieve}
\sum_{\substack{q\leqslant Q \\  (q, D)=1}}\max_{(a,q)=1}\vert \, \Delta_{\bal,\bbe} (M, N; q, a)\, \vert  \ll \frac X{(\log X)^A}\cdot
\end{equation}
\par 
Under the same hypotheses, the bound \eqref{Largesieve}  persists, uniformly for $1\leqslant u<v$, on replacing $\Delta_{\bal,\bbe}$ by $\Delta_{\bal,\bbe}^{u,v}$.
\end{lemma}

\subsection{Lemmas  from the theory of algebraic exponential sums}  Recall from \S\thinspace\ref{conventions} the definitions of the functions $g_q$  and $\mathfrak Y_0$, the latter being associated to the positive number $Y_0$. The following lemma is trivial when $Y_0=2$. When $Y_0>2$, it is a standard consequence of the fundamental lemma from sieve theory. 
\begin{lemma}\label{tau1} The following statement holds for suitable, absolute constant $C_0$. 
For each $\varepsilon >0$, suitable $\delta= \delta (\varepsilon)$, $c(\varepsilon) >0$, and  uniformly for $x \geqslant 1$,   $M\geqslant 1$, $Y_0\geqslant 2$, and    integers integers $a$, $q$, $t$, $D$  such that
 $$ 
2\leqslant Y_0<     x^{1/100} < M  \leqslant x, \quad 1\leqslant q\leqslant x^{1-\varepsilon}, 
\quad(q,aD)=1,\quad (t,D)=1,$$
 we have 
\begin{equation*}
\underset{\substack{m\simeq M  \\ m\equiv t \pmod D}} {\sum}g_q (m ; a)
 \ll\frac{D^{C_0}  }{\varphi (q)}x^{1 -\delta (\varepsilon)},
\end{equation*}
and, more generally,
\begin{equation*}
 \underset{\substack{m\simeq M \\   m\md tD}} {\sum } g_q (m ; a)\mathfrak  Y_0 (m )    
 \ll \frac{D^{C_0}}{\varphi (q)}x^{1-c(\varepsilon)/\log Y_0}.
\end{equation*} 
\end{lemma}
Lemma \ref{tau1} deals with the distribution of the function $\tau_1=\1$ in arithmetic progressions. 
The following two lemmas concern  the distribution of the functions $\tau_2$ and $\tau_3$. 
They assert that the exponent of distribution  of these functions can be taken $>1/2$.    It remains a  challenging problem to extend these results to the function $\tau_4$.  
\begin{lemma}\label{tau2} The following statement holds for suitable, absolute $C_0$. 
For each $\varepsilon >0$, suitable $\delta= \delta (\varepsilon)$, $c(\varepsilon) >0$, and uniformly for  $x \geqslant 1$,   $M_1, M_2\geqslant 1$, $Y_0\geqslant 2$, and integers $a$, $q$, $t_1$, $t_2$, $D$ , such that    
 \begin{align*} 
&2\leqslant Y_0\leqslant      x^{1/100} \leqslant  M_1 \leqslant M_2, \ M_1M_2 \leqslant x, \\ &1\leqslant q\leqslant x^{2/3-\varepsilon}, 
\ (q,aD)=1,\ (t_1t_2,D)=1,
\end{align*}
 we have 
\begin{equation}\label{tau2arprog}\underset{\substack{m_1\simeq M_1,\ m_2 \simeq M_2  \\ m_i \md{t_i}D\ (i=1,2)}} {\sum\ \quad \sum}g_q (m_1m_2; a)
  \ll\frac{ D^{C_0}}{\varphi (q)}x^{1-\delta (\varepsilon)} ,
\end{equation}
and, more generally,
\begin{equation}\label{tau2arprogmodif}
 \underset{\substack{m_1\simeq M_1,\ m_2 \simeq M_2\\   m_i \md{t_i}D\ (i=1,2)}} {\sum\ \quad \sum} g_q (m_1m_2; a)\mathfrak  Y_0 (m_1m_2)    
 \ll \frac{D^{C_0}}{\varphi (q)}x^{1-c  (\varepsilon) /\log Y_0} .
\end{equation}
\end{lemma}
\begin{lemma}
\label{tau3}
The following statement holds for suitable, absolute $C_0$. 
For each $\varepsilon >0$, suitable $\delta= \delta (\varepsilon)$, $c(\varepsilon) >0$, and uniformly for   $x \geqslant 1$,    $M_1, M_2, M_3\geqslant 1$, $Y_0\geqslant 2$, and integers $a$, $q$, $t_1$, $t_2$, $t_3$, $D$ such that     
 \begin{align*}
&2\leqslant Y_0\leqslant      x^{1/100} \leqslant  M_1 \leqslant M_2\leqslant M_3, \ M_1M_2 M_3\leqslant x,\\
&1\leqslant q\leqslant x^{21/41-\varepsilon},\  (q,aD)=1,\ (t_1t_2t_3,D)=1,
\end{align*}
 we have 
 \begin{equation}\label{tau3arprog}
\underset{\substack{m_1\simeq M_1,\, m_2 \simeq M_2,\, m_3 \simeq M_3 \\ m_i\md{t_i}D\ (1\leqslant i\leqslant 3)}} {\sum\ \quad \sum\ \quad \sum} g_q (m_1m_2m_3; a)  
  \ll\frac{D^{C_0}}{\varphi (q)}x^{1 -\delta (\varepsilon)},
\end{equation}
 and more generally 
\begin{multline}
\label{tau3arprogmodif}\underset{\substack{m_1\simeq M_1,\, m_2 \simeq M_2,\, m_3 \simeq M_3 \\ m_i\md{t_i}D\ (1\leqslant i\leqslant 3)}} {\sum\ \quad \sum\ \quad \sum}g_q (m_1m_2m_3; a)\mathfrak Y_0 (m_1m_2m_3)   \ll \frac{D^{C_0}}{\varphi (q)} x^{1-c  (\varepsilon) /\log Y_0}.
\end{multline}
\end{lemma}

\noindent{\it Proof of Lemmas \ref{tau2} and \ref{tau3}.} First consider the case where $D=1$.  The bound \eqref{tau2arprog} is then a classical consequence of Weil's bound for Kloosterman sums. As for \eqref{tau3arprog}, the first result with an exponent $>1/2$ in the  upper bound for $q$ is due to  Friedlander and Iwaniec \cite[th.  5]{FrIwd3} and it appeals to Deligne's deep bounds for multidimensional exponential sums. We use here Heath-Brown's result \cite[th. 1]{H-Bd3}  that any exponent $<21/41$ is admissible. Note that if  $q$ is assumed to be prime, the best exponent to date is $12/23 -\varepsilon$: see  \cite[th. 1]{FoKoMid3}.
\par 
The bounds \eqref{tau2arprogmodif} and \eqref{tau3arprogmodif} are variants of \eqref{tau2arprog} and \eqref{tau3arprog} in which the variables $m_i$ are slightly sifted. These extensions, useless if $\log Y_0 \gg \log x$,   are  classically obtained through   the fundamental lemma of sieve theory---see, e.g., \cite[lemma 2*] {BFI2}.
\par 
Let us now consider the case where $D>1$.  The congruences $m_i\md{t_i}D$ may be detected by means of additive characters modulo $D$.  It is standard to incorporate these extra characters  in the proofs of \eqref{tau2arprog} and \eqref{tau3arprog}, which are based on  the study of the oscillations of additive characters modulo $q$. The proofs are identical  with no loss up to the factor $D^{C_0}$. This factor turns out to be harmless in  our applications since $D$ will be a fixed power of $\log x$---see
  the comments after Corollary  \ref{easy2}.
\qed
\subsection{Lemmas from complex analysis.} The following stimate will be useful to deal with  some main terms appearing in \S\thinspace\ref{Dirichlet}. We use the following notations
\begin{equation}
\label{not-hlg}
\begin{aligned}
&h:=\prod_p \Bigl( 1  + \frac{1}{p(p-1)}\Bigr),\quad\lambda:= \gamma -\sum_{p} \frac{\log p}{1 +p(p-1)},\\
&g(n) := \prod_{p\vert n} \frac {1}{ 1+ p/(p-1)^2}, \quad  \vartheta (n) := \sum_{p\mid n} \frac{p^2 \log p}{(p-1)(p^2-p+1)}.
\end{aligned}
\end{equation}
\begin{lemma}
Uniformly for $n\in\N^*$ and  $R\geqslant 1$, we have 
\begin{equation}\label{T(R,n)}
T(R,n) := \sum_{\substack{r\leqslant R \\  (r,n) =1}} \frac {1}{\varphi (r)} = h g(n)
\bigl\{\log R + \lambda  + \vartheta (n) \bigr\} +O \bigg(\frac{b_{1/2} (n)}{R^{2/7}}\bigg). 
\end{equation} 
\end{lemma}

\begin{proof} For $\Re s >0$, consider the Dirichlet series
$$
\sum_{\substack{r\geqslant 1\\  (r,n)=1}} \frac{1}{\varphi (r)r^s} =\prod_{p \,\nmid\, n} \Bigl( 1 +  \frac{1}{(1-1/p)} \sum_{\nu \geqslant 1} \frac {1}{p^{\nu(s+1)}}
\Bigr)
= \zeta (s+1) H(s) G(n,s),
$$
with
\begin{align*}
&H(s) := \prod_p \Bigl( 1+ \frac{1}{p^{s+1} (p-1)}\Bigr),\\
& G(n,s):= \prod_{p\,\mid\,n} \frac{1-1/p^{s+1}}{1+1/  \{p^{s+1}(p-1)\}}=\prod_{p\,\mid\,n}\frac1{1+p/\{(p-1)(p^{s+1}-1)\}}\cdot
\end{align*}
Apply  Perron's formula in  effective form (see, e.g.,  \cite[cor. II.2.4, p.~220]{TenlivreUS}) and move the line of integration to the abscissa $\sigma =-\frac 12$. Since, we have, uniformly for $\sigma \geqslant -\frac 12$,
$$
G(n,s) \ll b_{1/2} (n),$$
we obtain, uniformly for $R \geqslant 1$ and $n\geqslant 1$, 
\begin{align*}
T(R,n)&=  \mathrm {Res} \bigl( R^s \zeta (s+1) H(s) G(n,s) /s;0\bigr) +O \Big(\frac{b_{1/2} (n)}{R^{2/7}}\Big)\\
& = H(0) G(n,0) \{ \log  R + \gamma\} +H'(0) G(n,0) + H(0) G'(n,0) + O \Big(\frac{b_{1/2} (n)}{R^{2/7}}\Big),\\
\end{align*}
which coincides with \eqref{T(R,n)}. 
\end{proof}

The above lemma may be exploited to evaluate the more general sum 
$$
T(R,m, n) := \sum_{\substack{r\leqslant R \\  (r,n) =1}} \frac {1	}{\varphi (mr)} \cdot$$
We retain notation \eqref{not-hlg} and further introduce, for $j=0,1$,
integers $m,n\geqslant 1$, and real $u\geqslant 1$, 
\begin{equation}
\label{not-Theta01}
\Theta_j(m,n;u):=\sum_{\substack{\delta \leqslant  u\\  \delta \vert m^\infty,\, (\delta,n)=1}} \frac{ (\log \delta)^j}{\delta} 
\end{equation}
\begin{lemma}\label{sumofphimodif} Uniformly for  $R\geqslant R_0   \geqslant 1$,  and integers $m,n\geqslant 1$,  we have  
\begin{equation}
\label{faTRmn}
\begin{aligned}T(R,m,n)=&\frac{ h g(mn)}{\varphi (m)}\bigg(
\bigl\{\log R + \lambda  + \vartheta (mn) \bigr\}  \Theta_0(m,n;R_0)-\Theta_1(m,n;R_0) \bigg)\\
& 
  +O \bigg( \frac{\tau (m n)^2 \log 2R}{\varphi (m) R_0^{1/4}}\bigg).    
\end{aligned}
 \end{equation}
\end{lemma}
\begin{proof}  Split the sum $T(R,m,n)$ according to the value of $\delta:= (r, m^\infty)$ and write $r = \delta s$. Since $(s, m\delta )=1$ we  get, with notation \eqref{T(R,n)},
\begin{align*}
T(R,m,n) &= \sum_{\substack{\delta\vert m^\infty\\  (\delta,n) =1}} \frac{1}{\varphi (m \delta)} \sum_{\substack{s\leqslant R /\delta\\  (s, mn)=1}} \frac{1}{\varphi (s)}= \sum_{\substack{\delta\vert m^\infty\\  (\delta,n)=1}} \frac{T(R/\delta ,mn)}{\varphi (m \delta)}.
\end{align*}
 To shorten this summation we use,  for $\delta > R_0$,  the trivial bound $T(R/\delta,mn)\ll \log 2R$.  Since $\varphi (m\delta) = \delta \varphi (m)$, Rankin's method eventually yields 
\begin{equation}\label{627}
T(R,m,n)= \sum_{\substack{\delta <R_0 \\  \delta \mid m^\infty,\,(\delta,n) =1}} \frac{T(R/\delta, mn) }{\delta\varphi (m)}  +O \bigg( \frac{\log 2R}{\varphi (m)} \sum_{\delta \mid m^\infty } \frac {1}{\delta} \Bigl( \frac {\delta}{R_0}\Bigr) ^{1/4}\bigg).
 \end{equation}
   Inserting \eqref{T(R,n)} into \eqref{627}, we obtain a formula for $T(R,m,n)$ with the stated main term and error term
 \begin{equation*}
   \ll \frac{b_{1/2} (mn) b_{5/7} (m)}{\varphi (m)R_0^{2/7}}  + \frac{b_{3/4} (m)\log 2R}{\varphi (m) R_0^{1/4}}\cdot 
  \end{equation*}
It can be checked that the order of magnitude of this expression does not exceed that of the error term appearing in \eqref{faTRmn}.
\end{proof}

\section {Proof of Theorem \ref{central} with the restrictions $Q^2 R\leqslant x$ and $(qr,D)=1$}

\subsection{First step of preparation}\label{preparation}
When $(qr, D)=1$, we have $\Delta_f (x;qr,D,a)=\Delta_f (x;qr,a)$ by \eqref{Delta=Delta}. 
 The purpose of this section is to establish the following statement. 
 \begin{proposition}\label{firststep}  Let  $K>0$.   For suitable absolute constant $C_0$ and all $\varepsilon$, $A>0$,
there exists   $C= C(\varepsilon, A )$  such that, uniformly for
\begin{equation}
\label{cond1}
\begin{aligned}
  &D\geqslant 1,\quad f\in \FF (D,K), \quad x\geqslant 1,\quad Q\geqslant 1,\quad 1\leqslant R\leqslant x^{1/105-\varepsilon},\\
  & Q^2R \leqslant x,\quad  (a,D)=1, \quad 1\leqslant \vert a \vert  \leqslant \LL^A, 
  \end{aligned}
\end{equation}
we have
\begin{equation}\label{244bis}
\sum_{\substack{r\leqslant R\\  (r,a D)=1}}\Bigl\vert \sum_{\substack{q\leqslant Q \\  (q,a D)=1}} \, \Delta_f (x;qr,a ) \Bigr\vert \leqslant  \frac{C\,D^{C_0}  x}{\LL ^A}\cdot
\end{equation}
The same  bound also holds, uniformly for integers $b$ and $c$   with $1\leqslant b, \, c \leqslant \LL^A,$ on replacing $f\in \FF (D,K)$ by $f_{b,c}$, as defined in \eqref{deffbb}. Under the   assumptions \eqref{cond1}, we therefore have
\begin{equation}\label{244bismodif}
\sum_{\substack{r\leqslant R\\  (r,a D)=1}}\Bigl\vert \sum_{\substack{q\leqslant Q \\  (q,aD )=1}}  \Delta_{f _{b,c}}(x;qr,a ) \Bigr\vert \leqslant  \frac{C\,D^{C_0}  x}{\LL ^A}\cdot
\end{equation}

 \end{proposition}

 In order to prove \eqref{244bis} under the assumptions \eqref{cond1} we first perform   a dyadic decomposition and  define accordingly
\begin{align}
V(Q,R) &:=\sum_{\substack{r\leqslant R\\  (r,aD)=1}} \Bigr\vert  \sum_{\substack{q\leqslant Q \\  (q,aD)=1}}  \bigl(\Delta_f (x;qr,a)-\Delta_f (x/2;qr,a)\bigr)\Bigr\vert \nonumber \\
& = \sum_{\substack{s\leqslant S \\  (s,aD)=1}} c_s\Biggl( \, \sum_{\substack{n\sim x/2 \\  n\equiv a \bmod s}} f(n)-\frac{1}{\varphi (s)} \sum_{\substack{n\sim x/2 \\  (n,s)=1}} f(n)
\, \Biggr)\label{308}
\end{align}
with 
\begin{equation}\label{defScs}
S:=QR \text{ and } c_s:= \sum_{\substack{r\leqslant R, \, q\leqslant Q\\  s =qr} }\xi_r,
\end{equation}
where $\xi_r$ is some coefficient satisfying $\vert \xi_r\vert \leqslant 1.$ Note the bounds $S \ll x^{53/105} $ and $\vert c_s \vert \leqslant \tau (s)$. Thus, the proof of Proposition \ref{firststep} is reduced to that of the estimate
\begin{equation}\label{S(QR)<<}
V(Q,R) \ll D^{C_0}x/\LL^{A},
\end{equation} 
 for both functions $f$ and $f_{b,c}$.

Let us now fix 
\begin{equation*}
Y_0:= \exp\big(\LL^{1/4}\big),
\end{equation*}
and recall from \S\thinspace \ref{conventions} the definition of the associated indicator function $\mathfrak Y_0$.
We factorize integers  $n\in]x/2,x]$ uniquely as
\begin{equation}\label{decomp}
n = \nu_n \prod_{1\leqslant j\leqslant J_n}p_j^{\alpha_j},
\end{equation}
with
\begin{equation*}
\nu_n:= \prod_{\substack{p\leqslant Y_0\\ 
p^\ell \Vert n}} p^\ell,\quad Y_0< p_1< p_2 < \cdots<p_{J_n},\quad\alpha_j\geqslant 1\  (J_n\geqslant 0,\,1\leqslant j\leqslant J_n).
\end{equation*}
Since $f$ is multiplicative, we have
\begin{equation}\label{f(n)=}
f(n) =f(\nu_n) \prod_{1\leqslant j\leqslant J_n}f({p_j}^{\alpha_j}) 
\end{equation}
and also
$$
f_{b,c} (n) =
\begin{cases}
 f(b\nu_n) \prod_{1\leqslant j\leqslant J_n}f({p_j}^{\alpha_j}) & \text{ if } (\nu_n, c)=1,\\
 0 &\text{ if } (\nu_n, c) >1,
\end{cases}
$$
since we may assume $1 \leqslant b,\, c \leqslant \LL^A$ and $x$ sufficiently large.

\subsection{Contribution of  non typical variables $n$\label{nontypical}} In that subsection, we will not appeal to the combinatorial structure of the coefficients $c_s$.  Let $\EE_0$ be the set of those integers $n\in]x/2,x]$ such that, with notation \eqref{decomp}, $$
\max_{1\leqslant j\leqslant J}\alpha_j\geqslant 2 \text{ or }\nu_n>Z_0:= \exp \big(\LL^{3/4}\big).
$$
From the trivial estimate $\sum_{p>z}1/p^2\ll1/z$ and Lemma \ref{gerald} we infer that
$$
\vert \EE_0\vert  \ll x \exp\bigl( -\LL^{1/4}\bigr).
$$
 Combined with \eqref{G<},  this bound implies that the contribution from integers in $\EE_0$ to the  left--hand side of \eqref{308}   is 
 \begin{equation}\label{morning0}
   \ll x \exp \bigl( -\tfrac14 \LL^{1/4}\bigr) + \tau_K (\vert a \vert) S \LL^\idc \     \ll  x \exp \bigl( -\tfrac14 \  \LL^{1/4}\bigr).
 \end{equation}
  \par 
We now introduce the dissection parameter
\begin{equation}\label{defrho}
\varrho := 1 + 1/\LL^{B_0},
\end{equation}
 where $B_0$ will be   specified later in terms of $K$ and $A$, and define
$$
Y_k:= Y_0\varrho^k\quad(k=0, 1, 2,...).
$$ 
Let $\EE_1$ be the set of those integers $n\in]x/2,x]\smallsetminus\EE_0$ such that $\nu_n\leqslant Z_0$ and  
$$
Y_k <  p_1< p_2\leqslant  Y_{k+1}.
$$
for some    $k\geqslant 0$, so   that the two smallest prime factors of $n/\nu_n$  are close to each other. We plainly have
$$
\vert \EE_1 \vert \leqslant  \underset{Y_0 <p_1<p_2 \leqslant \varrho  p_1  }{\sum} \frac{x}{p_1p_2} \ll   
  \sum_{p_1>Y_0 } \frac{x\log \varrho}{p_1\log p_1}\ll \frac x{\LL^{B_0}}\cdot $$
  From \eqref{G<},   we deduce that the contribution  to the left--hand side of  \eqref{308} arising from integers in $\EE_1$ is
\begin{equation}\label{morning1}
\ll \frac x{\LL^{B_0/3 -\idc}} + \tau_K (\vert a \vert) S \mathcal \LL^\idc \ll \frac x{\LL^{B_0/3 -\idc}}\cdot
\end{equation} 
\par 
Taking \eqref{308}, \eqref{morning0} and \eqref{morning1} into account, we can write
\begin{equation}\label{morning3}
V(Q,R)=\sum_{(t,D)=1} \  \sum_{k\geqslant 0}  \ \sum_{\ell \geqslant 1}V_{k,\ell}(Q,R;t) +O \bigg(\frac x{\LL^{B_0/3 -\idc}}\bigg),
\end{equation}
where $V_{k,\ell}(Q,R;t) $ is the subsum of $V(Q,R)$ corresponding to the supplementary conditions
\begin{equation}
\label{sky}
\begin{aligned}
&\alpha_j=1\ (1\leqslant j\leqslant J_n),\quad\nu_n\leqslant Z_0,\\
  &p_1\equiv t(\bmod D),\quad p_1 \in \II_{k,\ell}:=]Y_k,Y_{k+1}]\cap]\Upsilon_\ell,\Upsilon_{\ell+1}],\quad p_2>Y_{k+1}, 
\end{aligned}
\end{equation}
with notations \eqref{decomp}, and where the $\Upsilon_\ell$  appear in  Definition \ref{defM}.
The interval $\II_{k,\ell}$ may be empty, however the  number of sums $ V_{k,\ell}(Q,R;t)$ appearing in \eqref{morning3} is $\ll  D \LL^{B_0+K +2}$. We also observe that selecting $B_0= 3A +3 \idc $ implies that the error term in \eqref{morning3} is $\ll x/ \LL^{A}$,
sharper than required in \eqref{S(QR)<<}. 
\par 
\goodbreak
From the above remarks, we see  that   \eqref{S(QR)<<} follows from showing that, for suitable absolute $C_0$ and all $A>0$, we have 
\begin{equation}\label{purpose}
V_{k,\ell}(Q,R;t) \ll D^{C_0}x/\LL^{A},  
\end{equation}
uniformly for $D \geqslant 1$,   $(t,D)=1$, $k\geqslant 0$, $\ell \geqslant 1$,    for both  $f$ and $f_{b,c}$ whenever $1 \leqslant b, \, c \leqslant \LL^A$. 
\par 
Proving \eqref{purpose} is  the purpose of the next subsection, where we will use, in a crucial way,  the fact that, if $p$ belongs to $\mathfrak I_{k, \ell}$ and satisfies $p\equiv t\, (\bmod D)$, then $f(p)$
assume a constant value,  noted as    $z_t$. Recall that, by \eqref{f(p)leq}, we have $\vert z_t \vert \leqslant K$. 
Regarding the importance of the periodicity of $f|_\PP$, see Remark \ref{fperp}.
 
\subsection{Estimating $V_{k,\ell}(Q,R;t)$}
We consider two subcases according to the size of~$Y_{k}$.

\subsubsection{The case $Y_{k} \leqslant x^{2/7}$} \label{<2/7}  
In order to  apply Lemma \ref{lemma6} to  each sum $V_{k,\ell}(Q,R;t)$ we define 

$$\beta_n
:=\begin{cases}  z_{t} \text { if } n\in\PP,\ n\equiv t\, (\bmod D),\  n\in\mathfrak I_{k,\ell}, 
\\
0 \text{ otherwise}.
\end{cases}
$$  
Note that 
the sequence $(\beta_n)$  satisfies the condition  ${\rm  SW}(D,1)$.
Next, we define
\begin{equation*}
\alpha_m :=
\begin{cases}f(m) \text{ if } m=\idm h,\, P^+(\idm)\leqslant Y_0,\, P^-(h)>Y_{k+1},\, \idm\leqslant Z_0,\,\mu(h)^2=1,\\
0 \text{ otherwise}.
\end{cases}
\end{equation*}
  With these definitions we can rewrite $V_{k,\ell}(Q,R;t) $ as  
$$
V_{k,\ell}(Q,R;t) =  \sum_{\substack{s\leqslant S \\  (s,aD)=1}} c_s\Bigg( \, \underset  {  {\substack{mn\sim x/2 \\     mn\equiv a \bmod s }} }
{\sum }\alpha_m \beta_n-\frac{1}{\varphi (s)} 
\underset  {  \substack{mn\sim x/2 \\    (mn,s)=1} }
{\sum }\alpha_m \beta_n
\, \Bigg).
$$
We  now appeal to  the combinatorial structure of the coefficient $c_s$---see \eqref{defScs}---and   apply Lemma \ref{lemma6} with $N=Y_{k}\geqslant Y_0$, $ M\asymp x/Y_{k}$, $u=x/2$, $v=x$.  This furnishes the bound
\begin{equation}\label{Ykpetit}
V_{k,\ell}(Q,R;t)\ll x/ \LL^{A},
\end{equation}
for any $A$.
\par 
Extending the validity of this bound to $f_{b,c}$ is straightforward: it suffices to replace $f(m)$ by $f_{b,c}( m)$ in the definition of $\alpha_m$.  By \eqref{Ykpetit} we see that \eqref{purpose} holds (with $C_0 =0$) provided $Y_k \leqslant x^{2/7}.$ Thus, it remains to deal with the case when $Y_k$ is large.  

\subsubsection{The case $Y_{k} >x^{2/7}$}\label{>2/7}

 As a direct  consequence of the inequalities  $Y_{k}^4 >x$ and $x/2< n\leqslant x$,  we see that any $n$ contributing to $V_{k,\ell}(Q,R;t)$ may be represented in  one of the following three ways
$$
   n=\nu_n p_1,\quad n=\nu_n p_1 p_2,\quad n=\nu_n p_1p_2p_3,
$$
where $\nu_n$  and the $p_j$   are defined in \eqref{decomp}, and satisfy conditions \eqref{sky}.   The case $n=\nu_n p_1$ is very similar to the case treated in Theorem  \ref{pointdedepart1}. We will restrict to the situation when  $n= \nu_np_1p_2p_3$: indeed the other two cases are similar and actually simpler  from  a combinatorial aspect.
\par 
In order to homogenize the notations in the following computations, we substitute
$$
k\rightarrow k_1, \ell \rightarrow \ell_1, t\rightarrow  t_1.
$$
\par 
\goodbreak
With the above considerations in mind, it is natural to consider the expression
\begin{equation}
\label{defTk1ell1}
W_{k_1,\ell_1}(Q,R;t_1):=
\sum_{\substack{s\leqslant S \\  (s,aD)=1}} c_s\,   \underset{\nu, p_1, p_2,  p_3  }{\sum}\  g_s(\nu p_1p_2p_3,a)f(\nu) f (p_1p_2p_3) ,
\end{equation} 
where the  summation variables  satisfy the conditions  
\begin{equation}
\label{1109}
\begin{aligned}
&x/2 < \nu p_1p_2p_3 < x, \quad \nu  \leqslant  Z_0,\quad P^+(\nu )\leqslant Y_0,\\
& p_1 \in \II_{k_1, \ell_1},\quad p_1\equiv t_1\,(\bmod D),\quad Y_{k_1+1}< p_2 < p_3.
\end{aligned}
\end{equation}
\par 
 The proof of   \eqref{purpose} is hence reduced to showing  that, for a suitable absolute $C_0$ and all $A>0$, we have
\begin{equation}\label{purpose1}
W_{k_1,\ell_1}(Q,R;t_1) \ll_A D^{C_0} x/\LL^{ A}, 
 \end{equation}
 uniformly for $D \geqslant 1$, $(t_1,D)=1$, $k_1\geqslant \log (x^{2/7}/ Y_0)\big/ \log \varrho $, $\ell_1 \geqslant 1$,  for both $f$ and $f_{b,c}$, with $1 \leqslant b, \, c \leqslant \LL^A$.   
\par 
However,  the summation conditions given in \eqref{1109} are not sufficient  to determine  the value of $f(p_1p_2p_3)$ in \eqref{defTk1ell1}. To circumvent this difficulty, we split further
the sum $ W_{k_1,\ell_1}(Q,R;t_1)$ as  
\begin{equation}\label{split10}
W_{k_1,\ell_1}(Q,R;t_1)=\sum_{k_2, k_3}\  \sum_{\ell_2, \ell_3} \ \sum_{t_2, t_3 \bmod D}  W_{\bsk, \bsl} (Q, R ; \bst) +\EE,
\end{equation}
with \par 
 $\bullet$ $ \bsk := (k_1,k_2,k_3)$ satisfies $ k_3 > k_2 >k_1 \, (\geqslant \log (x^{2/7}/ Y_0)\big/ (\log \varrho))$, 
\par 
 $\bullet $ $\bsl :=(\ell_1, \ell_2, \ell_3)$ satisfies $\ell_2,\, \ell_3 \geqslant1$,
\par 
 $\bullet$ $\bst: =(t_1,t_2,t_3)$ satisfies $(t_2t_3, D) =1$,
\par 
 $\bullet $ 
\begin{equation}\label{defTboldkboldell}
W_{\bsk, \bsl} (Q, R ; \bst):=z_{t_1} z_{t_2} z_{t_3}
\sum_{\substack{s\leqslant S \\  (s,aD)=1}} c_s\,   \underset{\nu, p_1, p_2,  p_3\  }{\sum}\  g_s(\nu p_1p_2p_3,a)f(\nu) ,
\end{equation}
where the summation conditions of  \eqref{1109} are replaced by
\begin{equation*} 
x/2 < \nu p_1p_2p_3 < x,\ \nu \leqslant  Z_0, P^+(\nu)\leqslant Y_0,\ p_i \in \mathfrak I_{k_i, \ell_i},\,   p_i\equiv t_i \,(\bmod D)\ (1\leqslant i\leqslant 3),
\end{equation*} 
and where the error term $\EE$ arises from the contribution of those $(p_1,p_2,p_3)$  such that
$Y_{k_2} < p_2 < p_3 \leqslant Y_{k_2 +1}$ for some $k_2 > k_1$. Finally we denote by  $z_{t_i}$   the value of $f(p_i)$ when $p_i$ belongs to $\mathfrak  I_{k_i, \ell_i}$ and $p\md{t_i}D$.\par  
 By a computation similar to \eqref{morning3}, we see that, if $B_0$ is chosen sufficiently large, the error term $\EE$  (see \eqref{split10}) is bounded as required in \eqref{purpose1}.
 \par 
 The number of terms in the multiple sum of \eqref{split10} is $\ll D^2\LL^{2B_0+2K+4}$. Hence \eqref{purpose1} follows  from the validity of the bound \begin{equation}\label{purpose10}
 W_{\bsk, \bsl} (Q, R ; \bst) \ll D^{C_0} x/\LL^{A},
 \end{equation}
for suitable, absolute $C_0$ and all $A>0$, uniformly for  
\begin{equation}
\label{kellt}
\begin{aligned} 
&k_3 > k_2 > {k_1 \geqslant }\log (x^{2/7}/ Y_0)\big/ \log \varrho,\quad \min_{1\leq j \leq 3} \ell_j\geqslant 1, \\ &D\geqslant 1,\quad (t_1t_2t_3, D)=1.
\end{aligned}
\end{equation}
\par 
  It is time to replace, in \eqref{defTboldkboldell},  the characteristic function of the set of primes $\PP$   by    the classical von Mangoldt function $\Lambda$ and even better by the function $\Lambda \mathfrak Y_0$. Since these  techniques  classically generate  an admissible  error, the proof of  \eqref{purpose10}   is reduced  to show that, uniformly under conditions \eqref{kellt}, we have 
 \begin{equation}\label{purpose2}
  \widetilde  W_{\bsk,\boldsymbol\ell}(Q,R;\bst)\ll_A D^{C_0}x/\LL^{A}, 
  \end{equation}
 where 
\begin{equation}\label{defTildeT}
\widetilde W_{\bsk,\boldsymbol\ell}(Q,R;\bst) :=
\sum_{\substack{s\leqslant S \\  (s,aD)=1}} c_s\, \underset{\nu, n_1, n_2,  n_3  }{\sum }  g_s(\nu n_1n_2n_3,a)  G(\nu, n_1, n_2, n_3), 
\end{equation}
with
\begin{equation*}
G(\nu, n_1,n_2,n_3)  := f(\nu){\mathfrak Y_0}(n_1)\Lambda (n_1)  {\mathfrak Y_0}(n_2)\Lambda (n_2)
  {\mathfrak Y_0}(n_3)\Lambda (n_3),
\end{equation*}
and where the summation variables in \eqref{defTildeT}   satisfy  the  conditions  
\begin{equation}
\label{sun0}
\begin{aligned}
&x/2 < \nu n_1n_2n_3 < x,\quad \nu \leqslant  Z_0, \quad P^+(\nu)\leqslant Z_0,\\
&n_i \in \mathfrak I_{k_i, \ell_i},\,   n_i\equiv t_i\, (\bmod D)\ (1\leqslant i\leqslant 3). 
\end{aligned}
\end{equation}
\par   
Conditions \eqref{kellt} and \eqref{sun0}  imply  $x^{2/7} < n_1 < x^{1/3}$ and  $n_3 < x^{3/7}$.
 So we can apply Lemma \ref{H-Bid} to  each of the factors $\Lambda (n_i) $ ($1\leqslant i\leqslant 3$)  with 
 $y  := x^{4/7}$ and $J:= 2$. 
 Thanks to this identity, the summation over each variable $n_i$ ($1\leqslant i \leqslant 3$) in \eqref{defTildeT} is replaced by two summations, respectively  over
\begin{equation}\label{case}
  \text { pairs } (m_{i,1},  n_{i,1}),  \text { and    4-tuples } (m_{i,1}, m_{i,2}, n_{i,1}, n_{i,2}).
\end{equation}
Mixing all these cases leads to considering  eight types of sums. Since the other cases are similar, and actually simpler in the combinatorial aspect, we will concentrate on those sums arising from  the last cases in \eqref{case} for $i=1$, $2$ or $3$.  We therefore   consider the arithmetic function
 \begin{equation}
 \label{difficult}
 g(n) := \sum_\nu f(\nu) 
\underset
{\substack{m_{i,j},n_{i,j}\\ (1\leqslant i\leqslant 3;\,j=1,2) } }
{ \sum}
\prod_{\substack{1\leqslant i \leqslant 3\\ j=1,2}}\mu (m_{i,j})\mathfrak Y_0(m_{i,j}n_{i,j}) \prod_{1\leqslant i\leqslant 3}\log n_{i,1},
 \end{equation}
with the summation conditions  
 \begin{equation}\label{ineq1}
 \begin{cases}
 n =\nu  \prod_{1\leqslant i \leqslant 3} \prod_{1\leqslant j\leqslant 2} m_{i,j} n_{i,j}, \\
 \nu \leqslant Z_0, P^+(\nu)\leqslant Y_0,\\ 
  m_{i,1} m_{i,2} n_{i,1} n_{i,2} \in \mathfrak I_{k_i, \ell_i}\ (1\leqslant i\leqslant 3) ,\\
 m_{i,1},\, m_{i,2}\leqslant x^{2/7}\ (1\leqslant i\leqslant 3),\\
 m_{i,1} m_{i,2} n_{i,1} n_{i,2}\equiv t_i\,( \bmod D)\ (1\leqslant i\leqslant 3).
 \end{cases}
 \end{equation}
With this definition, we are led to consider the typical sum
\begin{equation}\label{defG(QR)}
G_{\bsk, \bsl}(Q,R;\bst) := \sum_{\substack{s\leqslant S\\  (s,aD) =1}} c_s \Bigl( \sum_{\substack{n\sim x \\  n \equiv a \bmod s}} g(n) -\frac{1}{\varphi (s)} \sum_{\substack{n \sim x\\  (n,s)=1} }g(n) \Bigr).
\end{equation}
Indeed, \eqref{purpose2} will follow from the validity of
\begin{equation}\label{purpose20} G_{\bsk, \bsl}(Q,R; \bst) \ll D^{C_0} x/\LL^{A},
\end{equation}  
for suitable absolute $C_0$, all $A>0$, and uniformly under  conditions \eqref{kellt}.
\par

The sum $G_{\bsk, \bsl}(Q,R;\bst)$ defined in \eqref{defG(QR)} is over fourteen variables, namely $s$, $\nu$, the $m_{i,j}$ and the $n_{i,j}$. In order to make the last twelve   variables arithmetically independent, 
we fix the reduced congruence class modulo $D$ of each $m_{i,j}$ and $n_{i,j}$. This involves splitting sum  $ G_{\bsk, \bsl} (Q,R; \bst)$
into $\ll D^{12}$  subsums where the last condition in \eqref{ineq1} is replaced by twelve conditions of the shape $m_{i,j} \equiv t_{i,j}\,(\bmod D)$ and $n_{i,j}\equiv t'_{i,j} $ where the $t_{i,j}$,  $t'_{i,j}$ are reduced classes 
modulo $D$. For notational simplicity  we will not recall these conditions in the sequel of the proof.
\par 
The presence of the factor involving $\mathfrak Y_0$ in \eqref{difficult} implies that each variable $m_{i,j}$, $n_{i,j}$ is  either $1$ or  $\geqslant Y_0$. Therefore, $G_{\bsk, \bsl}(Q,R; \bst)$ contains
subsums which can be handled by Lemma \ref{lemma6} as was performed in \S\thinspace\ref{<2/7}. More precisely, let 
$G^{(1,1)}_{\bsk, \bsl}(Q,R;\bst)$ denote the subsum of $G_{\bsk, \bsl}(Q,R; \bst)$ corresponding to the extra condition $m_{1,1}>1$, which implies $Y_0 \leqslant  m_{1,1} \leqslant x^{2/7}$. We may then apply Lemma \ref{lemma6} to the variables  $n:=m_{1,1}$,  $m:=\nu(\prod_{(i,j)\neq(1,1)}m_{i,j})\,(\prod n_{i,j})$ , in \eqref{defDab} with
 $$\beta_n=
 \begin{cases} \mu (n)\mathfrak Y_0 (n) &\text{ if }n\md{t_{1,1}}D,\\
 0& \text{ otherwise},
 \end{cases}
 $$    
  the definition   of   $\alpha_m$ being then  obvious. 
 By  \eqref{SWformuY0}, the sequence $\beta_n$ satisfies $\SW(D,K).$  Lemma \ref{lemma6} hence provides  the bound
\begin{equation}\label{m11}
G_{\bsk, \bsl}^{(1,1)}(Q,R; \bst) \ll x/\LL^{A}.
\end{equation}
\par 
Let us next consider the subsum
 $G_{\bsk, \bsl}^{(1,2)}(Q,R; \bst)$ corresponding to the extra hypothesis $m_{1,1}=1$, $m_{1,2} >1$, which similarly  implies $Y_0\leqslant m_{1,2}\leqslant  x^{2/7}$. We may again apply Lemma \ref{lemma6} to deduce 
 \begin{equation}\label{m12}
G_{\bsk, \bsl}^{(1,2)}(Q,R; \bst ) \ll x/\LL^{A}.
\end{equation}
\par 
Continuing this process on each of the variables $m_{i,j}$ yields upper bounds similar to \eqref{m11} and \eqref{m12}.  Having dealt with these easy subsums, we reduce the proof
of \eqref{purpose20} to that of the bound
\begin{equation}\label{purpose3}
G_{\bsk, \bsl}^*(Q,R; \bst) \ll D^{C_0} x/\LL^{A},
\end{equation}
where $G_{\bsk,\bsl}^*(Q,R;\bst)$ is the subsum of $G_{\bsk, \bsl}(Q,R, \bst)$ corresponding to the extra condition $m_{i,j} =1$ $(1\leqslant i\leqslant 3,\,j=1,2)$. 
\par 
The sum $G_{\bsk,\bsl}^*(Q,R;\bst)$
is then over eight variables, namely $s$, $\nu \leqslant Z_0$ and the $n_{i,j}$ which are equal to $1$ or $\geqslant Y_0$. By \eqref{SWforY0},
we know that, whenever $(t,D)=1$, the functions
$$
n\mapsto \beta_n  =
\begin{cases} \mathfrak Y_0 (n)(\log n)^j \text{ if } n\md tD,\\
0 \text{ otherwise,}
\end{cases}
$$  
satisfy $\SW (D,K)$ for $j=0$ or 1. 
Hence  Lemma \ref{lemma6} ensures that the bound \eqref{purpose3} holds for the subsum  of $G_{\bsk,\bsl}^*(Q,R; \bst)$ corresponding  to the case when at least one of the variables $n_{i,j}$ ($1\leqslant j\leqslant 3, \, 1\leqslant i \leqslant 2$) lies in the interval  $Y_0\leqslant n_{i,j} \leqslant x^{2/7}.$   \par 
Thus, we can state that the proof of \eqref{purpose3} is reduced to showing that, for suitable absolute $C_0$ and all $A>0$, we have
\begin{equation}\label{purpose4}
 G_{\bsk, \bsl}^\dagger(Q,R;\bst) \ll D^{C_0} x/\LL^{A},
\end{equation}
where $ G_{\bsk, \bsl}^\dagger(Q,R;\bst) $ is the subsum of $ G_{\bsk,\bsl}^*(Q,R;\bst) $ in which the $n_{i,j}$    satisfy the extra conditions 
$$
n_{i,j} =1 \text{ or } n_{i,j} >x^{2/7} \qquad (1\leqslant i\leqslant 3,\,j=1,2). $$
Since all the the $m_{i,j}$ are equal to $1$, we also have 
$$
\tfrac12x < \nu \prod_{1\leqslant i\leqslant 3} \ \prod_{1\leqslant j\leqslant 2} n_{i,j} \leqslant  x,
$$
hence the number of variables $n_{i,j}$  exceeding $x^{2/7}$ lies between 1 and 3, the others being equal to $1$. 
\par 
The  large variables $n_{i,j}$ are almost smooth, since the function $\mathfrak Y_0$ only involves a mild sifting. Therefore, we may apply Lemma \ref{tau1} provided only one large variable is involved, Lemma \ref{tau2} when two are, and  Lemma \ref{tau3} when three are. 
These lemmas substantiate respectively the equidistribution of the sequences
$$
g_s (\ell_1) \mathfrak Y_0 (\ell_1), \quad g_s (\ell_1\ell_2) \mathfrak Y_0 (\ell_1\ell_2),\quad g_s (\ell_1\ell_2 \ell_3) \mathfrak Y_0 (\ell_1\ell_2 \ell_3),
$$
(where the integers $\ell_i$ belong to some intervals and satisfy congruence conditions modulo $D$)  in every congruence class $a\!\!\pmod s$, with $(s,aD)=1$, uniformly in the respective range $$s \leq x^{1-\varepsilon}, \quad s \leq x^{2/3-\varepsilon},\quad s\leq x^{21/41-\varepsilon}$$ with error terms of the shape $\ll D^{C_0}\e^{-c(\varepsilon)\LL^{3/4}}x/\varphi (q)$. Summing over $s\leq S$  (note the 
inequalities  $1>2/3>21/41>53/105$) and noticing that 
  when actually present, the factor $\log n_{i,j}$ may be
treated by partial summation,  completes the proof
of \eqref{purpose4}.  This terminates the proof of Proposition \ref{firststep} for the case of the function $f$.
\par 
The extension to the function $f_{b,c}$ is straightforward on replacing  $f(m)$ by $f_{b,c}(m)$  in the   definition  of $\alpha_m$  and 
$f(\nu)$ by $f_{b,c} (\nu)$ in \S\thinspace\ref{>2/7}. 
This yields  \eqref{244bismodif}.


  \section {Proof of Theorem \ref{central} with the sole restriction  $Q^2 R\leqslant x$}
This section is devoted to deducing from Proposition~\ref{firststep} the following statement.
 \begin{proposition}\label{secondstep}  Let  $K>0$.   For a suitable absolute constant $C_0$ and all $A>0$,  $\varepsilon>0$,
there exists   $C= C(\varepsilon, A )$  such that, uniformly for
\begin{equation*}
\begin{aligned}
  &D\geqslant 1,\quad f\in \FF (D,K), \quad x\geqslant 1,\quad Q\geqslant 1,\quad 1\leqslant R\leqslant x^{1/105-\varepsilon},\\
  & Q^2R \leqslant x,\quad  (a,D)=1, \quad 1\leqslant \vert a \vert  \leqslant \LL^A, 
  \end{aligned}
\end{equation*}
we have
\begin{equation}\label{244bisD}
\sum_{\substack{r\leqslant R\\  (r,a )=1}}\Bigg\vert \sum_{\substack{q\leqslant Q \\  (q,a )=1}}  \Delta_f (x;qr,D,a ) \Bigg\vert \leqslant \frac{ C \, D^{C_0}x}{\LL ^A}\cdot
\end{equation}
Under the same hypotheses, the same  bound  holds uniformly for integers $b,c$,  with $1\leqslant b, \, c \leqslant \LL^A,$ on replacing $f\in \FF (D,K)$ by $f_{b,c}$, as defined in \eqref{deffbb}, viz.
\begin{equation}\label{244bismodif*}
\sum_{\substack{r\leqslant R\\  (r,a  )=1}}\Bigg\vert \sum_{\substack{q\leqslant Q \\  (q,a  )=1}} \, \Delta_{f _{b,c}}(x;qr,D,a ) \Bigg\vert \leqslant \frac{ C \, D^{C_0}x}{\LL ^A}.
\end{equation}

 \end{proposition}
 \begin{proof} Let $S$ and $c_s$ be defined  as in \eqref{defScs}. The sum studied in \eqref{244bisD} may be written as
 $$
 V(Q,R;D):= \sum_{\substack {s\leqslant S \\  (s,a)=1}} c_s\, \Delta_f (x;s,D,a ).
 $$
According to \eqref{factorizeq}, we factorize $s$ as
 $$
 s = s_D s'_D, \text{ with } s_D = (s, D^\infty).
 $$
Splitting the sum $V(Q,R;D)$ according to the value of $s_D$, we get
 $$
 V(Q,R;D) = \sum_{t \mid D^\infty} \sum_{\substack{\sigma\leqslant S/t \\  (\sigma, aD)=1}} c_{t\sigma }\, \Delta_f (x;t\sigma ,D,a ).
 $$
 The contribution of  large $t$ is estimated by Lemma~\ref{lemma1}: for  $T>1$ we have
 \begin{align*}
\gR^+(T)&:= \sum_{\substack{t \mid D^\infty\\  t >T}} \sum_{\substack{\sigma\leqslant S/t \\  (\sigma, aD)=1}} c_{t\sigma }\, \Delta_f (x;\sigma t,D,a )
\ll \sum_{ \substack{t \mid D^\infty\\  t >T}} \sum_{\substack{\sigma\leqslant S/t \\  (\sigma, aD)=1}} \frac{\tau(t \sigma) x\LL^{\idc}}{\varphi (t\sigma )} \\
 & \ll x\LL^{\idc {+2} } \sum_{\substack{t \mid D^\infty\\  T<t\leqslant S}} \frac{\tau (t)}{\varphi (t)} \ll x\LL^{\idc {+2} } \   \sum_{ t \mid D^\infty } \frac{\tau (t)}{\varphi (t)} \Bigl( \frac{t}{T}\Bigr)^{1/4}\ll \frac{x\LL^{\idc {+2}}b_{3/4}(D)^2}{T^{1/4}}\cdot
 \end{align*}
 It remains to select $T:=\LL^C$ with suitable $C= C(A)$ to obtain the bound
\begin{equation}\label{sigma>Xi}
\gR^+(T)\ll D x/\LL ^{A}.
\end{equation}
\par 
  We next turn our attention to the complementary sum 
 $$
\gR^-(T):=   \sum_{\substack{t \mid D^\infty\\  t \leqslant T}} \sum_{\substack{\sigma\leqslant S/t \\  (\sigma, aD)=1}} c_{t\sigma } \Delta_f (x;t\sigma ,D,a ).
  $$
  We now introduce Dirichlet characters modulo $t$---see \eqref{Dirichletcharacter}---to infer 
    \begin{equation}\label{Sigma1leq}
|\gR^-(T)|
 \leqslant  \sum_{\substack{t \mid D^\infty\\  t \leqslant T}}  {\frac1{\varphi(t)}} \sum_{\chi\!\pmod t}\ 
 \Bigl\vert \,  \sum_{\substack{\sigma\leqslant S/t \\  (\sigma, aD)=1}} c_{t\sigma }\, \Delta_{f \chi}(x; \sigma,a )
 \, \Bigr\vert.
 \end{equation}
 Since  $ f(p) \chi (p) $ is  periodic modulo  $ Dt$, we have $f \chi\in\FF  ( Dt, K)$.  
In order to apply   Proposition \ref{firststep} we must also check that the weight $\sigma\mapsto c_{t\sigma}$ can be suitably factorized.
However, since $(\sigma, t)=1$, the equality $qr= t\sigma $ implies a unique representation
$$
q=q_t q_\sigma, \ r=r_t r_\sigma\text{ with } q_tr_t =t \text{ and } q_\sigma r_\sigma = \sigma.
$$
Taking into account that $(\sigma, aD)=1\Leftrightarrow(\sigma, a Dt)=1$, we get that the absolute value of the inner sum in~\eqref{Sigma1leq} does not exceed
\begin{equation*}
\sum_{q_t r_t = t}  \sum_{\substack{r_\sigma\leq  R/r_t\\  (r_\sigma, a Dt)=1}} \Biggl\vert \, 
\sum_{\substack{q_\sigma\leqslant Q/q_t \\  (q_\sigma, aDt)=1}} \Delta_{f\chi} (x;q_\sigma r_\sigma, a)\,
\Biggr\vert\ll  \underset{q_t r_t = t} {\sum}  \frac{(Dt)^{C_0}x}{\LL^B}
\end{equation*}
by Proposition \ref{firststep}, where $B$ is arbitrary. 
Inserting back into  \eqref{Sigma1leq}  we obtain
 \begin{align*}
\gR^-(T) & \ll  \sum_{\substack{t \mid D^\infty\\  t \leqslant T}}  \sum_{\chi \!\pmod t}  \sum_{q_t r_t = t} \frac{(Dt)^{C_0} x}{ \LL^B\varphi(t)}\ll  \frac{D^{C_0} xT^{ C_0+1}}{\LL^B}\ll \frac{D^{C_0}x}{\LL^A},
\end{align*}
for a suitable choice of $B$, considering our choice for $T$.  
Combined with \eqref{sigma>Xi} this bound furnishes \eqref{244bisD}.
\par 
Extending the above proof  to obtain   \eqref{244bismodif*} is now standard and we omit the details. 
 \end{proof}

\section{Application of Dirichlet's hyperbola method}\label{Dirichlet} In this section, we aim at completing the proof of Theorem~\ref{central} from Proposition~\ref{secondstep}.  We may plainly assume that 
\begin{equation}\label{Q0<Q<xR-1}
Q_0:=\sqrt{x/R} \leqslant Q\leqslant x/R\LL^{B}, \quad R\leqslant  x^{1/105-\varepsilon}
\end{equation}
where $B=B(A,\varepsilon)$ has to be determined.  Given  $\bsxi=(\xi_r)_{r\geqslant 1}\in\CC^{\N^*}$  such that $\sup_r\vert \xi_r \vert \leqslant 1$, we introduce the quantity 
\begin{equation*}
\HH_D(Q_0, Q,R; \bsxi):=\sum_{\substack{r\sim R\\  (r,a)=1}}\xi_r\sum_{\substack{Q_0< q\leqslant Q \\  (q,a)=1} } \Delta_f (x;qr,D,a) .
\end{equation*}
 As a consequence of Proposition \ref{secondstep}, it remains to prove that, for a suitable absolute constant $C_0$, and all $A>0$, $\varepsilon>0$, there exist  $B=B(A,\varepsilon)$ and $C=C(A,\varepsilon)$ 
 such that,  the estimate
\begin{equation}\label{Sigma<<}
\vert \HH_D(Q_0, Q,R; \bsxi) \vert \leqslant C  D^{C_0}x/\LL^{A},
\end{equation} 
holds uniformly for $Q_0, \, Q$ and $R$ satisfying  \eqref{Q0<Q<xR-1}, $1\leqslant  \vert a \vert \leqslant \LL^A$,  and $\bsxi$ as above.
\par 
\goodbreak
From \eqref{350}, we may split $\HH_D(Q_0, Q,R; \bsxi) $ as
\begin{equation}\label{Sigma=S-S}
\HH_D(Q_0, Q,R; \bsxi)= \gS-\gE,
\end{equation}
where 
\begin{equation}\label{deffrakS}
\mathfrak S := \sum_{\substack{r\sim R\\  (r,a)=1}}\xi_r \,\sum_{\substack{Q_0< q\leqslant Q \\  (q,a)=1}} \sum_{\substack{n\leqslant x\\n\md a{qr} }} f(n),
\end{equation}
and 
\begin{equation}\label{deffrakS*}
\gE := \sum_{\substack{r\sim R\\  (r,a)=1}}\xi_r \,\sum_{\substack{Q_0< q\leqslant Q \\  (q,a)=1}} \frac{1}{\varphi (q'_Dr'_D)} \sum_{\substack{n\leqslant x\\ n\md a{q_Dr_D}
 \\ (n, q_D'r_D')=1}} f(n)
 \end{equation}
 is the expected main term.
\subsection{Transformation of $\mathfrak S$}
\label{transformofS}
We  tackle the sum $\mathfrak S$ by  Dirichlet's hyperbola method  as follows.
 We express the congruence 
$n \md a{qr}$
as
\begin{equation}\label{congequa}
n = a+uqr,
\end{equation}
and consider this relation as a congruence condition modulo $ur$, which is convenient since $u\leqslant  2Q_0$ by \eqref{Q0<Q<xR-1}. However  the condition $(a,u)=1$ could now fail. This induces 
 technical complications which have been completely ignored in \cite[pp.~239--240]{BFI1}---but, of course, disappear in the typical cases    $a=\pm 1$. \par 
  We address this difficulty by splitting $\gS $ as 
\begin{equation}\label{firstsplit}
\gS = \sum_{\Delta \mid a} \gS_\Delta,
\end{equation}
where $\gS_\Delta$ is defined as  $\gS$ in  \eqref{deffrakS} but with the extra constraint
$
(n,a) =\Delta.
$
Since $(a,qr) =1$, we deduce from \eqref{congequa} that $\Delta \mid u$. So we may write
\begin{equation}
\label{mbv}
n=\Delta m, \quad a =\Delta b,\quad u =\Delta v,
\end{equation}
and deduce from \eqref{congequa} the equality
\begin{equation}\label{928}
m=b +qv r,
\end{equation}
where the coprimality condition   $(m,b)=  1$ is now satisfied.   In order to simplify some summation conditions in the sequel, we will frequently use the trivial  fact that this condition
implies that any divisor of $m-b$ is coprime to $b$.
\par 
The representation \eqref{928} implies that $(b,q)=1$ are coprime, but not necessarily that $(a,q)=1$. So we introduce the integer
\begin{equation}\label{defalpha}
  \alpha=\alpha (a,b): = \prod_{ p\mid a,\, p\,\nmid\,b} p.
\end{equation}
Let $e$ be any divisor of $\alpha$ and write $q=e w$, so that \eqref{928} may be rewritten as
\begin{equation}\label{935}
m=b +ewv r.
\end{equation}
In order to apply M\"obius' formula to take account of the condition $(q, \alpha)=1$, we perform the further split
\begin{equation} \label{splitSigmaDelta}
\gS_{\Delta } := \sum_{ e \mid \alpha} \mu (e) \gS_{\Delta, e},
\end{equation}
with
\begin{align*}
\gS_{\Delta, e}&:=  \sum_{\substack{r\sim R\\  (r,a)=1}}\xi_r \sum_{Q_0/e< w\leqslant Q/e  }  \sum_{\substack{m\leqslant x/\Delta\\  (m,b) =1}} f(\Delta m)\\
&= \sum_{\substack{r\sim R\\  (r,a)=1}}\xi_r \,\sum_{Q_0/e<w\leqslant Q/e  }  \sum_{m\leqslant x/\Delta} f_{\Delta, b}( m)
\end{align*}
where $f_{\Delta, b}$ is defined in \S \ref{conventions} and the variable  $m$ runs through integers  satisfying \eqref{935}
for some  $v$---in other words $m\md b{ewr}$. 
\par 
This is time to apply Dirichlet's device in the form of summing over the smooth variable $v$ instead of $w$. Thus
$$
\gS_{\Delta, e}:=  \sum_{\substack{r\sim R\\  (r,a)=1}}\xi_r  \sum_{v\leqslant (x-a)/(\Delta Q_0 r) } \ \  \sumast_{m\equiv b \bmod evr} f_{\Delta, b}(  m), 
$$
where the asterisk indicates that $m$ satisfies the extra conditions
$$
Q_0 < (m-b)/v r \leqslant  Q,\quad m\leqslant x/\Delta,
$$
which we rephrase as
\begin{equation}\label{ineqforn1}
b +Q_0 vr <m \leqslant  \min \{ x/\Delta, \, b + Qv r\}.
\end{equation}
\par 
   We would like to apply Proposition \ref{firststep} to $\gS_{\Delta, e}$. However the bounds appearing in \eqref{ineqforn1}  are not fixed since they depend on the product $vr$.  This difficulty may be circumvented by appealing to a classical device in such context:  to split the summation on $r$  and $v$ into subsums over intervals of the form $\scI_k:= ]\varrho^k , \varrho^{k+1}]$    with $\varrho$ as in \eqref{defrho}.  This   leads to an estimate of the form
   \begin{equation}\label{decomp10}
\gS_{\Delta, e}= \sum_{k} \sum_{\ell} \gS_{\Delta, e}^{k,\ell} + E,
\end{equation}
 where
 $$ R \leqslant \varrho^k <\varrho^{k+1} < 2R,\quad 1\leqslant \varrho^\ell < \varrho^{\ell +1}< (x-a)/(\Delta Q_0 \varrho^{k+1}),$$ 
 the number of involved pairs $(k,\ell)$ is $\ll \LL^{2B_0+2}$, and 
$$
 \gS_{\Delta, e}^{k,\ell}:=  \sum_{\substack{r\in \scI_k \\  (r,a)=1}}\xi_r  \sum_{v\in \scI_\ell   } \ \  \sumdag_{m\equiv b \bmod evr} f_{\Delta,b}(  m), 
 $$
where the asterisk indicates the summation condition
\begin{equation}\label{condform}
b+Q_0 \varrho^{k+\ell +2} \leqslant m \leqslant M_{k,\ell}:=\min \{ x/\Delta, b +Q\varrho^{k+\ell} \},
\end{equation}
 and the error term $E$ corresponds to the contribution of the triplets $(r, v,m)$  contributing to $\gS_{\Delta,e}$ but 
to none of the $ \gS_{\Delta, e}^{k,\ell}$. 
\par 
Applying the bounds \eqref{G<} and \eqref{G<<} in a classical way yields the estimate
$$
E \ll x/\LL^{A},
$$
provided $B_0$ is chosen sufficiently large.
\par 
  We now consider two cases, according to the size of $M_{k,\ell}$, as defined in \eqref{condform}.\par 
\noindent {\it Case 1: } $M_{k,\ell} \leqslant x/\LL^{A+3B_0+2}$.\par 
 The trivial bound given by 
Lemma \ref{exceptional} then furnishes the bound
\begin{equation}\label{trivial}
\gS_{\Delta, e}^{k,\ell}\ll x/\LL^{A+3B_0-\idc+2}.
\end{equation}
\vskip .3cm
\noindent {\it Case 2:}  $ x/\LL^{A+3B_0+2} < M_{k,\ell} \leqslant x$.
\par 
 We then apply the estimate \eqref{244bismodif*} of Proposition \ref{secondstep},  with the changes of variables $r\to er$,
 $q\to v$, $R\to e \varrho^k$, $Q \to \varrho^\ell$, $b\to \Delta$ and $c\to b$. Since $(a,D)=1$ we also have 
$$
(e vr)_D = v_D r_D\text{ and } (evr)'_D =ev'_Dr'_D.
$$
 That the required hypotheses are satisfied follows from \eqref{Q0<Q<xR-1}. This gives that, for all $C$,  we have
\begin{multline}\label{deep}
\gS_{\Delta, e}^{k, \ell} =
\sum_{\substack{r\in \scI_k \\  (r,a)=1}}\xi_r  \sum_{\substack{v\in \scI_\ell \\  (v,b)=1}} \frac{1}{\varphi (ev'_Dr'_D)} 
\sumdag_{\substack{(m , ev'_D r'_D)=1\\ m\md b{v_Dr_D}}} f_{\Delta,b}( m) +O \Big(\frac{D^{C_0}x}{\LL^C}\Big).
\end{multline}
\par 
Now observe that \eqref{deep} actually also holds in Case 1 above  because the main term is then smaller than the error term---see \eqref{trivial}. 
\par 
Gluing back all estimates \eqref{deep} for $(k,\ell)$ arising  in \eqref{decomp10} and  suitably selecting~$C$, we obtain 
 \begin{multline}\label{almosttheend}
 \gS_{\Delta, e}
 =  \sum_{\substack{r\sim R\\  (r,a)=1}}\xi_r  \sum_{\substack{v\leqslant (x-a)/(\Delta Q_0 r)\\  (v, b) =1 }} \frac{1}{\varphi (ev'_Dr'_D)}  \sumast_{\substack{(m , e v'_Dr'_D)=1\\ m\md b {v_Dr_D}}} f_{\Delta, b}(  m)+O \Big(\frac{D^{C_0}x}{\LL^A}\Big).
 \end{multline}
 \par 
We now insert   \eqref{almosttheend} into \eqref{splitSigmaDelta}, carry back into \eqref{firstsplit}, revert summations,  split the sum according to the level sets  $r_D :=\idr$ and $v_D:=\idv$, and change $r'_D$ into $s$ and $v'_D$ into $v$, to obtain
  \begin{multline}
  \label{secondstepp}
 \gS  =  \sum_{\idr\idv \vert D^\infty}
 \sum_{\substack{s \sim R/\idr\\  (s,aD) =1}} \xi_{\idr s} 
   \sum_{\Delta \mid a} \ 
     \sum_{\substack{m\leqslant x/\Delta\\(m ,  s)=1\\ m\md b {\idr\idv}}} f_{\Delta, b}( m)    \sum_{\substack{e\vert \alpha (a,b) \\  (e,m)=1}} \mu (e)\\
 \times  
 \sum_{(v, bmD) =1 } \frac{1}{\varphi (ev s)} + O \Big(\frac{D^{C_0}x}{\LL^A}\Big), 
 \end{multline}
 where the last summation is restricted to those integers $v$ satisfying 
 \begin{equation}\label{ineqforq1}
\frac{m-b}{Q\idv\idr s} <v \leqslant  \frac{ m-b}{Q_0\, \idv\idr s},
\end{equation}
and where $b$ and $\alpha$ are as defined in \eqref{mbv} and \eqref{defalpha} respectively. 

 \subsection{Transformation of $\gE$} 
In order to compare $\gS$ with $\gE$ defined in \eqref{deffrakS*}, we  consider the approximation of $\gS$ given  in \eqref{secondstepp} and  transform  $\gE$ following a path parallel  to the treatment of  $\gS$ in \S\thinspace\ref{transformofS}. Thus, we fix $\Delta = (a, n)$, 
split $a= \Delta b$, $n= \Delta m$,  $q=\idv v$,  $r=\idr s$ with $\idv =q_D$ and $\idr = r_D$, which implies $(D, vs)=1$.
After inverting  summations, we obtain 
\begin{multline}\label{SEMT}
\gE = \sum_{\idr\idv \vert D^\infty }
\sum_{\substack{s\sim R/\idr\\  (s,aD) =1}} \xi_{\idr s}
 \sum_{\Delta \mid a} \  
\sum_{\substack{m\leqslant x/\Delta \\ (m, s)=1\\ m\md b{\idr\idv}}} f_{\Delta, b}( m) \ \sum_{\substack{Q_0/\idv < v \leqslant  Q/\idv\\   {(v, amD)=1} }} 
\frac{1}{\varphi (vs)}\cdot
\end{multline}
Substracting \eqref{SEMT} from \eqref{secondstepp},  we get  
\begin{equation}\label{S-SEMT}
\bigl\vert \gS-\gE \vert 
\leqslant   \sum_{\idr\idv \vert D^\infty }
\sum_{\substack{s\sim R/\idr\\  (s,aD) =1}} \vert \,\xi_{\idr s}\vert 
 \sum_{\Delta \mid a} \  
\sum_{\substack{m\leqslant x/\Delta \\ (m, s)=1\\ m\md b{\idr\idv}}} \vert f_{\Delta, b}( m)   \Omega_m \vert+O \Big(\frac{D^{C_0}x}{\LL^A}\Big),
\end{equation}
with 
\begin{equation}
\label{defOmega}
\begin{aligned}
\Omega_m &= \Omega_m (a,b,  D, s) \\ &:= \sum_{\substack{e\vert \alpha (a,b) \\  (e,m)=1}} \mu (e)  
 \sum_{(v, bmD) =1 } \frac{1}{\varphi (ev s)} -\sum_{\substack{Q_0/\idv <v \leqslant  Q/\idv\\   {(v, am D)=1} }} 
\frac{1}{\varphi (vs)},
\end{aligned}
\end{equation}
 {where, in the first $v$-sum, the summation variable satisfies \eqref{ineqforq1}}. 
Note that the two $v$-sums appearing in \eqref{defOmega} run over intervals with the same ratio $Q/Q_0$ between the upper and the lower bound,
which is a transcription of Dirichlet's approach. \par 
By a method already used above, we may shorten the summations over $\idv$ and $\idr$ in \eqref{S-SEMT}: for suitable $G= G(A,\varepsilon)$ we have 
\begin{multline}
\label{S-SEMT1}
\bigl\vert \gS-\gE \vert
\leqslant   \sum_{\substack{\idr\idv \mid D^\infty \\  \max(\idr,\idv) \leqslant \LL^G}}
\sum_{\substack{s\sim R/\idr\\  (s,aD) =1}}
 \sum_{\Delta \mid a}  \!
\sum_{\substack{m\leqslant x/\Delta \\ (m, br)=1\\ m\md b{\idr\idv}}} \tau_K (\Delta m) \vert \Omega_m \vert +O \Big(\frac{D^{C_0}\,x}{\LL^A}\Big).
\end{multline}
 
\subsection{Estimating $\Omega_m$} Lemma \ref{sumofphimodif} is relevant to evaluate $\Omega_m$, defined in \eqref{defOmega}. We will apply this statement four times, the parameter $R$ appearing in \eqref{faTRmn} taking successively the values
\begin{equation}
\label{list}
\frac{m-b}{Q\idv\idr s} ,\quad  \frac{ m-b}{Q_0\idv\idr s}, \quad \frac{Q}{\idv}, \quad \frac{Q_0}{\idv}\cdot
\end{equation}
Since $\Delta \leqslant \vert a \vert \leqslant \LL^A$, $QR\leqslant  x/\LL^{B}$, $\max(\idv,\idr)\leqslant \LL^G$, every number of the list \eqref{list}
is at least $$\gL:=\LL ^{B-G-A},$$ where $B(A,\varepsilon)$ has still to be specified. \par 
We apply Lemma  \ref{sumofphimodif} with $R_0:=\gL$ in the four cases
and we notice that parts of the main terms disappear when substracting. We obtain, with notations \eqref{not-hlg} and \eqref{not-Theta01}, 
\begin{equation}
\label{OmhU}
\Omega_m=h\log \Big(\frac Q{Q_0}\Big)\mathfrak U+\gV
\end{equation}
with
\begin{equation*}
\begin{aligned}
\gU&:=\sum_{\substack{e\mid \alpha(a,b)\\ (e,m)=1}}\frac{\mu(e)g(bDems)}{\varphi(es)}\Theta_0(es,bDm;\gL)-\frac{g(aD ms)}{\varphi(s)}\Theta_0(s,aDm;\gL),\\
\gV&\ll \frac{\LL}{\gL^{1/4}}\Bigg\{\sum_{e\mid \alpha(a,b)}\frac{\mu(e)^2\tau(bDems)^2}{\varphi(es)}+\frac{\tau(aD ms)^2}{\varphi(s)}\Bigg\}.
\end{aligned}
\end{equation*}
\par A standard computation involving sums of classical multiplicative functions shows that, provided $B$ is suitably chosen, the contribution of $\gV$ to the multiple sum in \eqref{S-SEMT1} may be absorbed by the error term. 
\par 
We next apply Lemma \ref{rankin} to evaluate the terms involving $\Theta_0$. We get
\begin{equation}
\label{UU*}
\gU=\gU^*+\gW
\end{equation}
with
\begin{equation*}
\begin{aligned}
\gU^*&:=\sum_{\substack{e\mid \alpha(a,b)\\ (e,m)=1}}\frac{\mu(e)g(bDems)}{\varphi(es)}\prod_{\substack{p\,\mid\, es\\ p\,\nmid\,bDm}}\Big(1-\frac1p\Big)^{-1}-\frac{g(aD ms)}{\varphi(s)}\prod_{\substack{p\,\mid\,s\\p\,\nmid\,aDm}}\Big(1-\frac1p\Big)^{-1}\\
\gW&\ll\frac1{\gL^{1/4}}\Bigg\{\sum_{e\mid \alpha(a,b)}\frac{\mu(e)^2\tau(bDems) b_{3/4}(es)}{\varphi(es)}+\frac{\tau(aD ms) b_{3/4}(s)}{\varphi(s)}\Bigg\}.
\end{aligned}
\end{equation*}
Here again, we check that, for suitable choice of $B$, the contribution of $\gW$ to the multiple sum in \eqref{S-SEMT1} may be absorbed by the error term.

\subsection{Vanishing of $\gU^*$} 
We now prove that, for all relevant values of $a,\,b,\,D,\,m,\, s$, we actually have
\begin{equation*}
\gU^*=0.
\end{equation*}
Inserting this back into \eqref{UU*}, \eqref{OmhU} and \eqref{S-SEMT1},  provides the expected bound
$$
\vert\gS -\gE\vert\ll D^{C_0} x / \LL^{A}.
$$
and hence, via \eqref{Sigma=S-S} and  \eqref{Sigma<<}, completes the proof of Theorem \ref{central} by choosing $B=B(A, \varepsilon)$ sufficiently large.
\par 
Observing that the summation conditions in \eqref{S-SEMT} imply $$(b,D)=(bDem,s)=(e,bDm)=1,$$
we see that the condition $p\nmid bDm$ in the first product arising in the definition of $\gU^*$ is superfluous. Therefore we may rewrite $\gU^*$ as
$$\frac{g(bDms)s}{\varphi(s)^2}\Bigg\{\sum_{\substack{e\mid \alpha(a,b)\\ (e,m)=1}}\frac{\mu(e)g(e)e}{\varphi(e)^2}-\prod_{\substack{p\,\mid\,\alpha(a,b)\\ p\,\nmid\,m}}g(p)\Bigg\}.$$
However, the last sum equals
$$\prod_{\substack{p\,\mid\,\alpha(a,b)\\ p\,\nmid\,m}}\Big\{1-\frac{pg(p)}{(p-1)^2}\Big\}=\prod_{\substack{p\,\mid\,\alpha(a,b)\\ p\,\nmid\,m}}g(p),$$
by the definition of $g(p)$. This is all we need.
\section{Proof of Theorem \ref{centralabsolutevalue}.\label{proofoftheoremabs}}
In this section we sketch the proof of Theorem  \ref{centralabsolutevalue} exploiting the combinatorial preparation of the variables 
given in the beginning of the proof of Proposition~\ref{firststep}---see \S \ref{preparation}, \ref{nontypical} for the notations.  To simplify the exposition we only consider the case
$$ D=1.$$ By dyadic dissection
we may restrict to studying
$$
W (Q) :=\sum_{q\leq Q} \max_{(a,q)=1} \vert \Delta_f (x;q,a) -\Delta_f( x/2; q,a)\vert 
$$
and set out to prove that
$$
W(Q) \ll x/\LL^{A},
$$
provided  $Q\leq \sqrt{x}/ \LL^{B} $ with suitable $B=B (A,K)$. 
Using the factorisation \eqref{f(n)=} and bounding trivially the contribution of non typical terms,  we are led to consider the sum
\begin{equation}\label{Wkl=}
W_{k,\ell} (Q) :=\sum_{q\leq Q}\,  \max_{(a,q)=1} \, \biggl\vert \underset{\nu, p_1 <p_2  <\cdots }{\sum}f (\nu) f(p_1) f(p_2)  \cdots g_q (\nu p_1p_2\cdots; a)\biggr\vert 
\end{equation}
where $k \geq 0$ and $\ell \geq 1$, 
where the variables satisfy the conditions
$$
\begin{cases}
\nu\leq Z_0, \quad P^+(\nu)\leqslant  Y_0,\\
p_1 \in \mathfrak I_{k\ell},\  Y_{k+1}< p_2 <p_3 <\cdots ,\\
\nu p_1p_2 \cdots \sim x,
\end{cases}
$$
and aim at establishing the bound
\begin{equation}\label{Wkl<<}
W_{k,\ell} (Q) \ll x/\LL^{A}.
\end{equation}
We now consider two cases :\\
\noindent $\bullet $ If  the variables $p_1$ and $p_2$ do exist on the right--hand side of \eqref{Wkl=}, we directly apply Lemma \ref{Largesieve}
with $N:= Y_k$. This furnishes \eqref{Wkl<<} \\
\noindent $ \bullet$ If   the variable $p_2$  does not actually appear in the multiple sum on the right--hand-side of \eqref{Wkl=} we cannot directly apply 
Lemma \ref{Largesieve} since the support of $\nu$ could be very small. The function $f(p)$  being constant on $\mathfrak I_{k\ell}$, 
we may appeal for instance to \cite[Theorem 8.4]{PanPan} which generalizes the Bombieri--Vinogradov theorem to the function $\bal  * \Lambda$, when the support of the general 
sequence $\bal$  has suitable size. We obtain \eqref{Wkl<<} here again.
\par 
This completes the proof of Theorem \ref{centralabsolutevalue}.

\section{Proof of Theorem \ref{EW}}
\label{pfTh2.1}
\subsection{Lemmas}
The main difficulty for the proof of Theorem \ref{EW} rests in assuming no more than \eqref{CNS-EW}. We need a number of lemmas.
The first is a easy estimate for the number of friable integers in 
\begin{equation}
\label{defExk}
\EE(x;k):=\{n\leqslant x:\omega(n)=k\}.
\end{equation}
It is useful to bear in mind that the Hardy-Ramanujan upper bound \cite{HR17}---see e.g. \cite[ex. 264]{TenlivreUS} or \cite[p. 257]{TW14} for an alternative proof--- states that, for a suitable absolute constant $a$, we have
\begin{equation}
\label{HR}
\pi_k(x)=|\EE(x;k)|\ll \frac{x(\log_2x+a)^{k-1}}{(k-1)!\log x}\qquad (x\geqslant 3,\,k\geqslant 1).
\end{equation}
\begin{lemma}
Uniformly for $x\geqslant 3$, $1\leqslant k\ll\log_2x$, $2\leqslant y\leqslant x$, $u:=(\log x)/\log y$, we have
\begin{equation}
\label{majEkf1}
\pi_k(x,y):=\sum_{\substack{n\in\EE(x;k)\\ P^+(n)\leqslant y}}1\ll\pi_k(x)\e^{-u/2}\cdot
\end{equation}
\end{lemma}
\begin{proof}
The stated estimate holds trivially if $y\leqslant 7$ for  then  $\pi_k(x,y)\ll(\log x)^4$. We may therefore assume henceforth that $y\geqslant 11$. In this circumstance, we may write, with $\alpha:=2/(3\log y)$,
\begin{equation*}
\pi_k(x,y)\leqslant x^{3/4}+x^{-3\alpha/4}\sum_{\substack{x^{3/4}<n\leqslant x\\\omega(n)=k\\ P^+(n)\leqslant y}}n^\alpha
\end{equation*}
Since $\e^{2/3}<2$, we the $n$-sum may be estimated by applying \cite[lemma 1]{GT00} to the  multiplicative function $n\mapsto\1_{\{P^+(n)\leqslant y\}}n^\alpha$. Under the assumption $k\ll\log_2x$, we obtain the upper bound
\begin{equation*}
\ll \pi_k(x)\exp\Big\{\frac{k-1}{\log_2x}\sum_{p\leqslant y}\frac{p^\alpha-1}{p}\Big\}\ll\pi_k(x).
\end{equation*}
This implies the required estimate, up to noticing that $x^{3/4}\ll x^{24/25}\e^{-u/2}$ for $y\geqslant 11$.
\end{proof}
Our next lemma refines the latter when $k$ is `small'.
\begin{lemma}
\label{Ekf2}
Under the conditions
\begin{equation}
\label{hyp}
\varepsilon_x=o(1),\quad k\geqslant 1,\quad k\log (1/\varepsilon_x)=o(\log_2x)\qquad (x\to\infty),
\end{equation}
we have
\begin{equation}
\label{estT}
\pi_k\big(x,x^{1-\varepsilon_x}\big)=o\Big(\pi_k(x)\Big).
\end{equation}
\end{lemma}
\smallskip
\begin{proof}
We may plainly assume $k\geqslant 2$.
Setting $y:=x^{\varepsilon_x}$, we have
\begin{equation}
\begin{aligned}
\pi_k\big(x,x^{1-\varepsilon_x}\big)\leqslant \sum_{\substack{p^\nu\leqslant x\\ p\leqslant x/y}}\pi_{k-1}\Big(\frac x\pnu,p\Big).\cr
\end{aligned}
\end{equation}
A routine Abel summation yields that the contribution of $\nu\geqslant 2$ is
$$\ll\frac{x(\log_2x)^{k-2}}{(k-2)!\log x}\ll\frac{k\pi_k(x)}{\log_2x}=o\Big(\pi_k(x)\Big).$$
The remaining contribution is, for a suitable absolute constant $a$,
\begin{equation}
\begin{aligned}
\label{maj}
&\leqslant \sum_{p\leqslant \sqrt{x}}\pi_{k-1}\Big(\frac xp,p\Big)+\sum_{\sqrt{x}<p\leqslant x/y}\pi_{k-1}\Big(\frac xp\Big)\cr&
\ll\sum_{p\leqslant \sqrt{x}}\frac{x^{1-1/(2\log p)}(\log_2x)^{k-2}}{(k-2)!p\log x}+\sum_{\sqrt{x<}p\leqslant x/y}\frac{x(\log_2x/p+a)^{k-2}}{(k-2)!\,p\log x/p}\cr&\ll\frac{k\pi_k(x)}{\log_2x}+\frac{x}{(k-2)!}\int_{\sqrt{x}}^{x/y}\frac{(\log_2x/t+a)^{k-2}}{t\log (x/t)}\d\pi(t)\cr
&\ll\frac{k\pi_k(x)}{\log_2x}+\frac1{(k-2)!}\int_{y}^{\sqrt{x}}{\pi\Big(\frac xt\Big)}\vbs{f'_k(t)}\d t\cr&\ll\frac{k\pi_k(x)}{\log_2x}+\frac{x}{(k-2)!\log x}\int_{y}^{\sqrt{x}}\frac{\vbs{f'_k(t)}}t\d t,\cr
\end{aligned}
\end{equation}
where we  applied \eqref{HR}, made use of \eqref{majEkf1}, and set
$$f_k(t):=\frac{t(\log_2t+a)^{k-2}}{\log t}\qquad (t\geqslant 3).$$
Since $f_k$ is actually non-decreasing throughout the integration domain, we have
\begin{equation}
\begin{aligned}
\int_{y}^{\sqrt{x}}\frac{\vbs{f'_k(t)}} t\d t\ll\frac{(\log_2x)^{k-2}}{\log x}+\int_y^x\frac{(\log_2t+a)^{k-2}}{t\log t}\d t.\cr
\end{aligned}
\end{equation}
Carrying back into \eqref{maj}, we get
$$\pi_k\big(x,x^{1-\varepsilon_x}\big)\ll\frac{k\pi_k(x)}{\log_2x}+\pi_k(x)\Big\{1-\Big(\frac{\log_2y+a}{\log_2x+a}\Big)^{k-1}\Big\}.$$
Taking account of the inequality 
$1-(1-v)^m\leqslant mv$ $(0\leqslant v\leqslant 1\leqslant m)$, we see that this bound does imply \eqref{estT} under the hypothesis \eqref{hyp}.
\end{proof}
For larger $k$, we shall invoke the following result in which we write 
\begin{equation}
\label{nep}
n_\varepsilon:=\prod_{\substack{p\leqslant x^\varepsilon\\ \pnu\|n}}\pnu\qquad (1\leqslant n\leqslant x,\,0<\varepsilon\leqslant \tfrac12).
\end{equation} 
We also recall notation \eqref{defExk}.
\begin{lemma}
\label{petitsfactsomk}
Uniformly for
$$x\geqslant 3, \quad1/\log_2x<\varepsilon\leqslant \tfrac12, \quad1\leqslant k\ll\log_2x,\quad r:=k/\log_2x,$$
we have
\begin{equation}
\label{propnep}
\omega\big(n_\varepsilon\big)>k-2r\log(1/\varepsilon),\quad n_\varepsilon\leqslant x^{\sqrt{\varepsilon}}
\end{equation}
for all but at most $\ll\varepsilon^{r(\log 4-1)}\pi_k(x)$ integers $n$ in $\EE(x;k)$.
\end{lemma}
\begin{proof}
By  \cite[lemma 1]{GT00}, we have, for any fixed $y>0$,
$$\sum_{n\in\EE(x;k)}y^{\omega(n)-\omega(n_\varepsilon)}\ll\pi_k(x)\varepsilon^{r(1-y)}$$
uniformly in the considered ranges for $x$, $k$, $\varepsilon$. Selecting $y=2$, and multiplying through by $\varepsilon^{r\log 4}$, we see that the number of $n$ contravening the first inequality in \eqref{propnep} is $\ll \varepsilon^{r(\log 4-1)}\pi_k(x)$.\par 
Similarly, for $\alpha:=1/(\varepsilon\log x)$, we have, again by \cite[lemma~1]{GT00},
$$\sum_{n\in\EE(x;k)}n_\varepsilon^\alpha\ll\pi_k(x).$$
This shows that at most $\ll\e^{-1/\sqrt{\varepsilon}}\pi_k(x)$ integers contravene the second condition in \eqref{propnep}.
\end{proof}

Next, we need an estimate for mean-values of some arithmetic functions over level sets related to a shifted argument. 
\begin{definition}
\label{MABeps}
Given constants $A>0$, $B>0$, $\varepsilon>0$, we designate by $\MM(A,B,\varepsilon)$ the class of those functions $G\geqslant 0$ such that
$$G(mn)\leqslant A^{\Omega(m)}G(n)\qquad \big((m,n)=1\big), \quad G(n)\leqslant  Bn^\varepsilon,\quad\sum_{p}\sum_{\nu\geqslant 2}\frac{G(p^\nu)}{p^\nu}\leqslant B.$$
\end{definition}
The following result could be generalized much further along the lines of \cite{NT98,GT18}. 
\begin{lemma}\label{lemmeG}
Let $A>0$, $B>0$, $R>0$, $0<\varepsilon<\tfrac12$.  Uniformly under the conditions $x\geqslant 2$, $G\in\MM(A,B,\varepsilon/3)$,  $1\leqslant k\leqslant R\log_2x$, we have
\begin{equation}
\label{estG}
S_k(x):=\sum_{\substack{1<n\leqslant x\\ \omega(n-1)=k}}G(n)\ll\frac{\pi_k(x)}{\log x}\sum_{n\leqslant x}\frac{G(n)}{n}\cdot
\end{equation}
\end{lemma}
\begin{proof}
In view of, for instance, \cite[cor. 3]{NT98}, the subsum over $n\leqslant x/(\log x)^c$,   with $c=c(R)$ sufficiently large, is negligible in front of the right-hand side of \eqref{estG}. By a standard splitting argument, we may hence restrict to finding an upper bound for the subsum, say $S_k^*(x)$, over the range $x/2<n\leqslant x$.\par 
For each $n$, let $\xi_n$ be the largest of those integers  $\xi$ such that 
$$a_n(\xi):=\prod_{\substack{p^\nu\|n(n-1)\\ p\leqslant \xi}} p^\nu \leqslant x^{2\varepsilon}.$$
Write $a_n:=a_n(\xi_n)$, $b_n:=n(n-1)/a_n$, $p_n:=P^-(b_n)$, $v_n:=v_{p_n}(b_n)$, so that
$$x^{2\varepsilon}/p_n^{v_n}<a_n\leqslant x^{2\varepsilon}.$$
Put  
$a_{jn}:=\prod_{\substack{p^\nu\|n-j,\ p\leqslant \xi_n}}p^\nu$ $(j=0,1)$, so that
$$a_{0n}|n,\ (a_{0n},a_{1n})=1,\ a_{1n}|(n-1).$$
\par For $j\in[1,3]$, we denote by $N_j(x)$ the contribution to $S_k^*(x)$ of those integers $n$ respectively satisfying  the conditions
\begin{equation*}
\begin{aligned}
(N_1)\qquad \qquad& a_n\leqslant x^\varepsilon\hbox{ and } p_n>x^{\varepsilon/2}\hskip30mm & \cr
(N_2)\qquad \qquad& a_n\leqslant x^\varepsilon\hbox{ and } p_n\leqslant x^{\varepsilon/2}& \cr
(N_3)\qquad \qquad&  a_n>x^\varepsilon.&\cr
\end{aligned}
\end{equation*}
\par 
If $n$ appears in $N_1(x)$, then conditions  $\Omega(b_n)\leqslant E:=5/\varepsilon$ and $\omega(n-1)=k$ imply $k-E\leqslant \omega(a_{1n})<k$, where the upper inequality arises from the assumption $a_{1n}\leqslant x^{\varepsilon}$.
Summing according to the fibers $a_{1n}=s,\,a_{0n}=t$, we may hence write
\begin{align*}N_1(x)&\leqslant \sum_{(k-E)^+\leqslant \kappa<k}\sum_{\substack{\omega(s)=\kappa\\ st\leqslant x^\varepsilon}}
\sum_{\substack {x/2<n\leqslant x\\ s|(n-1),\,t|n\\ P^-\big(n(n-1)/st\big)>x^{\varepsilon/2}}}G(n)\cr
&\ll\sum_{(k-E)^+\leqslant \kappa<k}\sum_{\substack{\omega(s)=\kappa\\ st\leqslant x^\varepsilon}} \frac {G(t)x}{\varphi(st)(\log x)^2} \cr&\ll \frac{x}{(\log x)^2}\sum_{(k-E)^+\leqslant \kappa<k}\frac{1}{\kappa!}\bigg(\sum_{p\leqslant x}\frac p{ (p-1)^2}\bigg)^\kappa\sum_{t\leqslant x}\frac{G(t)}{\varphi(t)}\cr
&\ll\frac{x\big(\log_2x\big)^{k-1}}{(\log x)^2(k-1)!}\sum_{n\leqslant x}\frac{G(n)}{n},\cr
\end{align*}
where the last upper bound stems from the following computation, in which we write $\varepsilon_p:=\sum_{\nu\geqslant 2}G(p^\nu)/p^\nu$ and define $\lambda$ as the multiplicative function such that $\lambda(p):=1/(p-1)$, $\lambda(p^\nu)=0$ $(\nu>1)$, for all prime numbers $p$:
\begin{equation}
\label{majsG}
\begin{aligned}
\sum_{t\leqslant x}\frac{G(t)}{\varphi (t)} &=\sum_{t\leqslant x}\frac{G(t)}t\sum_{d|t}\lambda(d)=\sum_{d\leqslant x}\lambda(d)\sum_{m|d^\infty}\frac{G(md)}{md}\sum_{n\leqslant x/md}\frac{G(n)}n\cr
&\leqslant K\sum_{n\leqslant x}\frac{G(n)}{n},\cr
\end{aligned}
\end{equation}
with
\begin{equation}
\label{majK}
\begin{aligned}
K&:=\sum_{d\geqslant 1}\frac{\lambda(d)}{d}\prod_{p\,\mid\,d}\sum_{j\geqslant 0}\frac{G(p^{j+1})}{p^{j}}\leqslant \prod_{p}\bigg\{1+\sum_{j\geqslant 0}\frac{\lambda(p)G(p^{j+1})}{p^{j+1}}\bigg\}\cr&\leqslant \prod_{p}\bigg\{1+\frac{G(p)\lambda(p)}{p}+\lambda(p)\varepsilon_p\bigg\}\leqslant \e^{A+B}.\cr
\end{aligned}
\end{equation}
\par 
\par 
Let us now consider an integer $n$ contributing to $N_2(x)$. We have $p_n^{v_n}>x^\varepsilon$ and $p_n\leqslant x^{\varepsilon/2}$. For each prime $p$ not exceeding $x^{\varepsilon/2}$, let $\nu(p)$ denote the smallest exposant $\nu$ such that $p^\nu>x^\varepsilon$. Then $\nu (p) \geq 2$ and  $p^{\nu(p)-1}\leqslant x^{\varepsilon}$, so $p^{\nu(p)}\leqslant p^{3\varepsilon/2}$. Therefore
\begin{align}
N_2(x)&\leqslant \sum_{p\leqslant x^{\varepsilon/2}}\sum_{\substack{x/2<n\leqslant x\\ n(n-1)\md0{p^{\nu(p)}}}}G(n)\ll Bx^{\varepsilon/3}\sum_{p\leqslant x^{\varepsilon/2}}\frac{x\nu(p)}{p^{\nu(p)}}\ll x^{1-\varepsilon/6}.
\end{align}
\par 
In order to bound $N_3(x)$, we consider $q_n:=P^+(a_n)$ and note that $$\Omega(b_n)\leqslant \eta(q_n):=3(\log x)/\log q_n.$$ It follows that
\begin{align*}
N_3(x)&\leqslant \sum_{q\leqslant x^{2\varepsilon}}\sum_{(k-\eta(q))^+\leqslant \kappa<k}\sum_{\substack{\omega(s)=\kappa\\ x^{\varepsilon}<st\leqslant x^{2\varepsilon}\\ P^+(st)=q}}G(t)\sum_{\substack{x/2<n\leqslant x\\ s|(n-1),\,t|n\\ P^-\big(n(n-1)/st\big)>q}}A^{\eta(q)}\cr
&\ll\sum_{q\leqslant x^{2\varepsilon}}\sum_{(k-\eta(q))^+\leqslant \kappa<k}\sum_{\substack{\omega(s)=\kappa\\ x^{\varepsilon}<st\leqslant x^{2\varepsilon}\\ P^+(st)=q}}\frac{G(t)A^{\eta(q)}x}{\varphi(st)(\log q)^2}\cdot\
\end{align*} 
We estimate the inner sum by Rankin's method, employing the weight $(st/x^\varepsilon)^v$ with $v:=C/\log q$, where $C$ is a sufficiently large constant. We obtain, for a suitable constant $c_0$,
\begin{align*}
N_3(x)&\ll\sum_{q\leqslant x^{2\varepsilon}}\sum_{(k-\eta(q))^+\leqslant \kappa<k}\frac{x^{1-C\varepsilon/\log q}(\log_22q+c_0)^{\kappa}(AR+1)^{\eta(q)}}{q(\log q)^2\kappa!}\sum_{\substack{t\leqslant x\\ P^+(t)\leqslant q}}\frac{G(t)t^v}{\varphi(t)}\cdot
\end{align*}
A computation similar to \eqref{majsG}-\eqref{majK} enables us to show that the last sum is 
$\ll\sum_{n\leqslant x}G(n)/n.$ We then observe that, provided $C\varepsilon>1+2\log(AR+1) $, we have
 \begin{align*}
 &\sum_{q\leqslant x^{2\varepsilon}}\sum_{(k-\eta(q))^+\leqslant \kappa<k}\frac{x^{1-C\varepsilon/\log q}(\log_22q+c_0)^{\kappa}(AR+1)^{\eta(q)}}{q(\log q)^2\kappa!}\cr&\ll\sum_{q\leqslant x^{2\varepsilon}}\frac{\eta(q)x^{1-1/\log q}(\log_2x+c_0)^{k-1}}{q(\log q)^2(k-1)!}\ll\pi_k(x).
 \end{align*}
 \end{proof}   
 Our next, and last, preliminary result is a sieve estimate for integers $n$ such that $\omega(n-1)$ is fixed.
 \begin{lemma}
 \label{crible-n/n-1} Let $R>0$.
 Uniformly for $x\geqslant 3$, $1\leqslant k\leqslant R\log_2x$, and all sets of prime numbers $\P\subset [2,x]$ such that $\sum_{p\in\P}1/p=o(1)$ as $x\to\infty$, we have
 \begin{equation}
\label{estcrib}
\sum_{\substack{1<n\leqslant x\\ \omega(n-1)=k\\ \exists p\in\P\,:\,p|n}}1=o\big(\pi_k(x)\big).
\end{equation}
\end{lemma}
\begin{proof}
We may plainly reduce the proof to showing the stated estimate for the contribution of those $n$ in $]x/2,x]$ to the left-hand side of \eqref{estcrib}. We consider two cases, according to the size of $k$. \par 
Let us first assume $k\leqslant \eta_x\log_2x$ for some function $\eta_x$ tending to 0 sufficiently slowly, to be specified later. We write $n-1=q_nm_n$ with $q_n:=P^+(n-1)$. By Lemma~\ref{Ekf2}, we may assume that $q_n>x^{1-\varepsilon_x}$ for some quantity $\varepsilon_x$ tending to $0$ sufficiently slowly as $x\to\infty$. Thus, we can also discard those integers $n$ such that $q_n^2|n-1$. Now, define $p_n$ as the smallest element of $\P$ such that $p_n|n$. We split the set of remaining integers $n$ into two subsets, according to whether $p_n\leqslant x^{1-2\varepsilon_x}$ or $x^{1-2\varepsilon_x}<p_n\leqslant x$.
\par 
By the Brun-Titchmarsh theorem, the contribution of the first subset does not exceed
\begin{equation}
\label{petitp}
\begin{aligned}
\sum_{\substack{m\leqslant x^{\varepsilon_x}\\ \omega(m)=k-1}}\sum_{\substack{p\in\P\\ p\leqslant x^{1-2\varepsilon_x}\\p\,\nmid\,m}}\sum_{\substack{q\in\PP\\ q\md{-\overline m}p\\ x/3m<q\leqslant x/m}}
1&\ll\sum_{\substack{m\leqslant x^{\varepsilon_x}\\ \omega(m)=k-1}}\sum_{\substack{p\in\P\\ p\leqslant x^{1-2\varepsilon_x}}}\frac x{mp\log (x/mp)}\\
&\ll \frac{x(\log_2x)^{k-1}}{(k-1)!\,\varepsilon_x\log x}\sum_{p\in\P}\frac1p=o\big(\pi_k(x)\big)
\end{aligned}
\end{equation}
provided we select $\varepsilon_x$ tending to 0 sufficiently slowly.
\par 
To bound the contribution of the second subset, we write $$n=q_nm_n+1=\nu_np_n,\mbox{ with } m_n\leqslant x^{\varepsilon_x},\ \nu_n\leqslant x^{2\varepsilon_x}.$$
By the large sieve (see, e.g. \cite[ch. 9]{FI10}), the searched for contribution is hence
\begin{align*}
&\leqslant \sum_{\nu\leqslant x^{2\varepsilon_x}}\!\!\sum_{\substack{m\leqslant x^{\varepsilon_x}\\ (m,\nu)=1\\ \omega(m)=k-1}}\sum_{\substack{q\md{-\overline m}\nu\\ (qm+1)/\nu\in\PP\\ x/3m<q\leqslant x/m}}1\ll \sum_{\nu\leqslant x^{2\varepsilon_x}}\!\!\sum_{\substack{m\leqslant x^{\varepsilon_x}\\ (m,\nu)=1\\ \omega(m)=k-1}}\frac{x}{m\varphi(\nu)(\log x)^2}\ll\varepsilon_x\pi_k(x).
\end{align*}
\par 
We next turn our attention to the case of large $k$, i.e. $\eta_x\log_2x<k\ll\log_2x$. Put $r:=k/\log_2x$. We  select a function $\varepsilon_x$ tending to 0 and write $m:=n-1=ab$, with $a=m_{\varepsilon_x^2}$, as defined in \eqref{nep}.  With a suitable choice of $\varepsilon_x$, Lemma \ref{petitsfactsomk} implies\footnote{This is where we take into account the sufficiently slow decay of $\eta_x$ to 0, so as to ensure that $\varepsilon_x^r=o(1)$.} that, at the cost of neglecting $o\big(\pi_k(x)\big)$ elements $m$ from $\EE(x;k)$, we may  assume that $\omega(b)\leqslant h_x:=4r\log (1/\varepsilon_x)$, $a\leqslant x^{\varepsilon_x}$. Retaining the definition of $p_n$, we consider the same two cases as previously, by comparing $p_n$ to $x^{1-2\varepsilon_x}$.
\par 
If $p_n\leqslant x^{1-2\varepsilon_x}$, the corresponding contribution is hence, parallel to \eqref{petitp},
\begin{equation}
\label{petitp2}
\begin{aligned}
&\leqslant \sum_{\substack{a\leqslant x^{\varepsilon_x}\\ k-h_x\leqslant \omega(a)\leqslant k-1}}\sum_{\substack{p\in\P\\p\,\nmid\,a\\ p\leqslant x^{1-2\varepsilon_x}}}\sum_{\substack{b\md{-\overline a}p\\ P^-(b)>x^{\varepsilon_x^2}\\ x/3a<b\leqslant x/a}}
1\\&\ll\sum_{\substack{a\leqslant x^{\varepsilon_x}\\ k-h_x\leqslant \omega(a)\leqslant k-1}}\sum_{\substack{p\in\P\\ p\leqslant x^{1-2\varepsilon_x}}}\frac {x}{ap\varepsilon_x^2\log x}\\
&\ll\frac{x}{\varepsilon_x^2\log x}\sum_{1\leqslant t\leqslant h_x} \frac{(\log_2x)^{k-t}}{(k-t)!}\sum_{p\in\P}\frac1p,
\end{aligned}
\end{equation}
where the first bound follows from the sieve.  The sum over $t$ does not exceed
$$h_x(1+R)^{h_x}\frac{(\log_2x)^{k-1}}{(k-1)!}\cdot$$
Therefore the last bound in \eqref{petitp2} is $o\big(\pi_k(x)\big)$ provided $\varepsilon_x$ tends to 0 sufficiently slowly.
\par 
If $x^{1-2\varepsilon_x}<p_n\leqslant x$, we have $\nu_n=n/p_n\leqslant x^{2\varepsilon_x}$ so, still writing $m=n-1$, we see that the number $S$ of these integers satisfies
\begin{equation*}
S\leqslant \sum_{\nu\leqslant x^{2\varepsilon_x}}\sum_{\substack{x/3<m\leqslant x\\\omega(m)=k\\ m\md{-1}\nu\\ (m+1)/\nu\in\PP}}1.
\end{equation*}
This quantity may be bounded above using the weights of the combinatorial sieve, as defined, for instance, in \cite[\S 6.5]{FI10}. 
If $\{\lambda_d^+(\nu)\}_{d=1}^{D}$ denotes the sequence of the upper bound sieve for primes $p\leqslant x^c$, $p\nmid \nu$, with a small positive constant $c$ and $D:=x^{cs}$ with some large, fixed $s$, we obtain, with a suitable constant $C$,
\begin{equation*}
\begin{aligned}
S&\leqslant \sum_{\nu\leqslant x^{2\varepsilon_x}}\sum_{d\leqslant D}\lambda_d^+(\nu)\pi_{k}\big(x;\nu d,- 1\big).
\\
&=\sum_{\nu\leqslant x^{2\varepsilon_x}}\sum_{d\leqslant D}\lambda_d^+(\nu)\Big\{\frac{\pi_k(x)}{\varphi(\nu d)}+\Delta_{w_k}\Big(x;\nu d,-1\Big)\Big\}\\
&\leqslant \sum_{\nu\leqslant x^{2\varepsilon_x}}\Bigg\{\frac{C\nu}{\varphi(\nu)^2}\prod_{p\leqslant x^c}\Big(1-\frac1p\Big)\pi_k(x)+\sum_{d\leqslant D}\Big|\Delta_{w_k}\Big(x;\nu d,-1\Big)\Big|\Bigg\},
\end{aligned}
\end{equation*}
where, as defined earlier, $w_k$ is the indicator function of the level set $\EE(x;k)$.
The contribution of the first term inside curly brackets is plainly
$$\ll\varepsilon_x\pi_k(x)=o\big(\pi_k(x)\big).$$
That of the second may be bounded using the Bombieri-Vinogradov theorem for~$w_k$, as established by Wolke and Zhan in \cite{WZ93}. Since we can select $c$ so small that, for instance, $cs\leqslant 1/4$, we obtain the bound
$$\ll\sum_{q\leqslant x^{1/3}}2^{\omega(q)}\max_{(v,q)=1}\Big|\Delta_{w_k}\Big(x;q,v\Big)\Big|.$$
 By the Cauchy-Schwarz inequality, this is 
\begin{align*}
&\ll\Bigg(\sum_{q\leqslant x^{1/3}}\frac{4^{\omega(q)}x}{\varphi(q)}\Bigg)^{1/2}\Bigg(\sum_{q\leqslant x^{1/3}}\max_{(v,q)=1}\Big|\Delta_{w_k}\Big(x;q,v\Big)\Big|\Bigg)^{1/2}\ll\frac x{(\log x)^{A}}
\end{align*}
for any constant $A>0$.\par 
This finishes the proof.
\end{proof}
\subsection{Completion of the proof}
\label{sectpfEW}
We aim at applying Lévy's continuity theorem (see e.g. \cite[th. III.2.4]{TenlivreUS}), according to which the required conclusion  holds if, and only if, under our assumption upon $k$, the  Fourier transforms
\begin{equation}
\label{dft}
\Phi_k(\vartheta;x):=\frac{1}{\pi_k(x)}\sum_{\substack{1<n\leqslant x\\ \omega(n-1)=k}}\e^{i\vartheta f(n)} \qquad (\vartheta\in\R)
\end{equation}
approach $\varphi_{F_r}(\vartheta)$ as defined in \eqref{CarFr} for each $\vartheta$ as $x\to\infty$.   The standard strategy is to employ Cauchy's formula
\begin{equation*}
\Phi_k(\vartheta;x):=\frac{1}{2\pi i\pi_k(x)}\oint_{|z|=r}\WW(x;\vartheta,z;f)\frac{\dd z}{z^{k+1}}\qquad (r>0)
\end{equation*}
with
\begin{equation}
\label{W}
\WW(x;\vartheta,z;f):=\sum_{1<n\leqslant x}\e^{i\vartheta f(n)}z^{\omega(n-1)}.
\end{equation}
To this end, we would like to expand $\e^{i\vartheta f(n)}$ as a sum over the divisors of $n$ and revert summations. However, in order to apply a result like Theorem \ref{central}, we need to restrict the sizes of the  divisors involved. This can be done by selecting a suitable parameter $y\in[1,x]$ and approximating $f(n)$ by the additive truncation
$f_y$, defined~by
\begin{equation}
\label{fy}
f_y(\pnu):=\begin{cases}f(\pnu) & \text{if } p\leqslant y,\\ 0 & \text{if } p>y.\end{cases}
\end{equation}
\par 
We select throughout
\begin{equation}
\label{defy}
y:=\exp\{(\log_2x)^{1/3}\}.
\end{equation}
The first step consists in establishing that, if they exist, the limiting distributions associated to $f$ and $f_y$ are identical, in other words that, given any $\varepsilon>0$, we have
\begin{equation}
\label{appffy}
\sum_{\substack{1<n\leqslant x\\ \omega(n-1)=k\\ |f(n)-f_y(n)|>\varepsilon}}1\ll \eta\pi_k(x)
\end{equation}
for some $\eta=\eta(\varepsilon)$ tending to 0 with $\varepsilon$.\par 
Showing this turns out to be the most difficult part of the proof and motivates all of the preparation displayed in the previous subsection. Since $y\to\infty$, the Hardy-Ramanujan upper bound \eqref{HR} enables us to discard those $n$ such that $p^2|n$ for  some $p>y$. Then, we consider the set $\P_\varepsilon:=\{p\in\PP:p>y,\,|f(p)|>\varepsilon^3\}$. By the convergence of the first two series in \eqref{CNS-EW}, we have
\begin{equation*}
\sum_{p\in\P_\varepsilon}\frac{1}{p}\leqslant \sum_{\substack{p>y\\|f(p)|>1}}\frac1p+ \sum_{\substack{p>y\\|f(p)|\leqslant 1}}\frac{f(p)^2}{\varepsilon^6p}=o(1).
\end{equation*}
Therefore, Lemma \ref{crible-n/n-1} yields that the contribution to \eqref{appffy} of those $n$ divisible by a prime from $\P_\varepsilon$ is negligible. To deal with the remaining integers, we define a multiplicative function $G_\varepsilon$ by
$$G_\varepsilon(\pnu):=\begin{cases}1& \text{if } p\leqslant y,\\ 
0 & \text{if } p>y \text{ and } |f(p)|>\varepsilon^3,\\
0 & \text{if } p>y \text{ and } \nu\geqslant 2,\\
\e^{|f(p)|/\varepsilon^2} & \text{if } p>y,\, \nu=1, \text{ and }|f(p)|\leqslant \varepsilon^3.\end{cases}$$
 Thus, if $n$ is not divisible by the square of a prime $>y$ and is free of prime factors from $\P_\varepsilon$, we have $$|f(n)-f_y(n)|>\varepsilon\Rightarrow G_\varepsilon(n)>\e^{1/\varepsilon}.$$
 Moreover, since $G_\varepsilon$ satisfies the hypotheses of Lemma \ref{lemmeG}, we have
 $$\sum_{\substack{1<n\leqslant x\\ \omega(n-1)=k}}G_\varepsilon(n)\ll\pi_k(x),$$
 which completes the proof of \eqref{appffy}.
 \par 
 We have thus reduced the proof of Theorem \ref{EW} to showing that, for each fixed $\vartheta\in\R$, we have
 \begin{equation}
 \label{conv*}
 \Phi_k^*(\vartheta;x):=\frac{1}{\pi_k(x)}\sum_{\substack{1<n\leqslant x\\ \omega(n-1)=k}}\e^{i\vartheta f_y(n)}=\varphi_{F_r}(\vartheta)+o(1)\qquad (x\to\infty).
 \end{equation}
 With the notation \eqref{W}, we hence write
 \begin{equation}
 \label{repPhi*}
 \Phi_k^*(\vartheta;x)=\frac{1}{2\pi i\pi_k(x)}\oint_{|z|=r}\WW(x;\vartheta,z;f_y)\frac{\dd z}{z^{k+1}}\qquad (r>0).
 \end{equation}
 Given $\vartheta\in\R$, let $h_\vartheta$ denote the multiplicative function defined by $h_\vartheta=\e^{i\vartheta f_y}*\mu$, so that
 $$ h_\vartheta(\pnu)=\begin{cases} \e^{i\vartheta f(\pnu)}-\e^{i\vartheta f(p^{\nu-1})}& \text{if } p\leqslant y,\,\nu\geqslant 1,\\
 0& \text{if } p>y.
 \end{cases}$$
 We have $$\e^{i\vartheta f_y(n)}=\sum_{\substack{d|n\\ P^+(d)\leqslant y}}h_\vartheta(d)$$
 so, for $\alpha:=K/\log y$, $|z|=r$, the contribution to $\WW(x;\vartheta,z;f_y)$ of those $d$ exceeding~$x^c$  is at most
 \begin{equation*}
 \begin{aligned}
 \sum_{1<n\leqslant x}&r^{\omega(n-1)}\sum_{\substack{d|n\\ P^+(d)\leqslant y}}\frac{|h_\vartheta(d)|d^\alpha}{x^{\alpha c}}\\
 &=x^{-c\alpha}\sum_{1<n\leqslant x}r^{\omega(n-1)}\prod_{\substack{\pnu\|n\\ p\leqslant y}}\Big(1+2\sum_{1\leqslant j\leqslant \nu}p^{j\alpha}\Big)\ll_r  x^{1-K/\log_2x},
 \end{aligned}
 \end{equation*}
 by the Cauchy-Schwarz inequality and  standard estimates for sums of non-negative multiplicative functions.
 Since we shall ultimately select $r\leqslant R$, we see that, given any fixed constant $c>0$, we may replace $\WW(x;\vartheta,z;f_y)$ in \eqref{repPhi*} by
 \begin{equation*}
 \WW^*(x;\vartheta,z;f_y):=\sum_{\substack{d\leqslant x^c\\ P^+(d)\leqslant y}}h_\vartheta(d)\sum_{\substack{n\leqslant x-1\\ n\md{-1}d}}z^{\omega(n)},
 \end{equation*}
to within an acceptable error. We are hence in a position to apply Corollary \ref{easy1} with 
\begin{align*}
&Q=D=1, \quad c<\tfrac{1}{105}, \quad R=x^c, \quad \xi_r=\begin{cases}h_\vartheta(r) & \text{ if } P^+(r)\leqslant y,\\ 0& \text{ if } P^+(r)>y,\end{cases}\\
& f(n)=z^{\omega(n)}, \quad a=-1.
\end{align*}
 We obtain
\begin{equation}
\label{estW*}
\WW^*(x;\vartheta,z;f_y)=\sum_{\substack{d\leqslant x^c\\ P^+(d)\leqslant y}}\frac{h_\vartheta(d)}{\varphi(d)}\sum_{\substack{n\leqslant x\\ (n,d)=1}}z^{\omega(n)}+O_A\Big(\frac x{(\log x)^A}\Big).
\end{equation}
Now a standard application of the Selberg-Delange method as displayed in \cite[ch. II.6]{TenlivreUS} yields, uniformly for $x\geqslant 2$, $d\geqslant 1$, $1\leqslant k\leqslant R\log_2x$, $r:=(k-1)/\log_2x\leqslant R$,
\begin{equation}
\label{pikd}
\sum_{\substack{n\leqslant x\\ (n,d)=1\\ \omega(n)=k}}1=\frac{\pi_k(x)}{\idb_r(d)}+O\Big(\frac{B_R(d)\pi_k(x)}{\log_2x}\Big)
\end{equation}
\vskip-1mm
\noindent
 with 
\begin{equation}
\label{brBR}
\idb_r(d):=\prod_{p\,\mid\,d}\Big(1+\frac r{p-1}\Big),\qquad B_R(d):=\prod_{p|d}\Big(1-\frac{1}{p^{3/4}}\Big)^{-R}.\end{equation}
Inserting  \eqref{pikd} back into \eqref{estW*} and carrying into \eqref{repPhi*} yields, by Cauchy's integral formula,
\begin{equation*}
\Phi_k^*(x;\vartheta)=\sum_{\substack{d\leqslant x^c\\ P^+(d)\leqslant y}}\frac{h_\vartheta(d)}{\idb_r(d)\varphi(d)}+o(1)
\end{equation*}
where the error term has been estimated, in view of \eqref{defy}, using the fact that $|h_\vartheta(d)|\leqslant 2^{\omega(d)}$ for all $d\geqslant 1$. By Rankin's method, we may extend the summation over $d$ to all $y$-friable values without altering the remainder term. To reach the required conclusion \eqref{conv*}, it only remains to observe that
$$\sum_{ P^+(d)\leqslant y}\frac{h_\vartheta(d)}{\idb_r(d)\varphi(d)}=\prod_{p\leqslant y}\Bigg(1+\Big(1-\frac1p\Big)\sum_{\nu\geqslant 1}\frac{e^{i\vartheta f(\pnu)}-1}{p^{\nu-1}(p-1+r)}\Bigg)=\varphi_{F_r}(\vartheta)+o(1),$$
due to the convergence of the infinite product.
\goodbreak
\section{Proof of Theorem \ref{thEK}}
Since the proof is very similar to that of Theorem \ref{EW}, we shall only sketch the main lines.
\par 
Let $y$ be defined as in the statement of the theorem and let $f_y$ denote the additive truncation defined by \eqref{fy}. The first step consists in showing that the distribution functions of $\{f(n)-A_x\}/B_x$ and $\{f_y(n)-A_x\}/B_x$ on $$\EE^*_k(x):=\{n\in]1,x]:\omega(n-1)=k\}$$ differ by at most $o(1)$. \par 
To this end, we apply the Hardy-Ramanujan  bound \eqref{HR} to show that at most $o\big(\pi_k(x)\big)$ elements of $\EE^*_k(x)$ are divisible by the square of a prime $>y$, and we invoke 
Lemma \ref{estcrib} to infer that, given $\varepsilon>0$, the same bound holds for the number of those $n\in\EE^*_k(x)$ having a prime factor $p>y$ such that $|f(p)|>\varepsilon^3 B_x$. Then, we may use, as previously, Lemma \ref{lemmeG} with now $G(p):=\e^{|f(p)|/\varepsilon^2B_x}$ to get that $$|f(n)-f_y(n)|=o\big( B_x\big)$$ holds for $\big\{1+o(1)\big\}\pi_k(x)$ elements of $\EE^*_k(x)$.
\par \goodbreak
Let $\vartheta$ be a real number and put $h_\vartheta:=\e^{i \vartheta f_y/B_x}*\mu$. For the second step we aim at showing that, for bounded $z\in\CC$ and any fixed $c>0$, the contribution of $d>x^c$ to the sum
\begin{equation*}
\VV(x;\vartheta,z,y):=\sum_{1<n\leqslant x}z^{\omega(n-1)}\e^{i\vartheta f_y(n)/B_x}=\sum_{1<n\leqslant x}z^{\omega(n-1)}\sum_{d|n}h_\vartheta(d)
\end{equation*} 
is negligible. As in Section \ref{sectpfEW}, this is achieved by a standard application of Rankin's method. This is where we need  $y$ to be taken smaller than any power of $x$. 
\par 
We then apply Corollary \ref{easy1} to obtain, for all fixed $A$,
\begin{equation}
\label{estVzy1}
\VV(x;\vartheta,z,y)=\sum_{\substack{d\leqslant x^c\\ P^+(d)\leqslant y}}\frac{h_\vartheta(d)}{\varphi(d)}\sum_{\substack{n\leqslant x\\(n,d)=1}}z^{\omega(n)}+O\bigg(\frac x{(\log x)^A}\bigg).
\end{equation}
The Selberg-Delange method now furnishes, uniformly for $x\geqslant 2$, $d\geqslant 1$, $|z|\leqslant R$,
\begin{equation*}
\sum_{\substack{n\leqslant x\\(n,d)=1}}z^{\omega(n)}=\frac{J_d(z)x(\log x)^{z-1}}{\Gamma(z)}+O\Big(B_R(d)x(\log x)^{z-2}\Big),
\end{equation*}
where $B_R$ is defined in \eqref{brBR} and
$$J_d(z):=\prod_{p\,\nmid\,d}\Big(1+\frac z{p-1}\Big)\Big(1-\frac1p\Big)^z\prod_{p\,\mid\, d}\Big(1-\frac1p\Big)^z.$$\par 
We then carry back into \eqref{estVzy1} and extend the summation over all $y$-friable $d$. This involves a global error $\ll L_\vartheta(y)x(\log x)^{\Re z-2}$ with
\begin{align*}
L_\vartheta(y)&:=\sum_{P^+(d)\leqslant y}\frac{|h_\vartheta(d)|B_R(d)}{\varphi(d)}\ll\exp\Big\{\sum_{p\leqslant y}\frac{|h_\vartheta(p)|}p\Big\}\\
&\ll\exp\bigg\{\frac{|\vartheta|}{B_x}\sum_{p\leqslant x}\frac{|f(p)|}p\bigg\}\ll(\log x)^{o(1)},
\end{align*} 
where the last bound follows from the Cauchy-Schwarz inequality. By Cauchy's formula with a standard treatment of the error term, we get, for each fixed $\vartheta$ and uniformly for $r:=(k-1)/\log_2x\leqslant R$,
\begin{equation}
\label{EK1}
\sum_{\substack{1<n\leqslant x\\ \omega(n-1)=k}}\e^{i\vartheta f_y(n)/B_x}=\frac x{2\pi i\log x}\oint_{|z|=r}\HH_\vartheta(z;x)(\log x)^z\frac{\dd z}{z^k}+
O\Big(\frac{\pi_k(x)}{\sqrt{\log x}}\Big),
\end{equation}
with
$$\HH_\vartheta(z;x):=\frac1{\Gamma(z+1)}\sum_{P^+(d)\leqslant y}\frac{h_\vartheta(d)J_d(z)}{\varphi(d)}.$$
The last sum is an entire function of $z$. We may hence compute it assuming first that $z\not\in(1-\PP)$ and then deleting this restriction by analytic continuation. We find that it is equal to
$$\prod_{p>y}\Big(1+\frac z{p-1}\Big)\Big(1-\frac1p\Big)^z\prod_{p\leqslant y}\Big(1-\frac1p\Big)^z\bigg(1+\frac {z+h_\vartheta(p)}{p-1}\bigg).$$
It readily follows from our hypotheses that $\HH_\vartheta(z;x)$ is bounded in the disk $|z|\leqslant 2R$. Thus, inserting Taylor's formula
$$\HH_\vartheta(z;x)=\HH_\vartheta(r;x)+(z-r)\HH_\vartheta'(r;x)+O\Big(|z-r|^2\sup_{0\leqslant s\leqslant 1}|\HH''_\vartheta(r+s(z-r);x)|\Big)$$
and noting that, by our choice of $r$, the contribution of the linear term to the Cauchy integral in \eqref{EK1} vanishes, we obtain
\begin{equation*}
\sum_{\substack{1<n\leqslant x\\ \omega(n-1)=k}}\e^{i\vartheta f_y(n)/B_x}=\{1+o(1)\}\frac{x\HH_\vartheta(r;x)(\log_2x)^{k-1}}{(k-1)!\log x}\cdot
\end{equation*}
Applying this also for $\vartheta=0$ and dividing yields
\begin{equation*}
\begin{aligned}
\frac1{\pi_k(x)}&\sum_{\substack{1<n\leqslant x\\ \omega(n-1)=k}}\e^{i\vartheta \{f_y(n)-A_x\}/B_x}=\e^{-i\vartheta A_x/B_x}\prod_{p\leqslant y}\Big(1+\frac{h_\vartheta(p)}{p-1+r}\Big)+o(1)\\
&=\exp\bigg\{\sum_{p\leqslant x}\frac{\e^{i\vartheta f(p)/B_x}-1-i\vartheta f(p)/B_x}p+o(1)\bigg\}\\
&=\exp\bigg\{\int_\R\frac{\e^{i\vartheta u}-1-i\vartheta u}{u^2}\d K_x(u)+o(1)\bigg\},
\end{aligned}
\end{equation*}
with $$K_x(u):=\frac1{B_x^2}\sum_{\substack{p\leqslant x\\ f(p)\leqslant uB_x}}\frac{f(p)^2}p\cdot$$ 
By \eqref{Lind}, $K_x(u)$ tends to $\1_{]0,\infty[}(u)$ as $x\to\infty$. Thus, as may be seen by partial integration and appeal to Lebesgue's dominated convergence theorem, the last integral approaches $-\vartheta^2/2$ as $x\rightarrow \infty$. This is all we need.
\section{Proof of Theorem \ref{thLIL}}
\label{pfLIL}
This is a reappraisal of the proof of \cite[th. 11]{HT88}, in which we make crucial use of Corollary \ref{easy1}. We provide all the details for convenience of the reader. The following result, which is a special case of \cite[cor. 3]{NT98}, will also be very useful. We recall Definition \ref{MABeps} for the class $\MM(A,B,\varepsilon)$.
\begin{lemma}
Let $A>0$, $B>0$, $\varepsilon\in]0,\tfrac1{100}]$. Then, uniformly for $F\in\MM(A,B,\varepsilon)$, $x\geqslant 2$, we have
\begin{equation}
\label{majNT}
\sum_{1<n\leqslant x}F(n)\tau(n-1)\ll x\sum_{n\leqslant x}\frac{F(n)}n\cdot
\end{equation}
\end{lemma}
 Letting $E_x$ denote the expectation relative to the probability $P_x$ brought up in Section \ref{sectapps}, we see that \eqref{majNT} may be restated~as
\begin{equation}
\label{majNT-exp}
E_x(F)\ll\frac1{\log x}\sum_{n\leqslant x}\frac{F(n)}n\cdot
\end{equation}
\par 
Let us start by proving \eqref{LIL-maj} or, equivalently, with notation \eqref{defsLIL},
\begin{equation}
\label{LIL-maj-Px}
P_x\big(|M(n,\xi)|>1+\varepsilon\big)=o(1)\qquad (x\to\infty).
\end{equation}
\par 
Provided $\xi(x)$ tends to infinity sufficiently slowly, which can be assumed without loss of generality, we may restrict our attention to the range $\xi(x)<t\leqslant x_1$ with $\log_2x_1=\log_2x-\xi(x)$. Indeed, if $t>x_1$, we have, with notation \eqref{omnt}, 
\begin{align*}
\omega(n,t)-\log_2t&\leqslant \omega(n)-\omega(n,x_1)+|\omega(n)-\log_2x|+\xi(x)\\
\omega(n,t)-\log_2t&\geqslant \omega(n)-\log_2x-\big\{\omega(n)-\omega(n,x_1)\big\}\end{align*}
and so
$$|\omega(n,t)-\log_2t|\leqslant \omega(n)-\omega(n,x_1)+|\omega(n)-\log_2x|+\xi(x).$$
Now, by  \eqref{majNT-exp}, we have for any $v\in[1,2]$
\begin{align*}
&E_x\big(v^{\omega(n)-\omega(n,x_1)}\big)\ll \e^{(v-1)\xi(x)},\\
&E_x\big(v^{\omega(n)-\log_2x}+v^{\log_2x-\omega(n)}\big)\ll (\log x)^{v-1-\log v}+(\log x)^{\log v-1+1/v}
\end{align*} 
Selecting $v=2$ in the upper estimate and $v=1+\xi(x)/\sqrt{\log_2x}$ in the lower one, we see that, assuming $\xi(x)=o\big(\sqrt{\log_4x}\big)$,  we have, for all $\varepsilon>0$,
$$P_x\Big(\sup_{x_1<t\leqslant x}|\Lambda(n,t)|>\varepsilon\Big)=o(1)$$
and so \eqref{LIL-maj-Px} will follow from
\begin{equation*}
P_x\big(|M_1(n,\xi)|>1+\varepsilon\big)=o(1), \text{ with } M_1(n,\xi):=\sup_{\xi(x)<t\leqslant x_1}|\Lambda(n,t)|.
\end{equation*}
\par 
Next, we consider the subset $U_x$ of $\{n:1<n\leqslant x\}$ comprising those integers $n$ such that $\prod_{\pnu\|n,\,p\leqslant x_1}\pnu\leqslant x^{1/4}$. By \eqref{majNT-exp}, we have
$$P_x(]1,x]\smallsetminus U_x)\leqslant x^{-1/(4\log x_1)}E_x\Big(\prod_{\pnu\|n,\,p\leqslant x_1}p^{\nu/\log x_1}\Big)\ll\exp\big\{-\tfrac14\e^{\xi(x)}\big\}.$$
\par 
We thus embark to prove that
\begin{equation}
\label{LIL-maj2-Px}
P_x\Big(n\in U_x,\,|M_1(n,\xi)|>1+\varepsilon\Big)=o(1).
\end{equation}
At the cost of replacing $\varepsilon$ by $2\varepsilon$ in \eqref{LIL-maj2-Px}, we may plainly restrict the variable $t$, in the definition of $M_1(n,\xi)$, to run through the sequence $$\{t_k:\xi(x)\leqslant k\leqslant \log_2x-\xi(x)\},$$ with $t_k:=\exp\exp k$ $(k\geqslant 1)$. Now, put $I:=\fl{(\log_3\xi(x))/\varepsilon}$, 
 $J:=\fl{(1+\log_3x_1)/\varepsilon }$,    and for each $j$, $I\leqslant j\leqslant J$, write $K_j:=\e^{\varepsilon j}$, so that $K_I\leqslant \log_2\xi(x)$, $K_J>\log_2x_1$. We then define $T_j:=\exp\exp K_j=\exp\exp\exp(\varepsilon j)$ and consider the set $S_j$ of those $n\in U_x$ such that 
$$\sup_{K_j\leqslant k\leqslant K_{j+1}}|\Lambda(n,t)|>1+\varepsilon.\leqno{(S_j)}$$
We also define $\psi(T):=(1+\varepsilon)\sqrt{\log_2T}-c$, where $c$ is a sufficienlty large constant, and denote by $A_j$ the set of those $n\in U_x$ for which
$$|\omega(n,T_{j+1})-\log T_{j+1}|>\psi(T_j)\sqrt{\log_2T_j}.\leqno{(A_j)}$$
In order to show that 
\begin{equation*}
P_x\Big(\bigcup_{I\leqslant j\leqslant J}S_j\Big)=o(1),
\end{equation*}
we shall actually prove, for $x>x_0(\varepsilon)$,
\begin{align}
&P_x\big(A_j\big)\ll\frac1{j^{1+\varepsilon/4}}\quad(I\leqslant j\leqslant J),&\label{PxAj}\\
&P_x\big(S_j\big)\leqslant 2P_x\big(A_j\big)\quad(I\leqslant j\leqslant J).&\label{PxSj}
\end{align}
\par 
The proof of \eqref{PxAj} is straightforward, since, for $x>x_0(\varepsilon)$, condition $(A_j)$ implies 
\begin{align*}
|\omega(n,T_{j+1})-\log_2T_{j+1}|&>(1+\tfrac23\varepsilon)\sqrt{2\log_2T_j\log_4T_j}\\
&>(1+\eta_j)\log_2T_{j+1},
\end{align*}
with $\eta_j:=(1+\tfrac17\varepsilon)\sqrt{2\log_4T_{j+1}/\log_4T_{j+1}}$ and so we only need to apply \eqref{majNT-exp} with $F(n):=v^{\omega(n,T_{j+1})}$ for $v=1\pm\eta_j$.
\par 
To prove \eqref{PxSj}, we split $S_j$ into $K_{j+1}-K_j$ disjoint subsets $S_{kj}$, $K_j<k\leqslant K_{j+1}$, defined by the extra condition
$$\max_{k_j<m<k}|\Lambda(n,t_m)|\leqslant 1+\varepsilon<|\Lambda(n,t_k)|.\leqno{(S_{kj})}$$
We clearly have
\begin{equation*}
P_x(S_j)\leqslant \sum_{K_j<k\leqslant K_{j+1}}P_x(S_{kj}).
\end{equation*}
For each $k\in]K_j,K_{j+1}]$, let $B_{kj}$ comprise those integers $n\in U_x$ such that
$$\bigg|\omega\big(n,T_{j+1}\big)-\omega(n,t_k)-\log \Big(\frac{\log T_{j+1}}{\log t_k}\Big)\bigg|\leqslant c\sqrt{\log_2T_{j+1}}.\leqno{(B_{kj})}$$
Since condition $(S_{kj})$ implies 
$$\left|\omega(n,t_k)-\log_2t_k\right|>(1+\varepsilon)\sqrt{2\log_2t_k\log_4t_k}$$
we see that $B_{kj}\cap S_{kj}\subset A_j$.  The $S_{kj}$ being disjoint for fixed $j$, we infer that
\begin{equation*}
\sum_{K_j<k\leqslant K_{j+1}}P_x\big(B_{kj}\cap S_{kj}\big)\leqslant P_x\big(A_j\big).
\end{equation*}
Therefore \eqref{PxSj} will follow from 
\begin{equation}
\label{SovB}
P_x\big(S_{kj}\cap\ov B_{kj}\big)\leqslant \tfrac12 P_x\big(S_{kj}\big)
\end{equation}
with $\ov B_{kj}:=U_x\smallsetminus B_{kj}$. We shall prove that \eqref{SovB} holds for sufficiently large $c$.
\par 
Let $a$, $b$ denote respectively generic integers such that $P^+(a)\leqslant t_k$, $P^-(b)>t_k$. We have
\begin{equation*}
\sum_{n\in S_{kj}}\tau(n-1)=\sum_{\substack{a\leqslant x^{1/4}\\ a\in S_{kj}}}\sum_{\substack{b\leqslant x/a\\ ab>1}}\tau(ab-1)\geqslant \sum_{\substack{a\leqslant x^{1/4}\\ a\in S_{kj}}}\sum_{d\leqslant \sqrt{x}}\sum_{\substack{b\leqslant x/a\\ ab>1\\ b\md{\ov a}d}}1.
\end{equation*}
Corollary \ref{easy1} enables us to bound the double inner sum from below. We obtain, for any $A>0$,
\begin{equation*}
\sum_{n\in S_{kj}}\tau(n-1)\geqslant  \sum_{\substack{a\leqslant x^{1/4}\\ a\in S_{kj}}}\frac xa\Big\{\sum_{d\leqslant \sqrt{x}}\frac{1}{\varphi(d)\log t_k}+O\Big(\frac{1}{(\log x)^A}\Big)\Big\},
\end{equation*}
whence
\begin{equation}
\label{PxSkj}
P_x\big(S_{kj}\big)\gg \e^{-k}\sum_{\substack{a\leqslant x^{1/4}\\ a\in S_{kj}}}\frac1a\cdot
\end{equation}
Now observe that, for $n=ab\in S_{kj}\cap\ov B_{kj}$, at least one of the inequalities $\alpha_1(b)<0$ or $\alpha_2(b)>0$ holds, with
\begin{align*}
\alpha_h(b)=\omega(b,T_{j+1})-\big(K_{j+1}-k)+(-1)^hc\sqrt{K_j}\qquad (h=1,2).
\end{align*}
Therefore, for any $y_1$, $y_2$, with $\tfrac12\leqslant y_1<1<y_2\leqslant \tfrac32$, we have
\begin{equation}
\label{PxSkjovBkj}
P_x\big(S_{kj}\cap\ov B_{kj}\big)\ll\frac{1}{x\log x}\sum_{\substack{a\leqslant x^{1/4}\\ a\in S_{kj}}}\sum_{b\leqslant x/a}\tau(ab-1)\big\{y_1^{\alpha_1(b)}+y_2^{\alpha_2(b)}\big\}.
\end{equation}
Inserting the upper bound $$\tau(ab-1)\leqslant 2\sum_{\substack{d\leqslant \sqrt{x}\\ ab\md1d}}1$$ and appealing to Corollary \ref{easy1} again furnishes
\begin{align*}
\sum_{b\leqslant x/a}\tau(ab-1)y_h^{\alpha_h(b)}&\leqslant 2\sum_{d\leqslant \sqrt{x}}\sum_{\substack{b\leqslant x/a\\ ab>1\\ b\md{\ov a}d}}y_h^{\alpha_h(b)}\\
&\ll \sum_{d\leqslant \sqrt{x}}\frac{1}{\varphi(d)}\sum_{\substack{b\leqslant x/a\\ (b,d)=1}}y_h^{\alpha_h(b)}+\frac x{a(\log x)^A}\\
&\ll\frac{x\log x}{a\,\e^k}\exp\Big\{(K_{j+1}-k)(y_h-1-\log y_h)-c|\log y_h|\sqrt{K_j}\Big\}.
\end{align*} 
We select $y_h:=1+(-1)^h/2$ or $y_h=1+(-1)^hc\sqrt{K_j}/(K_{j+1}-k)$ according as $k>K_{j+1}-2c\sqrt{K_j}$ or not. Then the expression in curly brackets is $\leqslant -\kappa c$, where $\kappa$ is an absolute constant. Inserting back into \eqref{PxSkjovBkj}, we obtain
$$P_x\big(S_{kj}\cap\ov B_{kj}\big)\ll\e^{-k-\kappa c}\sum_{\substack{a\leqslant x^{1/4}\\ a\in S_{kj}}}\frac1a,$$
from which \eqref{SovB} follows for sufficiently large $c$, in view of \eqref{PxSkj}.
\par 
This completes the proof of \eqref{LIL-maj}.
\par \smallskip
We now turn our attention to proving \eqref{LIL+} and \eqref{LIL-}. By symmetry, we restrict to the  first property. We write $X:=x^{1/\log_2x}$ and embark on showing 
\begin{equation}
\label{minLIL}
P_x\big(M^+(n,\xi)\leqslant 1-\varepsilon\big)=o(1)\qquad (x\to\infty).
\end{equation}
Given a large constant $D=D(\varepsilon)$ to be specified later, we put $K:=\fl{\log_2\xi(x)/\log D}$, $L:=\fl{\log_3X/\log D}$. For $K<k\leqslant L$,  set $s_k:=\exp\exp D^k$, $I_k:=]s_{k-1},s_k]$, and define
$$\omega_k(n):=\sum_{p|n,\, p\in I_k}1.$$
As in the corresponding part of the proof of \cite[th. 11]{HT88}, we aim at establishing that the level of independence of the $\omega_k(n)$ $(K<k\leqslant L)$ is sufficient to implement the classical probabilistic approach. Of course, independence is here understood  with respect to~$P_x$.
\par 
 Let $z_k$ $(K<k\leqslant L)$ be complex numbers such that $|z_k|\leqslant 2$, and write $$\bz:=(z_{K+1},\ldots, z_L).$$ The first step consists in evaluating the characteristic function
 $$\Psi(\bz):=E_x(f), \text{ where } f(n):=\prod_{K<k\leqslant L}z_k^{\omega_k(n)}.$$ 
 We shall show that
 \begin{equation}
 \label{evalPsibz}
 \Psi(\bz)=\Theta(\bz)\bigg\{1+O\Big(\frac1{\log x}\Big)\bigg\},
  \end{equation}
 with $$\Theta(\bz):=\prod_{p}\bigg\{1+\frac{(f(p)-1)(p-1)}{p^2}\bigg\}=\prod_{K<k\leqslant L}\prod_{p\in I_k}\bigg\{1+\frac{(z_k-1)(p-1)}{p^2}\bigg\}.$$\par 
 For any $\beta>\tfrac12$, we have
 \begin{equation}
 \label{NumPxEx}
\begin{aligned}
 S(x;\bz):=\sum_{1<n\leqslant x}f(n) \tau(n-1)&=2\sum_{d\leqslant \sqrt{x}}\sum_{\substack{d^2<n\leqslant x\\ n\md1d}}f(n)+O\big(x^\beta\big)\\
 &=2\sum_{d\leqslant \sqrt{x}}\sum_{\substack{n\leqslant x\\ n\md1d}}f(n)+O\big(\gR+x^\beta\big),\end{aligned}
  \end{equation}
with 
$$\gR:=\sum_{x^{1/4}<d\leqslant \sqrt{x}}\sum_{\substack{n\leqslant d^2\\ n\md1d}}f(n).$$
 Let $\Delta:=1+1/\LL^B$, where $B$ is a large constant to be determined later. We split the outer summation range into intervals $V_j:=\big\{d:x^{1/4}\Delta^j<d\leqslant x^{1/4}\Delta^{j+1}\big\}$. In each corresponding subsum we may replace the condition $n\leqslant d^2$ by $n\leqslant \sqrt{x}\Delta^{2j}$ at the cost of an error 
 $$\ll \frac{x^{1/4}\Delta^{j}(\Delta-1)}{\log x}\sum_{n\leqslant x}\frac{|f(n)|}n\ll\frac{x^{1/4}\Delta^j}{\LL^{B-1}}.$$ 
 Indeed, this readily follows from \cite[cor. 3]{NT98}. Summing over $j$, we obtain that the global error involved is  $\ll \sqrt{x}\LL$. Next, we apply Corollary \ref{easy1} to each subsum, viz.
 \begin{equation*}
 \sum_{d\in V_j}\sum_{\substack{n\leqslant \sqrt{x}\Delta^{2j}\\ n\md1d}}f(n)=\sum_{d\in V_j}\frac{1}{\varphi(d)}\sum_{\substack{n\leqslant \sqrt{x}\Delta^{2j}\\ (n,d)=1}}f(n)+O\Big(\frac{\sqrt{x}\Delta^{2j}}{\LL^{2B}}\Big).
 \end{equation*}
 The inner sum is relevant to \cite[th. 02]{HT88}. It is
 $$\sqrt{x}\Delta^{2j}\prod_{p\,\nmid\, d}\Big(1+\frac{f(p)-1}p\Big)+O\Big(\frac{\sqrt{x}\Delta^{2j}}{\log x}\Big).$$
Carrying back into \eqref{NumPxEx}, we get, after a short computation,
$$S(x;\bz)=2\sum_{d\leqslant \sqrt{x}}\sum_{\substack{n\leqslant x\\ n\md1d}}f(n)+O\Big(x \Theta(\bz)\Big),$$
which is compatible with \eqref{evalPsibz}. Applying Corollary \ref{easy1} again, we get
\begin{equation}
\label{Eftau1}
S(x;\bz)=2\sum_{d\leqslant \sqrt{x}}\frac1{\varphi(d)}\sum_{\substack{n\leqslant x\\ (n,d)=1}}f(n)+O\Big(x \Theta(\bz)\Big),
\end{equation}
 since the error involved from the Bombieri-Vinogradov estimate may be absorbed by the previous one. Now, letting $\chi_d$ denote the indicator of the set of integers coprime to $d$, we write $f\chi_d=g*\chi_d$ where $g$ is the multiplicative function defined as
 $$ g(\pnu)=\begin{cases}0& \text{if } p\mid d \text{ or } \nu>1\\
 f(p)-1& \text{if } p\nmid d,\,\nu=1.
 \end{cases}
 $$
 We have 
 \begin{align*}
 \sum_{\substack{n\leqslant x\\ n\md1d}}f(n)&=\sum_{m\leqslant x}g(m)\sum_{\substack{n\leqslant x/m\\ (n,d)=1}}1=\frac{x\varphi(d)}d\sum_{m\leqslant x}\frac{g(m)}m+O\Big(2^{\omega(d)}\sum_{m\leqslant x}|g(m)|\Big)\\
 &=\frac{x\varphi(d)}d\prod_{p\,\nmid\, d}\Big(1+\frac{f(p)-1}p\Big)+O\Big(\frac{x2^{\omega(d)}}{(\log x)^B}\Big),
 \end{align*}
 where the last estimate may be obtained by a standard application of Rankin's method, using the fact that $g(m)$ vanishes if $P^+(m)>X$. Carrying back into \eqref{Eftau1}, we get, keeping in mind that $f(p)=1$ for all small $p$,
 \begin{equation*}
 \begin{aligned}
 S(x;\bz)&=2\sum_{d\leqslant \sqrt{x}}\frac1d\prod_{p\,\nmid\, d}\Big(1+\frac{f(p)-1}p\Big)+O\big(x \Theta(\bz)\big)\\
 &=2\prod_{p}\Big(1+\frac{f(p)-1}p\Big)\sum_{d\leqslant \sqrt{x}}\frac{1}{dr(d)}+O\big(x \Theta(\bz)\big),
 \end{aligned}
 \end{equation*}
with $r(d):=\prod_{p|d}\{1+(f(p)-1)/p\}$. The Selberg-Delange method yields \eqref{evalPsibz}. We omit the details.
\par 
From this point on, the argument is essentially identical with the corresponding part of \cite[th. 11]{HT88}, and  we only sketch the main steps. We use $$\exp\Big\{\frac{z-1}{p-1}+\frac{2}{p(p-1)}\Big\}$$ as a majorant series for $1+z(p-1)/p^2$ and note that
$$\exp\Big\{2+\sum_{K<k\leqslant L}(z-1)H_k\Big\}$$ 
is a majorant series for $\Theta(\bz)$ with
$$H_k:=\sum_{p\in I_k}\frac1{p-1}=D^{k-1}(D-1)+O(1)\qquad (K<k\leqslant L).$$
\par 
For $j_k\leqslant 2H_k$, we evaluate 
$P_x\big(\omega_k(n)=j_k\,(K<k\leqslant L)\big)$
by Cauchy's integral formula and then show that, with $$ h_k:=H_k+\sqrt{2H_k\log_2D^k}\qquad (K<k\leqslant L),$$ we have
 $$P_x\Big(\sup_{K<k\leqslant L}\{\omega_k(n)-h_k\}>0\Big)=1+o(1).$$ 
 We then conclude by noting that, if $\omega_\ell(n)-h_\ell>0$ and if $M_1(n,\xi)\leqslant 1+\varepsilon$, then
 $$\omega(n,s_{\ell-1})\geqslant D^{\ell-1}-(1+\varepsilon)\sqrt{2D^{\ell-1}\log_2 D^\ell}$$
 and so, if $D=D(\varepsilon)$ is sufficiently large,
  \begin{align*}
\omega(n,s_\ell)&=\omega(n,s_{\ell-1})+\omega_\ell(n)\\
&\geqslant D^{\ell-1}+H_\ell+\Big(\sqrt{H_\ell}-(1+\varepsilon)\sqrt{D^{\ell-1}}\Big)\sqrt{2\log_2D^{\ell}}\\
&\geqslant D^{\ell}+(1-\varepsilon)\sqrt{2D^{\ell}\log_2D^\ell}.
 \end{align*}
 This completes the proof of \eqref{LIL+} alias \eqref{minLIL}.
\goodbreak
\vskip10mm
\bibliographystyle{abbrv}
\bibliography{Fouvry-biblio} 

\begin{thebibliography}{10}

\bibitem{BVL65}
M.~B. Barban, A.~I. Vinogradov, and B.~V. Levin.
\newblock Limit laws for functions of the class {$H$} of {I}. {P}. {K}ubilius
  which are defined on a set of ``shifted'' primes.
\newblock {\em Litovsk. Mat. Sb.}, 5:5--8, 1965.

\bibitem{BFI1}
E.~Bombieri, J.~Friedlander, and H.~Iwaniec.
\newblock Primes in arithmetic progressions to large moduli.
\newblock {\em Acta Math.}, 156(3--4):203--251, 1986.

\bibitem{BFI2}
E.~Bombieri, J.~Friedlander, and H.~Iwaniec.
\newblock Primes in arithmetic progressions to large moduli, {II}.
\newblock {\em Math. Ann.}, 277(3):361--393, 1987.

\bibitem{BFI4}
E.~Bombieri, J.~Friedlander, and H.~Iwaniec.
\newblock Some corrections to an old paper.
\newblock {\em arXiv:1903.01371}, pages 1--3, 2019.

\bibitem{DrTo}
S.~Drappeau and B.~Topacogullari.
\newblock Combinatorial identities and {T}itchmarsh's divisor problem for
  multiplicative functions.
\newblock {\em Algebra \& Number Theory}, 13(10):2383--2425, 2019.

\bibitem{Elliott79}
P.~D. T.~A. Elliott.
\newblock {\em Probabilistic number theory. {I}}, volume 239 of {\em
  Grundlehren der Mathematischen Wissenschaften [Fundamental Principles of
  Mathematical Science]}.
\newblock Springer-Verlag, New York-Berlin, 1979.
\newblock Mean-value theorems.

\bibitem{FouCrelle}
{\'E}.~Fouvry.
\newblock Sur le probl\`eme des diviseurs de {T}itchmarsh.
\newblock {\em J. Reine Angew. Math.}, 357:51--76, 1985.

\bibitem{FoKoMid3}
{\'E}.~Fouvry, E.~Kowalski, and {\relax Ph}.~Michel.
\newblock On the exponent of distribution of the ternary divisor function.
\newblock {\em Mathematika}, 61(1):121--144, 2015.

\bibitem{FouRadzi1}
{\'E}.~Fouvry and M.~Radziwi\l\l.
\newblock Level of distribution of unbalanced convolutions.
\newblock {\em Ann. Sci. \'{E}c. Norm. Sup\'{e}r. (4)}, to appear.

\bibitem{FrIwd3}
J.~Friedlander and H.~Iwaniec.
\newblock Incomplete {K}loosterman sums and a divisor problem. {W}ith an
  appendix by {B}ryan {J.} {B}irch and {E}nrico {B}ombieri.
\newblock {\em Ann. of Math. (2)}, 121(2):319--350, 1985.

\bibitem{FI10}
J.~Friedlander and H.~Iwaniec.
\newblock {\em Opera de cribro}, volume~57 of {\em American Mathematical
  Society Colloquium Publications}.
\newblock American Mathematical Society, Providence, RI, 2010.

\bibitem{Go20}
{\'E}.~Goudout.
\newblock Th{\'e}or\`eme d'{E}rd{\H o}s-{K}ac pour les translat{\'e}s d'entiers
  ayant {$k$} facteurs premiers.
\newblock {\em J. Number Theory}, 210:280--291, 2020.

\bibitem{GS19}
A.~Granville and X.~Shao.
\newblock Bombieri-{V}inogradov for multiplicative functions, and beyond the
  {$x^{1/2}$}-barrier.
\newblock {\em Adv. Math.}, 350:304--358, 2019.

\bibitem{HT88}
R.~R. Hall and G.~Tenenbaum.
\newblock {\em Divisors}, volume~90 of {\em Cambridge Tracts in Mathematics}.
\newblock Cambridge University Press, Cambridge, 1988.

\bibitem{HR17}
G.~H. Hardy and S.~Ramanujan.
\newblock The normal number of prime factors of a number {$n$} [{Q}uart. {J}.
  {M}ath. {\bf 48} (1917), 76--92].
\newblock In {\em Collected papers of {S}rinivasa {R}amanujan}, pages 262--275.
  AMS Chelsea Publ., Providence, RI, 2000.

\bibitem{H-Bd3}
D.~Heath-Brown.
\newblock The divisor function $d_3 (n)$ in arithmetic progressions.
\newblock {\em Acta Arith.}, 47(1):29--56, 1986.

\bibitem{IwKo}
H.~Iwaniec and E.~Kowalski.
\newblock {\em Analytic number theory}, volume~53 of {\em American Mathematical
  Society Colloquium Publications}.
\newblock American Mathematical Society, Providence, RI, 2004.

\bibitem{Ka68}
I.~K\'{a}tai.
\newblock On distribution of arithmetical functions on the set prime plus one.
\newblock {\em Compositio Math.}, 19:278--289, 1968.

\bibitem{Kub64}
J.~Kubilius.
\newblock {\em Probabilistic methods in the theory of numbers}.
\newblock Translations of Mathematical Monographs, Vol. 11. American
  Mathematical Society, Providence, R.I., 1964.

\bibitem{NT98}
M.~Nair and G.~Tenenbaum.
\newblock Short sums of certain arithmetic functions.
\newblock {\em Acta Math.}, 180(1):119--144, 1998.

\bibitem{PanPan}
C.~Pan and C.~Pan.
\newblock {\em Goldbach conjecture}.
\newblock Science {P}ress {B}eijing, Beijing, China, 1992.

\bibitem{Sh}
P.~Shiu.
\newblock A {B}run-{T}itchmarsh theorem for multiplicative functions.
\newblock {\em J. Reine Angew. Math.}, 313:161--170, 1980.

\bibitem{GT00}
G.~Tenenbaum.
\newblock A rate estimate in {B}illingsley's theorem for the size distribution
  of large prime factors.
\newblock {\em Q. J. Math.}, 51(3):385--403, 2000.

\bibitem{TenlivreUS}
G.~Tenenbaum.
\newblock {\em Introduction to analytic and probabilistic number theory},
  volume 163 of {\em Graduate Studies in Mathematics}.
\newblock American Mathematical Society, Providence, RI, third edition, 2015.
\newblock Translated from the 2008 French edition by Patrick D. F. Ion.

\bibitem{GT18}
G.~Tenenbaum.
\newblock Note sur les lois locales conjointes de la fonction nombre de
  facteurs premiers.
\newblock {\em J. Number Theory}, 188:88--95, 2018.

\bibitem{TW14}
G.~Tenenbaum and {J. Wu (coll.)}.
\newblock {\em {T}h\'eorie {A}nalytique et {P}robabiliste des {N}ombres, 307
  exercices corrigés}.
\newblock \'Echelles. Belin, 2014.

\bibitem{TV20}
G.~Tenenbaum and J.~Verwee.
\newblock Effective {E}rd{\H o}s-{W}intner theorems.
\newblock {\em Proc. Steklov Inst. Math.}, 314, 2021.
\newblock To appear.

\bibitem{Wolke73}
D.~Wolke.
\newblock \" {U}ber die mittlere {V}erteilung der {W}erte zahlentheoretischer
  {F}unktionen auf {R}estklassen. {I}.
\newblock {\em Math. Ann.}, 202:1--25, 1973.

\bibitem{WZ93}
D.~Wolke and T.~Zhan.
\newblock On the distribution of integers with a fixed number of prime factors.
\newblock {\em Math. Z.}, 213(1):133--144, 1993.

\end{thebibliography}
\end{document}